\documentclass[11pt]{amsart}
\usepackage{amsmath} 
\usepackage{amssymb} 
\usepackage{amsthm}
\usepackage{amsfonts}            
\usepackage{mathrsfs}
\usepackage{amsthm,amsfonts,amssymb,amsmath,amsxtra,mathtools}
\usepackage[all]{xy}
\usepackage{mathtools}
\usepackage{tikz-cd}
\tikzcdset{scale cd/.style={every label/.append style={scale=#1},
    cells={nodes={scale=#1}}}}
\usepackage{mathabx}
\usepackage{colonequals}
\usepackage{xcolor}
\usepackage{todonotes}

 \usepackage{comment}

\usepackage{geometry}
\usepackage{setspace}

\usepackage[pagebackref]{hyperref} 
\hypersetup{  
    colorlinks=true,                               
    breaklinks=true,          
    urlcolor= black,                               
    linkcolor= blue,                                
    bookmarksopen=true,
    citecolor=black,
    linkcolor=red!50!black,            
    filecolor=black,
    citecolor=blue,
    linkbordercolor=blue
}

\makeatletter
\tikzset{
  edge node/.code={%
      \expandafter\def\expandafter\tikz@tonodes\expandafter{\tikz@tonodes #1}}}
\makeatother
\tikzset{
  subseteq/.style={
    draw=none,
    edge node={node [sloped, allow upside down, auto=false]{$\in$}}},
  Subseteq/.style={
    draw=none,
    every to/.append style={
      edge node={node [sloped, allow upside down, auto=false]{$\in$}}}
  }
}

\title{}
\author{}
\date{}
\begin{document}
\theoremstyle{definition}        
\newtheorem{definition}{Definition}[subsection] 
\newtheorem{example}[definition]{Example}
\newtheorem{remark}[definition]{Remark}

\theoremstyle{plain}
\newtheorem{theorem}[definition]{Theorem}
\newtheorem{lemma}[definition]{Lemma}         
\newtheorem{proposition}[definition]{Proposition}
\newtheorem{corollary}[definition]{Corollary}
\newtheorem{conjecture}[definition]{Conjecture}

\newcommand{\LL}{L^{\sharp}}
\newcommand{\Zp}{\mathbb{Z}_{p}}
\newcommand{\Z}{\mathbb{Z}}
\newcommand{\Qp}{\mathbb{Q}_{p}}
\newcommand{\Sp}{\mathbf{Sp}}
\newcommand{\RepL}{\textup{Rep}_{\Tilde{T}}^{\boldsymbol{\mu}}(\mathcal{O}_{F}/\varpi^{t})}
\newcommand{\Repl}{\textup{Rep}_{T,i}^{\boldsymbol{\mu}}(\mathcal{O}_{F}/\varpi^{t})}
\newcommand{\pr}{\prime}

\newcommand{\BX}{\mathbb{X}}
\newcommand{\BG}{\mathbb{G}}
\newcommand{\BV}{\mathbb{V}}
\newcommand{\BF}{\mathbb{F}}
\newcommand{\BD}{\mathbb{D}}
\newcommand{\BL}{\mathbb{L}}
\newcommand{\BW}{\mathbb{W}}

\newcommand{\CN}{\mathcal{N}}
\newcommand{\CM}{\mathcal{N}}
\newcommand{\M}{\mathcal{M}_L}
\newcommand{\Ms}{\mathcal{M}_{L^{\sharp}}}
\newcommand{\pis}{\pi_{L^{\sharp}}}
\newcommand{\CZ}{\mathcal{Z}}
\newcommand{\CY}{\mathcal{Y}}
\newcommand{\CO}{\mathcal{O}}
\newcommand{\zm}{\mathcal{Z}^{\textup{bl}}}
\newcommand{\ym}{\mathcal{Y}^{\textup{bl}}}

\newcommand{\Spf}{\mathrm{Spf}}
\newcommand{\Hom}{\mathrm{Hom}}
\newcommand{\End}{\mathrm{End}}
\newcommand{\disc}{\mathrm{disc}}
\newcommand{\id}{\mathrm{id}}

\newcommand{\Int}{\mathrm{Int}}
\newcommand{\univ}{\mathrm{univ}}
\newcommand{\GSpin}{\mathrm{GSpin}}
\newcommand{\GL}{\mathrm{GL}}
\newcommand{\SO}{\mathrm{SO}}

\newcommand{\RZ}{\mathrm{RZ}}
\newcommand{\Of}{\mathcal{O}_{F}}
\newcommand{\Den}{\mathrm{Den}}
\newcommand{\q}{q}
\newcommand{\exc}{\textup{Exc}}
\newcommand{\crys}{\textup{crys}}
\newcommand{\vcrys}{\mathbf{V}_{\textup{crys},z}}
\newcommand{\ofb}{\mathcal{O}_{\breve{F}}}
\newcommand{\cN}{\mathcal{N}}
\newcommand{\Q}{\mathbb{Q}}

\newcommand{\tc}{\textcolor{teal}}
\title[Weighed cycles]{Weighted special cycles on Rapoport--Zink spaces with almost self-dual level}

\author[Qiao He]{Qiao He}
\address{Department of Mathematics, 
	Columbia University,
	2990 Broadway,
	New York, NY 10027, USA}
\email{qh2275@columbia.edu}

\author[Zhiyu Zhang]{Zhiyu Zhang}
\address{Department of Mathematics, Stanford University,  450 Jane Stanford Way Building 380, Stanford, CA 94305}
\email{zyuzhang@stanford.edu}

\author[Baiqing Zhu]{Baiqing Zhu}
\address{Department of Mathematics, 
	Columbia University,
	2990 Broadway,
	New York, NY 10027, USA}
\email{bz2393@columbia.edu}

\begin{abstract}
We introduce a ``vector valued'' version of special cycles on $\GSpin$ Rapoport--Zink spaces with almost self-dual level in the context of the Kudla program, with certain linear invariance and local modularity features. They are local analogs of special cycles on $\GSpin$ Shimura varieties with almost self-dual parahoric level (e.g. Siegel threefolds with paramodular level). We establish local arithmetic Siegel--Weil formulas relating arithmetic intersection numbers of these special cycles and derivatives of certain local Whittaker functions in any dimension. The proof is based on a reduction formula for cyclic quadratic lattices.
\end{abstract}
\maketitle

\pagenumbering{roman}
\tableofcontents
\newpage
\pagenumbering{arabic}

\section{Introduction}
\subsection{Background}
The arithmetic of quadratic forms is a rich topic. For positive definite quadratic forms, the classical Siegel--Weil formula relates associated theta series to explicit Siegel-Eisenstein series on metaplectic groups. Kudla \cite{Kud97Duke} \cite{Kud97Ann} initiated an influential program on geometric and arithmetic analogs of theta correspondences, see the nice survey \cite{Li-survey}. For quadratic spaces of signature $(n-1, 2)$ over $\mathbb Q$, we expect an \emph{arithmetic Siegel--Weil formula} that relates arithmetic degrees of arithmetic theta series on associated $\GSpin$ Shimura varieties of signature $(n-1,2)$ to derivatives of certain Siegel Eisenstein series, which is crucial to prove an \emph{arithmetic inner product formula} that relates derivatives of automorphic $L$-functions for symplectic groups to special theta cycles on $\GSpin$ Shimura varieties. This generalizes Gross--Zagier formulas \cite{GZ86} on modular curves ($n=2$).  Moreover, we expect such arithmetic Seigel-Weil formulas over general totally real number fields.

Via Fourier-Whittaker expansion, the non-singular part of arithmetic Siegel--Weil formulas could be reduced to local arithmetic Siegel--Weil formulas at each local place. At archimedean places, such local formulas are known mainly by works of Garcia--Sankaran \cite{GS19} and Bruinier--Yang \cite{BY20}. At good primes, such local formulas are known by the ground-breaking work of Li--Zhang \cite{LiZhang-orthogonalKR}, which are precise identities between the arithmetic intersection numbers of special cycles on GSpin Rapoport--Zink spaces and the derivatives of local representation densities of quadratic forms. From above global perspective, we need to establish local arithmetic Siegel-Weil formulas at a bad prime $p$ where the GSpin Rapoport--Zink space has bad reduction. From a purely local perspective, the study of special cycles is also interesting and gives us an understanding of the geometry of Rapoport--Zink spaces, where new phenomena appear. Correspondingly, we may discover new interesting behaviors of the local representation densities for general quadratic lattices.

Throughout this paper, we fix a prime number $p>3$ and consider the simplest bad reduction case: Rapoport--Zink space with the special maximal parahoric level structure. Choosing such level is equivalent to choosing  a quadratic lattice $L$ of rank $m=n+1 \geq 3$ over $\Zp$ that is almost self-dual. In other words,  it satisfies $L^{\vee}/L\simeq\Z/p\Z$. The corresponding Rapoport--Zink space associated with this quadratic lattice is a regular formal scheme over $\breve{\mathbb{Z}}_p$ of total dimension $n$. The corresponding level group is the stabilizer of this quadratic lattice. From the global point of view, canonical integral models of GSpin Shimura varieties with such level structure (e.g. Siegel threefolds with paramodular level when $n=4$, see \cite{yu2011geometry}) have been constructed and studied by Madapusi \cite{Mad16}. Certain (weighted) special cycles on such integral models have been constructed in \cite{Mad16} and Andreatta--Goren--Howard--Madapusi \cite{AGHM} and Howard-Madapusi \cite{HM20,HM22}. We give a local analogue of their construction on the Rapoport--Zink spaces with nice properties, and establish the local arithmetic Siegel--Weil formula for these weighted cycles.
\par

Below is a table lists some low dimensional GSpin Shimura varieties with almost self-dual level at $p$ that happen to be of PEL type.
\begin{align*}
    \begin{array}{c|c}
\hline \text { PEL type special cases  } & \text {Isometric classes of  } L \\
\hline \text{Modular curve with level $\Gamma_0(pN)$}   \text {where $p \nmid N$ as in \cite{SSY,Zhu23}  } & L=H_2^+\obot \Lambda_1  \\
\text { Shimura curve over a ramified prime  as in \cite{KRshimuracurve,KRYbook,Ter08} } & L=H_2^-\obot \Lambda_1   \\
\text {Hilbert modular surface over a ramified prime as in \cite{BachmatGoren}  } &   L=H_3\obot \Lambda_1\\
\text {Siegel threefold with paramodular level as in \cite{yu2011geometry}, \cite{Wang} } & L=H_4^+\obot \Lambda_1 \\
\text {Quaternionic unitary Shimura variety as in \cite{OkiQuaternionic}, \cite{Wang}] } & L=H_4^-\obot \Lambda_1   \\
\text {GU(2,2) Shimura variety over a ramified prime as in \cite{OkiGU22}   } & L=H_5\obot \Lambda_1 
\end{array}
\end{align*}
Here $\Lambda_1$ is a $p$-modular lattice and $H_n$ is a self-dual lattice of rank $n$. When $n$ is even, $H_n^+$ (resp. $H_n^{-})$ is the lattice containing (resp. not containing) a rank $\frac{n}{2}$ totally isotropic lattice.

The (local) arithmetic Siegel-Weil formula for some of the above low-dimensional examples  have been studied before: Terstiege \cite{Ter08}, Kudla--Rapoport \cite{KRshimuracurve} and Kudla--Rapoport--Yang \cite{KRYbook} consider the Drinfeld space, which corresponds to the case $n=2$ and the self-dual part of $L=H_2^-\obot \Lambda_1$. Sankaran, Shi and Yang \cite{SSY} and Zhu \cite{Zhu23} consider the modular curve case where   $L=H_2^+\obot \Lambda_1$.   Our work may be regarded as a systematic way to understand the geometry of  special cycles on such Rapoport--Zink spaces. 

We also remark that  odd dimensional exotic smooth model for a unitary RZ space over a ramified prime is analogous to the GSpin RZ space with almost self-dual level. The local arithmetic Siegel-Weil formula in this case was proved by Haodong Yao in \cite{Yao}.

\subsection{Main results}
Let $p$ be an odd prime (may equal to $3$ for now). Let $F=\mathbb{Q}_p$ with uniformizer $\varpi=p$ and residue field sized $q=p$. Let $\breve{F}$ be the completion of the maximal unramified extension of $F$. Let $m = n +1 \geq 3$ be an integer. Let $L$ be a quadratic $\mathcal{O}_F$-lattice such that $L^{\vee}/L\simeq\mathcal{O}_F/(\varpi)$. There exists a unique self-dual quadratic $\mathcal{O}_F$-lattice $L^{\sharp}$ of rank $m+1$ such that $L$ can be isometrically embedded into $L^{\sharp}$. Fix such an embedding $\iota_L:L\rightarrow L^{\sharp}$, define $\Lambda=\{x\in L^{\sharp}:(x,y)=0,\,\,\textup{for all}\,\,y\in L\}$. Then $\Lambda$ is a rank 1 quadratic lattice spanned by an element $x_0\in L^{\sharp}$. The self-duality of $L^{\sharp}$ implies that there exists a natural isomorphism $r:\Lambda^{\vee}/\Lambda\xrightarrow{\sim}L^{\vee}/L$ (Lemma \ref{cyclic duality for L and Lambda}).
\par
Associated to $L$ we have a local Shimura--Hodge data $(G, b, \mu, C^{\sharp})$ constructed by Oki \cite{Oki20}, where $G = \textup{GSpin}(V_F), b \in G( \breve{F})$ is a basic element, $\mu : \mathbb{G}_m \rightarrow G$ is a certain cocharacter, and $C^{\sharp} = C(L^{\sharp} )$ is the Clifford algebra of $L^{\sharp}$. Associated to this local Shimura--Hodge data, we have a GSpin Rapoport--Zink space $\textup{RZ}_L$ of Hodge type with parahoric level structure constructed by Hamacher and Kim \cite{HodgeRZparahoric}. The space $\textup{RZ}_L$ is a regular formal scheme over $\textup{Spf}\,\ofb$, formally locally of finite type and of relative dimension $n -1$ over $\textup{Spf}\,\ofb$ (see $\S$\ref{rz-con} for more details), and admits a decomposition $\textup{RZ}_L = \bigsqcup\limits_{l\in\mathbb{Z}}\textup{RZ}^{(l)}_L$ into (isomorphic) connected components. Define $\CN_L=\textup{RZ}^{(0)}_L$ to be a connected component of $\textup{RZ}_L$, a regular formal scheme of (total) dimension $n$.

Denote by $\mathbb{X}$ the framing $p$-divisible group over $\overline{\mathbb{F}}_p$ for $\textup{RZ}_L$, by $X^{\textup{univ}}$ the universal $p$-divisible group over $\CN_L$. The element $x_0\in L^{\sharp}$ acts by left-multiplication on $C(L^{\sharp})\otimes_{\mathcal{O}_F}\ofb\simeq\mathbb{D}(\mathbb{X})$ where $\mathbb{D}(\mathbb{X})$ is the (covariant) Dieudonne module of $\mathbb{X}$. Hence $x_0$ can be viewed as a quasi-endomorphism of $\mathbb{X}$. The quasi-isogeny $x_0$ turns out to deform to an endomorphism of $X^{\textup{univ}}$ (Theorem \ref{almost-divisor}, see also \cite[Theorem 4.5 (i)]{Oki20}).
Let $\mathbb{V}$ be the unique (up to isomorphism) quadratic space over $F$ of dimension $m$, Hasse invariant $\epsilon(\mathbb{V}) = -\epsilon(L_F)$ and $\chi(\mathbb{V}) = \chi(L_F )$ (see $\S$\ref{not-ql} for notations of quadratic spaces). Then $\mathbb{V}$ can be identified as the space of special quasi-endomorphisms $\mathbb{V}\subset\textup{End}(\mathbb{X})\otimes\mathbb{Q}$. For a subset $M\subset\mathbb{V}$, the special cycle $\mathcal{Z}(M)$ (Definition \ref{spe-almo}) is a closed formal subscheme of $\CN_L$, over which all special quasi-endomorphisms $x\in M$ deform to endomorphisms.
\par
The above construction of $\mathcal{Z}$-cycles follows the same line as the unramified case in the works of Li--Zhu \cite{Li_Zhu_2018}. We obtain a bunch of new kinds of cycles:
\begin{itemize}
    \item [(a).]$\mathcal{Y}$-cycles: For a subset $M\subset\mathbb{V}$, the special cycle $\mathcal{Y}(M)$ (Definition \ref{cycle-on-blow-up}) is a closed formal subscheme of $\CN_L$, over which all special quasi-endomorphisms $x_0\circ x\in M$ deform to endomorphisms.
    \item[(b).] Weighted cycles: For an element $\mu\in L^{\vee}/L$ and an element $x\in\mathbb{V}$, the weighted cycle $\mathcal{Z}(x,\mu)$ (Definition \ref{cycle-on-blow-up}) is a closed formal subscheme of $\CN_L$, over which all special quasi-endomorphisms $x+\eta(\mu)$ deform to endomorphisms where $\eta:L^{\vee}/L\rightarrow\Lambda^{\vee}$ is an arbitrary lift of the natural isomorphism $r^{-1}: L^{\vee}/L\xrightarrow{\sim}\Lambda^{\vee}/\Lambda$. Notice that when $\mu=0$, this recovers the $\mathcal{Z}$-cycle $\mathcal{Z}(x)$.
\end{itemize}
The $\CY$-cycles defined here should be regarded as an analogue of the $\CY$-cycles studied in the unitary case (see \cite{Cho,CHZ} for example).  

\par
Let $x\in\mathbb{V}$ be a non-zero element. Contrary to the $\mathcal{Z}$-cycle $\mathcal{Z}(x)$, the $\mathcal{Y}$-cycle $\mathcal{Y}(x)$ is quite complicated to understand because of the presence of $x_0$ in the deformation condition and \emph{fails} to be a divisor. In this article, we come up with a way to remedy this. The set of non-formally smooth points $\CN_L^{\textup{nfs}}$ of $\CN_L$ is discrete and is indexed by vertex lattice of type $1$ in $\mathbb{V}$. Let $\pi_L:\M\rightarrow\CN_L$ be the blow up morphism of the formal scheme $\CN_L$ at non-formally smooth points, $\M$ is also a regular formal scheme over $\ofb$ of total dimension $n$. Denote by $\mathcal{Y}^{\textup{bl}}(x)$ (resp. $\mathcal{Z}^{\textup{bl}}(x,\mu)$) the pullback of $\mathcal{Y}(x)$ (resp. $\mathcal{Z}(x,\mu)$) to $\M$.
\begin{proposition}[Proposition \ref{decom-Y}]
    The special cycle $\ym(x)\subset\M$ is an effective Cartier divisor. Moreover, we have the following equality as Cartier divisors on $\M$:
    \begin{equation*}
        \ym(x)=\sum\limits_{\mu\in\Lambda^{\vee}/\Lambda}\zm(x,\mu)+\exc_L(x),
    \end{equation*}
    where $\exc_L(x)=\bigsqcup\limits_{z\in\CN_L^{\textup{nfs}}\cap\mathcal{Y}(x)}\pi_L^{-1}(z)$.
\end{proposition}

The above proposition suggests that the formal scheme $\M$ is the ``correct" model that we should work with when we consider arithmetic intersection problems between different kinds of cycles. Let $M\subset\mathbb{V}$ be an $\mathcal{O}_F$-lattice of rank $n$ with basis $\boldsymbol{x}=\{x_i\}_{i=1}^{n}$. Define the arithmetic intersection numbers of $\mathcal{Z}$-cycles and $\mathcal{Y}$-cycles to be (Definition \ref{int-spe-cycle})
\begin{align*}
    \Int^{\mathcal{Z}}(\M,M)=\chi\left(\M,[\CO_{\mathcal{Z}^{\textup{bl}}(x_1)}\otimes^{\mathbb{L}}_{\CO_{\M}}\cdots\otimes^{\mathbb{L}}_{\CO_{\M}}\CO_{\mathcal{Z}^{\textup{bl}}(x_n)}]\right),\\
    \Int^{\mathcal{Y}}(\M,M)=\chi\left(\M,[\CO_{\mathcal{Y}^{\textup{bl}}(x_1)}\otimes^{\mathbb{L}}_{\CO_{\M}}\cdots\otimes^{\mathbb{L}}_{\CO_{\M}}\CO_{\mathcal{Y}^{\textup{bl}}}(x_n)]\right).
\end{align*}
The \textit{linear invariance} property (Lemma \ref{linear-invariance-Y}, Corollary \ref{linear-invariance-al}) guarantees that the above definition is independent of the choice of the basis $\boldsymbol{x}=\{x_i\}_{i=1}^{n}$. For an element $\boldsymbol{\mu}\in(L^{\vee}/L)^{n}$, define the arithmetic intersection numbers of weighted cycles to be
\begin{equation*}
    \Int_{\boldsymbol{\mu}}(\M,\boldsymbol{x})=\chi\left(\M,{^{\mathbb{L}}\mathcal{Z}_{\boldsymbol{\mu}}^{\textup{bl}}}(\boldsymbol{x})\coloneqq[\CO_{\mathcal{Z}^{\textup{bl}}(x_1,\mu_1)}\otimes^{\mathbb{L}}_{\CO_{\M}}\cdots\otimes^{\mathbb{L}}_{\CO_{\M}}\CO_{\mathcal{Z}^{\textup{bl}}(x_n,\mu_n)}]\right).
\end{equation*}
Our main theorem will connect the above purely geometric quantities to purely analytic quantities --- (derived) local densities.
\par
For two quadratic lattices $L$ and $M$, Define $\textup{Rep}_{L,M}$ to be the scheme of integral representations, an $\mathcal{O}_F$-scheme such that for any $\mathcal{O}_F$-algebra $R$, $\textup{Rep}_{L,M}(R) = \textup{QHom}(M\otimes_{\mathcal{O}_F}
R, L\otimes_{\mathcal{O}_F}R)$. Here $\textup{QHom}$ denotes the set of homomorphisms of quadratic modules. The local density of associated to $L$ and $M$ is defined to be
\begin{equation*}
    \textup{Den}(L,M)=\lim\limits_{d\rightarrow\infty}\frac{\#\textup{Rep}_{L,M}(\mathcal{O}_F/\varpi^{d})}{q^{d\cdot \textup{dim}(\textup{Rep}(L,M))_{F}}}.
\end{equation*}
Let $H_{2k}^{+}=(H_{2}^{+})^{\obot k}$ be a self-dual lattice of rank $2k$, where the quadratic form on $H_{2}^{+}=\mathcal{O}_F^{2}$ is given by $(x,y)\in \mathcal{O}_F^{2}\mapsto xy$. There exist local density polynomials $\Den(L,M,X),\Den(L^{\vee},M,X)\in\mathbb{Q}[X]$ such that
\begin{align*}
    \Den(L,M,X)\big\vert_{X=q^{-k}}=\Den(L\obot H_{2k}^{+},M),\,\,\,\,\Den(L^{\vee},M,X)\big\vert_{X=q^{-k}}=\Den(L^{\vee}\obot H_{2k}^{+},M).
\end{align*}
Both polynomial vanishes at $X=1$ since $M\subset\mathbb{V}$ and hence cannot be embedded into $L$ and $L^\vee$. Suppose we have a decomposition $L=H\obot\langle x\rangle$ where $H$ is a self-dual quadratic lattice of rank $n$ and $x\in L$ with $\nu_\varpi(q(x))=1$. We consider the (normalized) \textit{derived local density}
\begin{align*}
    \partial\Den^{L}\left(M\right)\coloneqq-2\cdot\frac{\textup{d}}{\textup{d}X}\bigg\vert_{X=1}\frac{\Den(L,M,X)}{\textup{Nor}^{\flat}(1,H)},\,\,\,\,
        \partial\Den^{L^{\vee}}\left(M\right)\coloneqq-2q^{n}\cdot\frac{\textup{d}}{\textup{d}X}\bigg\vert_{X=1}\frac{\Den(L^{\vee},M,X)}{\textup{Nor}^{\flat}(1,H)},
\end{align*}
where $\textup{Nor}^{\flat}(1,H)$ is normalizing factor (Definition \ref{normal1}).
\par
The (derived) local densities are closely related to local Whittaker functions $W_T(g,s,\varphi)$ of Siegel Eisenstein series in the following way: Let $\varphi\in\mathscr{S}(\mathbb{V}^{n})$ be the characteristic function of the lattice $L^{n}$, let $T\in\textup{Sym}_n(F)$ be the inner product matrix of an $\mathcal{O}_F$-basis of $M$, then for all positive integers $k\geq0$, we have
\begin{align*}
W_T\left(1, k, \varphi\right)=\gamma(L)\cdot q^{-n/2}\cdot\operatorname{Den}(  L\obot H_{2k}^{+}, M ).
\end{align*}
The main theorem of this article is the following (which is a generalization of local arithmetic Siegel-Weil formula \cite[Thm 1.2.1.]{LiZhang-orthogonalKR} to our levels):
\begin{theorem}[Theorem \ref{main}]\label{main-intro}
Let $M\subset\mathbb{V}$ be a $\Of$-lattice of rank $n$. Let $\boldsymbol{x}=(x_1,\cdots,x_n)\in\mathbb{V}^{n}$ be an element such that $\{x_{i}\}_{i=1}^{n}$ be a basis of $M$. Let $T$ be the inner product matrix of $M$ with respect to the basis $\boldsymbol{x}$. Let $\boldsymbol{\mu}\in(L^{\vee}/L)^{n}$ be an element. Then
    \begin{align*}
        \Int_{\boldsymbol{\mu}}(\M,\boldsymbol{x})=\Int_{\boldsymbol{\mu}}(\CN_L,\boldsymbol{x})=\frac{2q^{n/2}}{\gamma(L)\cdot\log(q)\cdot\textup{Nor}(1,H)}\cdot W_{T}^{\prime}(1,0,1_{\boldsymbol{\mu}+L^{n}}).
    \end{align*}
Especially, when $\boldsymbol{\mu}=\boldsymbol{0}$, we have
\begin{equation}
    \Int^{\mathcal{Z}}(\M,M)=\Int(\CN_L,M)=\partial\Den^{L}(M).
\end{equation}
Moreover,
\begin{equation}\label{eq: main Y}
    \Int^{\mathcal{Y}}(\M,M)=\partial\Den^{L^{\vee}}(M)+(-1)^{n}\cdot\Den^{\mathcal{L}^{\vee}}(M),
\end{equation}
where $\mathcal{L}\subset\mathbb{V}$ is the unique (up to isometry) vertex lattice of type $1$ and $\Den^{\mathcal{L}^{\vee}}(M)=q^{n}\cdot\frac{\Den(\mathcal{L}^{\vee},M)}{\textup{Nor}^{\flat}(1,H)}$.
\end{theorem}

\subsection{Strategy of the proof}
On the geometric side, there are two new main ingredients (1).  understanding the geometry of the $\mathcal{Y}$-cycles. (2). relating the local arithmetic intersection numbers to derived local densities.
\par
Our method is based on the fact that there is a closed immersion $\CN_L$ to the $\GSpin$ Rapoport--Zink spaces $\CN_{L^\sharp }$ associated with the self-dual lattice $L^\sharp$. The closed immersion identifies $\CN_L$ with the special divisor $\CZ^{\sharp}(x_0)$ of $\CN_{L^{\sharp}}$ (Theorem \ref{almost-divisor}). The closed immersion $\CN_L\rightarrow\CN_{L^{\sharp}}$ enables us to give explicit deformation descriptions of the formal scheme $\CN_{L}$ and special cycles on it (Lemma \ref{deformo2}) based on that of $\CN_{L^{\sharp}}$ (Lemma \ref{deformo}). The proof of Proposition \ref{decom-Y} is based on this moduli description.
\par
The closed immersion $\CN_L\rightarrow\CN_{L^{\sharp}}$ also enables us to reduce the computations of local arithmetic intersection numbers on $\CN_L$ to $\CN_{L^{\sharp}}$, on which the relation between derived local densities and local arithmetic intersection numbers are readily proved \cite[Theorem 1.2.1]{LiZhang-orthogonalKR}. The succinct relation of the geometric side motivates us to prove a relation between local Whittaker functions of genus $n$ and $n+1$ (Theorem \ref{reduction}).
\par
This idea has appeared before to understand $\mathcal{Z}$-cycles, see e.g. the work of Terstiege \cite{Ter08}. In the works of Sankaran, Shi and Yang \cite{SSY} and Zhu \cite{Zhu23}, the modular curve $\mathcal{X}_0(p)$ is embedded into the product $\mathcal{X}_0(1)\times\mathcal{X}_0(1)$. Our paper uses such embeddings to understand $\mathcal{Y}$-cycles systematically.
\par
We remark that global analogs of such special cycles on integral models of orthogonal Shimura varieties with maximal levels are defined in the work of Howard--Madapusi \cite[Section 6.2, Proposition 6.4.3.]{HM20}\cite[Lemma 3.3.1.]{HM22} by similar isometric embedding methods.

\subsection{Structure of the paper}
In Part 1, we study cyclic quadratic lattices and construct embedding to self-dual lattices ($\S$2). In $\S$3, we explain the local densities and their relations to local Whittaker functions, we also prove the relation between local Whittaker functions of genus $n$ and $n+1$.
\par
In Part 2, we introduce several families ($\mathcal{Z}$, $\mathcal{Y}$ and weighted) of special Kudla--Rapoport cycles on GSpin Rapoport--Zink space $\CN_L$ associated to $L$ in $\S$4. In $\S$5, we study the total transform of these cycles under the blow up at the non-formally smooth points of $\CN_L$. We prove the linear invariance properties and form intersection numbers. Then we give the proof of Theorem \ref{main-intro}. In $\S$6, we consider the local modularity of the derived cycles.

\subsection{Notations}
Throughout this article, we fix an odd prime $p$. Let $F/\mathbb{Q}_{p}$ be a finite extension, with ring of integers $\mathcal{O}_{F}$, and uniformizer $\varpi$. Let $\mathbb{F}$ be the algebraic closure of $\mathbb{F}_p$. Let $\breve{F}$ be the completion of the maximal unramified extension of $F$. Let $\sigma\in\textup{Gal}(\Breve{F}/F)$ be the arithmetic Frobenius automorphism of the field $\breve{F}$. Let $\ofb$ be the integer ring of $\breve{F}$. 
\par
For a scheme or formal scheme $S$ over $\ofb$, denote by $\textup{NCRIS}_{\ofb}(S/\textup{Spec}\,\ofb)$ the big fppf nilpotent crystalline site of $S$ over $\ofb$ (cf. \cite[Definition B.5.7.]{FGL08}). Denote by $\mathcal{O}_{S}^{\textup{crys}}$ the structure sheaf in this site. For a point $z\in S(\mathbb{F})$, let $\widehat{\mathcal{O}}_{S,z}$ be the complete local ring of $S$ at the point $z$. Let $\textup{Nilp}_{\ofb}$ be the category of $\ofb$-schemes on which $\varpi$ is locally nilpotent. For an object $S$ in $\textup{Nilp}_{\ofb}$, we use $\overline{S}$ to denote the scheme $S\times_{\ofb}\mathbb{F}$. 
\par
Denote by $\textup{Alg}_{\ofb}$ the category of noetherian $\varpi$-adically complete $\ofb$-algebras. Denote by $\textup{Art}_{\ofb}$ the category of local artinian $\ofb$-algebras with residue field $\mathbb{F}$. Denote by $\textup{ANilp}_{\ofb}^{\textup{fsm}}$ be the category of noetherian adic $\ofb$-algebras, in which $\varpi$ is nilpotent, which are formally finitely generated and formally smooth over $\ofb/\varpi^{k}$ for some $k\geq1$.

\subsection{Acknowledgement} 
We thank Chao Li, Keerthi Madapusi, Siddarth Sankaran, Shou-Wu Zhang and Wei Zhang for their helpful comments.

\part{Analytic side}

\section{Reductions on quadratic lattices}
Let $p$ be an odd prime. Let $F$ be a $p$-adic local field with ring of integers $\Of$, uniformizer $\varpi\in\Of$ and residue field $\Of/(\varpi)\simeq \mathbb{F}_{q}$ of size $q$. Let $\nu_F: F\rightarrow\mathbb{Z}\cup\{\infty\}$ be the normalized valuation on $F$ determined by $\varpi$. 

\subsection{Notations on quadratic lattices}\label{not-ql}
A quadratic lattice $(M,q)$ is a finite free $\mathcal{O}_{F}$-module equipped with a quadratic form $q: M\rightarrow F$. Let $M_F\coloneqq M\otimes_{\Of}F$ be the $F$-vector space associate to $M$. The quadratic form $q$ induces a symmetric bilinear form $M_F\times M_F\stackrel{(\cdot,\cdot)}\longrightarrow F$ by $(x,y)=q(x+y)-q(x)-q(y)$. Let $M^{\vee}=\{x\in M\otimes_{\mathcal{O}_{F}}F:(x,M)\subset \mathcal{O}_{F}\}$. We say a quadratic lattice $M$ is integral if $q(x)\in\mathcal{O}_{F}$ for all $x\in M$, is self-dual if it is integral and $M$ = $M^{\vee}$.
\par
We assume $\textup{rank}_{\Of}M=m$ and the symmetric bilinear form $(\cdot,\cdot)$ on $M_F$ is nondegenerate. Let $\{x_{i}\}_{i=1}^{m}$ be an $\Of$ basis of $M$, and $t_{ij}=\frac{1}{2}(x_{i},x_{j})$, the discriminant of the quadratic lattice $M$ is
\begin{equation*}
    \textup{disc}(M) = (-1)^{m\choose 2}\textup{det}\left((t_{ij})\right)\in F^{\times}/(F^{\times})^{2}.
\end{equation*}
Let $\chi_F: F^{\times}/(F^{\times})^{2}\rightarrow\{\pm1,0\}$ be the quadratic residue symbol, we define $\chi_F(M)=\chi_F(\textup{disc}(M))$. Then $\chi_F(M) = 0$ if and only if $\nu_F(M):=\nu_F(\disc(M))$ is odd.
\par
We may choose $\{x_{i}\}_{i=1}^{n}$ to be an orthogonal basis of the quadratic space $M_F$, i.e. $t_{ij}=0$ if $i\neq j$ and $t_{ii}\neq 0$. Let $\mathrm{Hilb}(\cdot,\cdot)_{F}$ be the Hilbert symbol on the local field $F$. The Hasse invariant of the quadratic space $M_F$ is
\begin{equation*}
    \epsilon(M_F) = \prod\limits_{1 \leq i<j \leq m} \mathrm{Hilb}(t_{ii},t_{jj})_{F}\in\{\pm1\}.
\end{equation*}
\subsection{Sublattices of self-dual quadratic lattices}
Let $M$ be a self-dual lattice over $\Of$ of rank $m$, then $\epsilon(M)=1$ and $\chi_F(M)\neq0$. Two self-dual quadratic lattices $M$ and $M^{\pr}$ of the same rank are isometric if and only if $\chi_F(M)=\chi_F(M^{\pr})$. Therefore in the following, we will use $H_m^{\varepsilon}$ to denote the unique (up to isometry) self-dual lattice of rank $m$ and $\chi_F(M)=\varepsilon\in\{\pm1\}$.
\par
Let $m\geq2$ be a positive integer and $\varepsilon\in\{\pm1\}$, let $x_0\in H_m^{\varepsilon}$ be an element such that $\nu_F(q(x_0))=1$. Let $K=\{x\in H_m^{\varepsilon}: (x,x_0)=0\}$ be the sublattice of $H_m^{\varepsilon}$ orthogonal to $x_0$.
\begin{lemma}
    Assume that $q(x_0)=\varepsilon_0\cdot\varpi$ where $\varepsilon_0\in\Of^{\times}$, then 
    \begin{equation*}
        \textup{disc}(K)=(-1)^{m-1}\varepsilon_0\varepsilon\varpi,\,\,\,\epsilon(K)=(\varpi,(-1)^{m-2\choose 2}\varepsilon)=\chi_F((-1)^{m-2\choose 2}\varepsilon).
    \end{equation*}
\end{lemma}
\label{unique}
\begin{proof}
    Since $\nu_F(q(x_0))=1$, we have $x_0\notin\varpi H_m^{\varepsilon}$, hence there exists $y_0\in H_m^{\varepsilon}$ such that $(x_0,y_0)=1$. Let $\langle x_0,y_0\rangle$ be the $\Of$-lattice spanned by $x_0$ and $y_0$, then we have a decomposition
    \begin{equation*}
        H_m^{\varepsilon}\simeq\langle x_0,y_0\rangle\obot H_{m-2}^{\varepsilon}.
    \end{equation*}
    By this decomposition, it's easy to verify that $K\simeq\Of\cdot y\obot H_{m-2}^{\varepsilon}$ where $y\in \langle x_0,y_0\rangle$ is an element such that $q(y)=-q(x_0)=-\varepsilon_0\varpi$. Then the lemma follows from the definition of the discriminant $\textup{disc}(K)$ and the Hasse invariant $\epsilon(K)$.
\end{proof}

\subsection{Cyclic quadratic lattices}
Let $n\geq1$ be an integer, let $L$ be a non-degenerate quadratic lattice of finite rank over $\Of$. Let $V=L\otimes_{\Of}F$ be the quadratic space associated to $L$. Let $L^{\vee}=\{x\in V: (x,L)\subset\mathcal{O}_{F}\}$.
\begin{definition}
    A nondegenerate quadratic lattice $L$ over $\mathcal{O}_F$ is called cyclic if $L^{\vee}/L\simeq\Of/(\varpi)^{l}$ for some non-negative integer $l$. We call the number $l$ the order of the cyclic lattice $L$. When $l=1$, we also say $L$ is an almost self-dual lattice. 
\end{definition}
For a cyclic quadratic lattice $L$ of order $l$, there exists $[\frac{l}{2}]+1$ cyclic quadratic lattices $\{L(i)\}_{i=0}^{[l/2]}$ such that
\begin{equation}
    L=L(0)\subset L(1)\subset L(2)\subset\cdots\subset L\left([\frac{l}{2}]\right)\subset L\left([\frac{l}{2}]\right)^{\vee}\subset L(1)^{\vee}\subset L(0)^{\vee}=L^{\vee},
    \label{intermediate}
\end{equation}
and $L(i+1)/L(i)\simeq\mathbb{F}_q$ for any $0\leq i\leq [\frac{l}{2}]-1$. The lattice $L(i)$ corresponds to the subgroup $(\varpi^{l-i})/(\varpi^{l})$ of $L^{\vee}/L$.
\begin{lemma}\label{embedding L to self dual}
    Let $L$ be a cyclic quadratic lattice of rank $n+1$ and order $l\geq1$, then there exists a unique (up to isometry) self-dual quadratic lattice $L^{\sharp}$ of rank $n+2$ and a primitive isometric embedding $\iota: L\rightarrow L^{\sharp}$.
\end{lemma}
\begin{proof}
    The existence of such $L^{\sharp}$ is proved in \cite[Lemma 6.8]{Mad16}. Now we prove its uniqueness (up to isometry). By the classification of self-dual lattices, we only need to show that its discriminant is determined by $L$. By Lemma \ref{unique}, we have $\epsilon(L)=\chi_F\left((-1)^{n\choose 2}\textup{disc}(L^{\sharp})\right)$, therefore $\textup{disc}(L^{\sharp})\in F^{\times}/(F^{\times})^{2}$ is determined by $L$.
\end{proof}
From now on, we fix a self-dual quadratic lattice $L^{\sharp}$ over $\Of$ of rank $n+2$ and a primitive isometric embedding $\iota: L\rightarrow L^{\sharp}$. Let $V^{\sharp}=L^{\sharp}\otimes_{\Of}F$. The isometric map $\iota$ can be extended to an isometric map $\iota:V\rightarrow V^{\sharp}$. Let $\Lambda=\{x\in L^{\sharp}:(x,y)=0\,\,\textup{for any}\,\,y\in L\}$ be the sublattice of $L^{\sharp}$ orthogonal to $L$, then $\Lambda$ has $\Of$-rank 1. 
\begin{lemma}\label{cyclic duality for L and Lambda}
    Let $L$ be a cyclic quadratic lattice, there is a canonical isomorphism of $\Of$-modules,
    \begin{equation*}
        r:\Lambda^{\vee}/\Lambda\stackrel{\sim}\longrightarrow L^{\vee}/L,
    \end{equation*}
    such that under the natural morphism $\Lambda^{\vee}/\Lambda\oplus L^{\vee}/L\rightarrow V^{\sharp}/L^{\sharp}$, the morphism $r$ satisfies that $r(\overline{\mu})+\overline{\mu}$ is mapped to $0\in V^{\sharp}/L^{\sharp}$ for every $\mu\in\Lambda^{\vee}$.
    \label{dual}
\end{lemma}
\begin{proof}
    The primitivity of $\iota$ implies the following two exact sequences of $\Of$-modules,
    \begin{align}
        0\longrightarrow \Lambda\longrightarrow L^{\sharp}\longrightarrow L^{\vee}\longrightarrow 0,\label{e1}\\
        0\longrightarrow L\longrightarrow L^{\sharp}\longrightarrow \Lambda^{\vee}\longrightarrow 0.\label{e2}
    \end{align}
    The exact sequence (\ref{e1}) implies that $L^{\sharp}/(\Lambda\oplus L)\simeq L^{\vee}/L$, the second exact sequence (\ref{e2}) implies that $L^{\sharp}/(\Lambda\oplus L)\simeq \Lambda^{\vee}/\Lambda$, therefore $\Lambda^{\vee}/\Lambda\simeq L^{\sharp}/(\Lambda\oplus L)\simeq L^{\vee}/L$, let $r:\Lambda^{\vee}/\Lambda\stackrel{\sim}\longrightarrow L^{\vee}/L$ be the composite isomorphism, the last assertion of $r$ can be easily checked.
\end{proof}
Let $l$ be the order of the cyclic quadratic lattice $L$. For every $0\leq i\leq [\frac{n}{2}]$, let $\Lambda(i)=\varpi^{-i}\Lambda$, there exists a self-dual lattice $L^{\sharp}(i)$ of $V^{\sharp}$ such that $\varpi^{i}L^{\sharp}\subset L^{\sharp}(i)\subset\varpi^{-i}L^{\sharp}$, and the morphism $\iota$ restricts to a primitive isometric embedding $\iota: \Lambda(i)\obot L(i)\rightarrow L^{\sharp}(i)$.
\begin{equation}
    \begin{tikzcd}
    {\Lambda(i)\obot L(i)}
    \arrow[d] \arrow[r, "{\iota}"]
    & {L^{\sharp}(i)}
    \arrow[r]
    &{\varpi^{-i}L^{\sharp}}
    \arrow[d]
     \\
{\Lambda(i+1)\obot L(i+1)}
    \arrow[r, "{\iota}"]
    & {L^{\sharp}(i+1)}
    \arrow[r]
    &{\varpi^{-i-1}L^{\sharp}.}
\end{tikzcd}
\label{compatibility}
\end{equation}
\begin{lemma}
    Let $L$ be a cyclic quadratic lattice of order $l$, for every $0\leq i\leq [\frac{n}{2}]$, let $r_i:\Lambda(i)^{\vee}/\Lambda(i)\stackrel{\simeq}\rightarrow L(i)^{\vee}/L(i)$ be the canonical isomorphism, then the following diagram commutes,
    \begin{equation}
        \begin{tikzcd}
    {\Lambda^{\vee}/\Lambda}
    \arrow[d] \arrow[r, "{r}"]
    & {L^{\vee}/L}
    \arrow[d]
    \\
    {\Lambda^{\vee}/\Lambda(i)}
    \arrow[d, "\times\varpi^{i}"'] \arrow[r, "{\simeq}"]
    & {L^{\vee}/L(i)}
    \arrow[d, "\times\varpi^{i}"]
     \\
{\Lambda(i)^{\vee}/\Lambda(i)}
    \arrow[r, "{r_i}"]
    & {L(i)^{\vee}/L(i)}
    \label{relation}
\end{tikzcd}
    \end{equation}
\label{compatible}
\end{lemma}
\begin{proof}
    Let $\overline{\mu}\in\Lambda^{\vee}/\Lambda$ where $\mu\in\Lambda^{\vee}$, let $x\in L^{\vee}$ be a lift of $r(\overline{\mu})\in L^{\vee}/L$, then we have $x+\mu\in L^{\sharp}$, hence $\varpi^{i}x+\varpi^{i}\mu\in\varpi^{i}L^{\sharp}\subset L^{\sharp}(i)$, therefore $r_i(\overline{\varpi^{i}\mu})=\overline{\varpi^{i}x}=\varpi^{i}\cdot\overline{x}$, i.e., the diagram (\ref{relation}) commutes.
\end{proof}

\section{Vector valued Whittaker functions}\label{sec: analytic side} 

\subsection{Local densities}
\begin{definition}
Let $L, M$ be two quadratic $\mathcal{O}_{F}$-lattices. Let $\textup{Rep}_{M,L}$ be the scheme of integral representations, an $\mathcal{O}_{F}$-scheme such that for any $\mathcal{O}_{F}$-algebra $R$,
\begin{equation*}
    \textup{Rep}_{M,L}(R)=\textup{QHom}(L\otimes_{\mathcal{O}_{F}}R, M\otimes_{\mathcal{O}_{F}}R),
\end{equation*}
where \textup{QHom} denotes the set of injective module homomorphisms which preserve the quadratic forms. The local density of integral representations is defined to be
\begin{equation*}
    \textup{Den}(M,L)=\lim\limits_{d\rightarrow\infty}\frac{\#\textup{Rep}_{M,L}(\mathcal{O}_{F}/\pi^{d})}{q^{d\cdot \textup{dim}(\textup{Rep}_{M,L})_{F}}}.
\end{equation*}
\end{definition}
\begin{remark}
If $L,M$ have rank $n,m$ respectively and the generic fiber $(\textup{Rep}_{M,L})_{F}\neq\varnothing$, then $n\leq m$ and $\textup{dim}(\textup{Rep}_{M,L})_{F}=mn-\frac{n(n+1)}{2}$.
\end{remark}
\begin{definition}
Let $L, M$ be two quadratic $\mathcal{O}_{F}$-lattices. Let $\textup{PRep}_{M,L}$ be the $\mathcal{O}_{F}$-scheme of primitive integral representations such that for any $\mathcal{O}_{F}$-algebra $R$,
\begin{equation*}
    \textup{PRep}_{M,L}(R)=\{\phi\in\textup{Rep}_{M,L}(R):\textup{$\phi$ is an isomorphism between $L_{R}$ and a direct summand of $M_{R}$}\}.
\end{equation*}
where $L_{R}$ (resp. $M_{R}$) is $L\otimes_{\mathcal{O}_{F}}R$ (resp. $M\otimes_{\mathcal{O}_{F}}R$). The primitive local density is defined to be
\begin{equation*}
    \textup{Pden}(M,L)=\lim\limits_{d\rightarrow\infty}\frac{\#\textup{PRep}_{M,L}(\mathcal{O}_{F}/\pi^{d})}{q^{d\cdot \textup{dim}(\textup{Rep}_{M,L})_{F}}}.
\end{equation*}
\end{definition}
\begin{remark}
For a positive integer $d$, a homomorphism $\phi\in\textup{Rep}_{M,L}(\mathcal{O}_{F}/\pi^{d})$ or $\textup{Rep}_{M,L}(\mathcal{O}_{F})$ is primitive if and only if $\overline{\phi} \coloneqq \phi\,\,\textup{mod}\,\pi\in\textup{PRep}(\mathcal{O}_{F}/\pi)$, which is equivalent to $\textup{dim}_{\mathbb{F}_{q}}(\phi(L)+\pi\cdot M)/\pi\cdot M = \textup{rank}_{\mathcal{O}_{F}}(L)$. 
\end{remark}
\begin{definition}
Let $n\geq2$, for $\varepsilon\in\{\pm 1\}$, we define the normalizing factors to be
    \begin{align*}
    \textup{Nor}(X,H_{n}^{\varepsilon})=\left(1-\frac{1+(-1)^{n}}{2}\cdot\varepsilon q^{-(n)/2}X\right)\prod\limits_{1\leq i<n/2}\left(1-q^{-2i}X^{2}\right),\,\,\textup{Nor}^{\flat}(X,H_{n}^{\varepsilon})=2\cdot\textup{Nor}(X,H_{n}^{\varepsilon}).
\end{align*}   \label{normal1}
\end{definition}
\begin{remark}
    It is shown in the work of Li and Zhang \cite[$\S$3.4, $\S$3.5]{LiZhang-orthogonalKR} that for all $n\geq2$ and $\varepsilon,\varepsilon^{\pr}\in\{\pm 1\}$,
    \begin{equation*}
        \textup{Nor}(1,H_{n}^{\varepsilon})=\Den(H_{n}^{\varepsilon},H_{n-1}^{\varepsilon^{\pr}}),\,\,\,\,\textup{Nor}^{\flat}(1,H_{n}^{\varepsilon})=\Den(H_n^{\varepsilon},H_n^{\varepsilon}).
    \end{equation*}\label{nor-local-density}
\end{remark}

\subsection{Counting lattices}
In this subsection, we establish a formula that counts quadratic lattices using local density.

Let $L$ be an almost self-dual quadratic lattice of rank $n+1$. Then
\begin{equation*}
    L^{\vee}\simeq H_n^{\varepsilon}\obot\langle x\rangle,
\end{equation*}
where $\varepsilon\in\{\pm1\}$ and $\nu_\varpi(q(x))=-1$. We say an element $x\in L^{\vee}$ is primitive in $L^{\vee}$ if $x\notin \varpi L^{\vee}$. Let $\textup{O}(L^{\vee})$ be the orthogonal group of the lattice $L^{\vee}$, i.e., $\textup{O}(L^{\vee})=\{g\in\textup{Aut}_{\mathcal{O}_F}(L^{\vee}):(g(x_1),g(x_2))=(x_1,x_2)\,\,\textup{for all }x_1,x_2\in L^{\vee}\}$.
\begin{lemma}
    Let $y=y_1+y_2\in L^{\vee}$ be a primitive element where $y_1\in H_n^{\varepsilon}$ and $y_2\in \langle x\rangle$.
    \begin{itemize}
        \item [(a)]If $y_2$ is a primitive element in $\langle x\rangle$, there exists an element $g\in\textup{O}(L^{\vee})$ such that $g(y)\in \langle x\rangle$.
        \item[(b)] If $y_2$ is not a primitive element in $\langle x\rangle$, there exists an element $g\in\textup{O}(L^{\vee})$ such that $g(y)\in H_n^{\varepsilon}$.
    \end{itemize}    \label{primitive-in-vertex}
\end{lemma}
\begin{proof}
    We first prove (a). The primitivity of $y_2$ implies that $\nu_\varpi(q(y))=-1$. We have the following decomposition
    \begin{equation*}
        L^{\vee}=\langle y\rangle\obot\langle y\rangle^{\bot}.
    \end{equation*}
    By the uniqueness of the Jordan decomposition of quadratic forms \cite[Corollary 5.21]{schulzepillot2021lecture}, there exist two isometric maps $\phi_1:\langle y\rangle\xrightarrow{\sim}\langle x\rangle$ and $\phi_2:\langle y\rangle^{\bot}\xrightarrow{\sim}H_n^{\varepsilon}$. Define $g\in\textup{O}(L^{\vee})$ to be the composition $L^{\vee}=\langle y\rangle^{\bot}\obot\langle y\rangle\xrightarrow{\phi_1\obot\phi_2}H_n^{\varepsilon}\obot \langle x\rangle\rightarrow L^{\vee}$, then $g(y)\in \langle x\rangle$.
    \par
    Now we prove (b). The primitivity of $y_1$ in $H_n^{\varepsilon}$ implies that there exists an element $z\in H_n^{\varepsilon}$ such that $\mathbb{F}_q\cdot\overline{y_1}+\mathbb{F}_q\cdot\overline{z}\subset H_n^{\varepsilon}/\pi H_n^{\varepsilon}$ is a non-degenerate quadratic subspace since $\overline{y_1}\neq0$ and $H_n^{\varepsilon}$ is self-dual. Let $M\coloneqq\mathcal{O}_F\cdot y+\mathcal{O}_F\cdot z\subset L^{\vee}$. The inner product matrix of $M$ under the basis $\{y,z\}$ is
    \begin{equation*}
        \begin{pmatrix}
             q(y_1)+q(y_2) & \frac{1}{2}(y_1,z)\\
            \frac{1}{2}(y_1,z) & q(z)
        \end{pmatrix}=\begin{pmatrix}
            q(y_1) & \frac{1}{2}(y_1,z)\\
            \frac{1}{2}(y_1,z) & q(z)
        \end{pmatrix}+\varpi^{2}\begin{pmatrix}
            q(y_2/\varpi) & 0\\
            0 & 0
        \end{pmatrix}.
    \end{equation*}
    Therefore $M^{\vee}=M$. Notice that for all elements $l\in L^{\vee}$ and $m=ay+bz\in M$, we have $\nu_\varpi\left((m,l)\right)=\nu_{\varpi}(a(y,l)+b(z,l))\geq\min\{\nu_{\varpi}((y_1,l)),\nu_{\varpi}((y_2,l)),\nu_{\varpi}((z,l))\}\geq0$. Then we have a decomposition
    \begin{equation*}
        L^{\vee}=M\obot M^{\bot}.
    \end{equation*}
    Then $M^{\bot}$ is a vertex lattice, therefore there exists a self-dual lattice $M_1$ such that $M^{\bot}\simeq M_1\obot\langle x^{\pr}\rangle$ where $x^{\pr}\in L^{\vee}$ such that $\nu_\varpi(q(x^{\pr}))=-1$. We get $L^{\vee}\simeq M\obot M_1\obot \langle x^{\pr}\rangle$. The uniqueness of Jordan decomposition of quadratic forms implies that there exist two isometric maps $\phi_1: M\obot M_1\xrightarrow{\sim}H_n^{\varepsilon}$ and $\phi_2:\langle x^{\pr}\rangle\xrightarrow{\sim}\langle x\rangle$. Define $g\in\textup{O}(L^{\vee})$ to be the composition $L^{\vee}\simeq M\obot M_1\obot \langle x^{\pr}\rangle\xrightarrow{\phi_1\obot\phi_2}H_n^{\varepsilon}\obot \langle x\rangle \rightarrow L^{\vee}$, then $g(y)\in H_n^{\varepsilon}$.
\end{proof}
We say a sublattice $M\subset L^{\vee}$ is primitive if $M=L^{\vee}\cap \left(M\otimes_{\mathcal{O}_F}F\right)$.
\begin{lemma}
    Let $M\subset L^{\vee}\simeq H_n^{\varepsilon}\obot\langle x\rangle$ be a primitive sublattice of rank $n$. Then
    \begin{equation*}
        M\simeq H_{n}^{\varepsilon}\,\,\textup{or}\,\,M\simeq H_{n-2}^{\varepsilon}\obot\langle y\rangle\obot\langle x\rangle.
    \end{equation*}\label{primi-class}
\end{lemma}
\begin{proof}
    Since the quadratic form on $L^{\vee}$ is nondegenerate. There exists an element $y^{\pr}\in L^{\vee}\otimes_{\mathcal{O}_F}F$ such that 
    \begin{equation*}
        M=\{x\in L^{\vee}:(x,y^{\pr})=0\}.
    \end{equation*}
    We can assume the element $y^{\pr}$ is primitive in $L^{\vee}$. By Lemma \ref{primitive-in-vertex}, we can further assume that $y^{\pr}\in\langle x\rangle$ or $H_n^{\varepsilon}$. If $y^{\pr}\in\langle x\rangle$, we have $M\simeq H_n^{\varepsilon}$. If $y^{\prime}\in H_n^{\varepsilon}$, we have $M\simeq H_{n-2}^{\varepsilon}\obot\langle y\rangle\obot\langle x\rangle$ for some element $y\in L$ such that $q(y)=-q(y^{\prime})$ by \cite[Lemma 7.1.1]{Zhu23}.
\end{proof}
\begin{lemma}\label{lem: pden}
    Let $M\subset L^{\vee}\simeq H_n^{\varepsilon}\obot\langle x\rangle$ be a sublattice of rank $n$. Then
    \begin{equation}
        \textup{Pden}(L^{\vee},M)=q^{-n}\cdot\textup{Nor}^{\flat}(1,H_n^{\varepsilon})\cdot\begin{cases}
            1, &\textup{if $M\simeq H_{n}^{\varepsilon}$ or $M\simeq H_{n-1}^{\varepsilon^{\prime}}\obot\langle x\rangle$;}\\
            2, &\textup{if $M\simeq H_{n-2}^{\varepsilon}\obot\langle y\rangle\obot\langle x\rangle$ where $\nu_\varpi(q(y))\geq1$;}\\
            0, &\textup{otherwise}.
        \end{cases}
    \end{equation}\label{pri-local}
\end{lemma}
\begin{proof}
    By the definition of primitive local density, the number $\textup{Pden}(L^{\vee},M)$ is nonzero if and only if $M\subset L^{\vee}$ is primitive.
    \par
    If $M\simeq H_n^{\varepsilon}$, we have $\textup{Pden}(L^{\vee},M)=\textup{Pden}(H_n^{\varepsilon}\obot\langle x\rangle, H_n^{\varepsilon})$. Since $\nu_\varpi(q(x))=-1$, we have the following isomorphism of sets for all positive integers $N\geq1$
    \begin{equation*}
        \textup{Pden}(H_n^{\varepsilon}\obot\langle x\rangle, H_n^{\varepsilon})(\mathcal{O}_F/\varpi^{N})\xrightarrow{\sim}\left(\varpi\langle x\rangle/\varpi^{N}\langle x\rangle\right)^{n}\times\textup{Pden}(H_n^{\varepsilon}, H_n^{\varepsilon})(\mathcal{O}_F/\varpi^{N}).
    \end{equation*}
    Hence
    \begin{align*}
        \textup{Pden}(H_n^{\varepsilon}\obot\langle x\rangle, H_n^{\varepsilon})&=\lim\limits_{N\rightarrow\infty}\frac{\textup{Pden}(H_n^{\varepsilon}\obot\langle x\rangle, H_n^{\varepsilon})(\mathcal{O}_F/\varpi^{N})}{q^{Nn\left(n+1-(n+1)/2\right)}}\\
        &=q^{-n}\cdot\lim\limits_{N\rightarrow\infty}\frac{\textup{Pden}(H_n^{\varepsilon}, H_n^{\varepsilon})(\mathcal{O}_F/\varpi^{N})}{q^{Nn\left(n-(n+1)/2\right)}}\\
        &=q^{-n}\cdot\textup{Pden}(H_n^{\varepsilon},H_n^{\varepsilon})=q^{-n}\cdot\textup{Nor}^{\flat}(1,H_n^{\varepsilon}).
    \end{align*}
    \par
     If $M\simeq H_{n-1}^{\varepsilon^{\pr}}\obot\langle x\rangle$ for some $\varepsilon^{\pr}\in\{\pm1\}$, we have $\textup{Pden}(L^{\vee},M)=\textup{Pden}(H_n^{\varepsilon}\obot\langle x\rangle, H_{n-1}^{\varepsilon^{\pr}}\obot\langle x\rangle)=\textup{Pden}(H_n^{\varepsilon}\obot\langle x\rangle,\langle x\rangle)\cdot\textup{Pden}(H_n^{\varepsilon},H_{n-1}^{\varepsilon^{\pr}})$. By \cite[Proposition 2.5]{katsurada1999explicit}, we have $\textup{Pden}(H_n^{\varepsilon}\obot\langle x\rangle,\langle x\rangle)=\textup{Pden}(\langle x\rangle,\langle x\rangle)=2q^{-n}$. It's easy to verify from Definition \ref{normal1} and Remark \ref{nor-local-density} that $\textup{Pden}(L^{\vee},M)=q^{-n}\cdot\textup{Nor}^{\flat}(1,H_n^{\varepsilon})$.
     \par
     Now we consider the case $M\simeq H_{n-2}^{\varepsilon}\obot\langle y\rangle\obot\langle x\rangle$ where $y$ is an element such that $\nu_\varpi(q(y))\geq1$.
     \begin{align*}
         \textup{Pden}(L^{\vee},M)&=\textup{Pden}(H_{n}^{\varepsilon}\obot\langle x\rangle, H_{n-2}^{\varepsilon}\obot\langle y\rangle\obot\langle x\rangle)=\textup{Pden}(H_n^{\varepsilon}\obot\langle x\rangle,H_{n-2}^{\varepsilon}\obot\langle x\rangle)\cdot\textup{Pden}(H_2^{+},\langle y\rangle)\\
         &=\textup{Pden}(H_n^{\varepsilon}\obot\langle x\rangle,\langle x\rangle)\cdot\textup{Pden}(H_n^{\varepsilon},H_{n-2}^{\varepsilon})\cdot2(1-q^{-1})\\
         &=4q^{-n}(1-q^{-1})\cdot\textup{Pden}(H_n^{\varepsilon},H_{n-2}^{\varepsilon}).
     \end{align*}
     It's esay to verify that $4q^{-n}(1-q^{-1})\cdot\textup{Pden}(H_n^{\varepsilon},H_{n-2}^{\varepsilon})=2q^{-n}\cdot\textup{Nor}^{\flat}(1,H_n^{\varepsilon})$ by \cite[Definition 3.5.1]{LiZhang-orthogonalKR}.
\end{proof}
\begin{proposition}\label{lattice-counting}
    Let $M\subset L_F$ be a nondegenerate sublattice of rank $n$. Then
    \begin{equation*}
        \Den(L^{\vee},M)=q^{-n}\cdot\textup{Nor}^{\flat}(1,H_n^{\varepsilon})\cdot\sharp\{L^{\pr}\subset L_F:\textup{$L^{\prime}$ is isometric to $L$ and } M\subset L^{\pr\vee}\}.
    \end{equation*}
\end{proposition}
\begin{proof}
    Notice that we have
    \begin{equation*}
        \Den(L^{\vee},M)=\sum\limits_{M\subset M^{\pr}\subset L_F}\textup{Pden}(L^{\vee},M^{\pr}).
    \end{equation*}
    The number $\textup{Pden}(L^{\vee},M^{\pr})$ has been calculated in Lemma \ref{lem: pden}.   It is nonzero if and only of $M^{\pr} \simeq H_{n}^{\varepsilon}\,\,\textup{or}\,\,M\simeq H_{n-2}^{\varepsilon}\obot\langle y\rangle\obot\langle x\rangle$ for some elements $y$ such that $\nu_\varpi(q(y))\geq0$.
    \par
    If $M^{\pr} \simeq H_{n}^{\varepsilon}\,\,\textup{or}\,\,\simeq H_{n-2}^{\varepsilon}\obot\langle y\rangle\obot\langle x\rangle$ for some elements $y$ such that $\nu_\varpi(q(y))=0$, there exists exactly 1 sublattice $L^{\pr}$ such that $L^{\pr}$ is isometric to $L$ and $M^{\pr}\subset L^{\pr\vee}$.
    \par
    If $M^{\pr}\simeq H_{n-2}^{\varepsilon}\obot\langle y\rangle\obot\langle x\rangle$ for some elements $y$ such that $\nu_\varpi(q(y))\geq1$, there exists exactly 2 sublattice $L^{\pr}$ such that $L^{\pr}$ is isometric to $L$ and $M^{\pr}\subset L^{\pr\vee}$.
    \par
    The proposition follows from the above analysis and Lemma \ref{pri-local}.
\end{proof}
\begin{remark}
    We can compare the above formula with the unitary analogue (e.g. \cite[Lemma 5.7]{HLSY}). However, the fact that one still has such formula even when   $\mathrm{rk}(L^\vee)>\mathrm{rk}(M)$ is interesting and a bit surprising. 
\end{remark}

\subsection{Normalized local densities}
Let $L$ be an almost self-dual quadratic lattice of rank $n+1$ over $\mathcal{O}_F$. Then
\begin{equation*}
    L\simeq H_n^{\varepsilon}\obot\langle x\rangle,
\end{equation*}
where $\varepsilon\in\{\pm1\}$ and $\nu_\varpi(q(x))=1$.
\begin{lemma}
    For $\varepsilon,\varepsilon^{\pr}\in\{\pm1\}$, we have
    \begin{equation*}
        \textup{Den}(H_n^{\varepsilon}\obot\langle x\rangle, H_{n-1}^{\varepsilon^{\pr}}\obot\langle x\rangle)=\textup{Nor}^{\flat}(1,H_n^{\varepsilon}).
    \end{equation*}
\end{lemma}
\begin{proof}
    It's easy to see that $\textup{Den}(H_n^{\varepsilon}\obot\langle x\rangle, H_{n-1}^{\varepsilon^{\pr}}\obot\langle x\rangle)=\Den(H_n^{\varepsilon}\obot\langle x\rangle, H_{n-1}^{\varepsilon^{\pr}})\cdot\Den(H_1^{\varepsilon^{\pr\pr}}\obot\langle x\rangle,\langle x\rangle)$ for some $\varepsilon^{\pr\pr}\in\{\pm1\}$ determined by $n, \varepsilon$ and $\varepsilon^{\pr}$. Katsurada's calculations \cite[Proposition 2.5]{katsurada1999explicit} implies that $\Den(H_n^{\varepsilon}\obot\langle x\rangle, H_{n-1}^{\varepsilon^{\pr}})=\Den(H_n^{\varepsilon}, H_{n-1}^{\varepsilon^{\pr}})=\textup{Nor}(1,H_n^{\varepsilon})$. It's also easy to verify that $\Den(H_1^{\varepsilon^{\pr\pr}}\obot\langle x\rangle,\langle x\rangle)=q^{-1}\Den(\langle x\rangle,\langle x\rangle)=q^{-1}\cdot2q=2$. Hence $\textup{Den}(H_n^{\varepsilon}\obot\langle x\rangle, H_{n-1}^{\varepsilon^{\pr}}\obot\langle x\rangle)=2\textup{Nor}(1,H_n^{\varepsilon})=\textup{Nor}^{\flat}(1,H_n^{\varepsilon})$.
\end{proof}
\par
Let $M$ be a quadratic lattice of rank $n$. There exists a polynomials $\textup{Den}^{\varepsilon}(L,M,X),\textup{Den}^{\varepsilon}(L^{\vee},M,X)\in\mathbb{Q}[X]$ such that
\begin{equation*}
    \textup{Den}(L,M,X)\,\big\vert_{X=q^{-k}} = \Den(L\obot H_{2k}^{+},M),\,\,\,\,\textup{Den}(L^{\vee},M,X)\,\big\vert_{X=q^{-k}} = \Den(L^{\vee}\obot H_{2k}^{+},M).
\end{equation*}
\begin{definition}
    Let $L$ be an almost self-dual quadratic lattice of rank $n+1$ such that $L\simeq H_n^{\varepsilon}\obot\langle x\rangle$. Let $M$ be a quadratic lattice of rank $n$. Define the normalized local density
    \begin{equation*}
        \Den^{L}(M)=\frac{\Den(L,M)}{\textup{Nor}^{\flat}(1,H_n^{\varepsilon})},\,\,\,\,\Den^{L^{\vee}}(M)=q^{n}\cdot\frac{\Den(L^{\vee},M)}{\textup{Nor}^{\flat}(1,H_n^{\varepsilon})}.
    \end{equation*}
    Define the derived local density
    \begin{equation*}
        \partial\Den^{L}(M)=-2\cdot\frac{\textup{d}}{\textup{d}X}\bigg\vert_{X=1}\frac{\textup{Den}(L,M,X)}{\textup{Nor}^{\flat}(1,H_n^{\varepsilon})},\,\,\,\,\partial\Den^{L^{\vee}}(M)=-2q^{n}\cdot\frac{\textup{d}}{\textup{d}X}\bigg\vert_{X=1}\frac{\textup{Den}(L^{\vee},M,X)}{\textup{Nor}^{\flat}(1,H_n^{\varepsilon})}.
    \end{equation*}\label{normalized-def}
\end{definition}

\subsection{Whittaker functions}
For a $\mathbb{Q}_p$-vector space $W$ of dimension $m\geq2$, let $\mathscr{S}(W)$ be the Schwartz function space on $W$. Assume that $W$ is equipped with a nondegenerate quadratic form, then for all integers $k\geq0$, let $W_k=W \obot (H_{2k}^+\otimes_{\Of}F)$. Let $\varphi\in\mathscr{S}(W^{m-1})$ be a Schwartz function, and let $\varphi_k=\varphi\otimes1_{(H_{2k}^{+})^{m-1}}\in\mathscr{S}\left(W_k^{m-1}\right)$.
\begin{definition}
   Let  $T$ be an $(m-1)\times (m-1)$ symmetric matrix over $F$. 
   Define
$$
W_T\left(1, k, \varphi\right)
\coloneqq\gamma(W)\cdot\int_{\operatorname{Sym}_{m-1}\left(F\right)} \int_{W_k^{m-1}} \psi(\operatorname{tr} b( q(x)-T)) \varphi_k(x)  \textup{d} x \textup{d} b .
$$
Here $\gamma(W)$ is the Weil index \cite[Proposition A.4]{Kud97Ann}
\begin{equation*}
    \gamma(W)=\gamma_F(\textup{det}(W),\psi)^{1-m}\epsilon(W)^{m-1}\gamma_F(\psi)^{-m(m-1)}.
\end{equation*}
Here we choose the Haar measure $\textup{d} x$ and $\textup{d} b$ to be the product measure of the standard Haar measure on $F$ (i.e., $\operatorname{meas}\left(\mathcal{O}_F\right)=1$ ). Also, $\psi(x)=e^{-2 \pi i \lambda(x)}$ is the 'canonical' additive character of $F$, where $\lambda: F \longrightarrow F/ \mathcal{O}_F \hookrightarrow  F/\mathcal{O}_F$. 
\end{definition}
Let the Schwartz function $\varphi$ be the characteristic function of the lattice $M^{m-1}$ for some lattice $M\subset W$. According to \cite[Proposition A.4, A.6]{Kud97Ann}, for a quadratic lattice $L$ represented by $T$, we have
\begin{align}\label{eq: Whitt=Den}
W_T\left(1, k, \varphi\right)=\gamma(W)\cdot\vert\textup{det}(M)\vert_F^{(m-1)/2}\cdot\operatorname{Den}(  M\obot H_{2k}^{+}, L ).
\end{align}
Notice that when $W$ contains a self-dual lattice, the constant $\gamma(W)=1$.

\subsection{Reduction of Whittaker functions}
Before stating the main theorem of this section, we fix some notations. Let $L$ be a cyclic quadratic lattice of rank $n+1$ over $\Of$ and order $l$. Recall that we have fixed a primitive isometric embedding $\iota: L\rightarrow L^{\sharp}$ where $L^{\sharp}$ is a self-dual quadratic lattice of rank $n+2$ over $\Of$. Let $\Lambda= L^{\sharp}\cap \iota(L)^{\perp}$, it is a rank 1 $\Of$ lattice. Recall that there is a canonical isomorphism  $r:\Lambda^{\vee}/\Lambda\rightarrow L^{\vee}/L$ (cf. Lemma \ref{dual}).
\par
Let $x_0$ be an $\Of$-generator of the lattice $\Lambda$. For any $n\times n$ symmetric matrix $T\in\textup{Sym}_{n}(F)$ and $\boldsymbol{\mu}=(\mu_{1},\mu_{2},\cdots,\mu_{n})\in\left(\Lambda^{\vee}\right)^{n}$, let $T(\boldsymbol{\mu})$ be the following symmetric matrix of size $(n+1)\times(n+1)$,
    \begin{equation*}
        M(T,\boldsymbol{\mu})=\begin{pmatrix}
q(x_{0}) & \begin{matrix}\frac{1}{2}(x_{0},\mu_{1})  &\cdots& \frac{1}{2}(x_{0},\mu_{n})\end{matrix}\\
\begin{matrix}
    \frac{1}{2}(x_{0},\mu_{1})\\ \vdots \\ \frac{1}{2}(x_{0},\mu_{n}) 
\end{matrix} & T+\frac{1}{2}(\mu_{i},\mu_{j})\end{pmatrix}.
    \end{equation*}

\begin{theorem}\label{thm: reduction formula}
    Let $V=L\otimes\mathbb{Q}$. For any $n\times n$ symmetric matrix $T\in\textup{Sym}_{n}(F)$ and $\overline{\boldsymbol{\mu}}=(\overline{\mu_{1}},\cdots,\overline{\mu_{n}})\in (\Lambda^{\vee}/\Lambda)^{n}$. Let $0\leq j\leq [\frac{l}{2}]$ be the largest integer such that $r(\overline{\boldsymbol{\mu}})\coloneqq\left(r(\overline{\mu_1}),\cdots,r(\overline{\mu_n})\right)\in (L(j)^{\vee}/L)^{n}$, we have the following equality for all positive integers $k$,
    \begin{align*}        
    W&_{M(T,\boldsymbol{\mu})}(1,k,1_{(L^{\sharp})^{n+1}}) =\frac{q^{nl/2}}{\gamma(L)}\cdot
    \sum\limits_{i=0}^{j}q^{-(n+2k)i}\cdot\textup{Pden}(L^{\sharp}\obot H_{2k}^{+},\Lambda(i))\cdot W_{T}(1,k,1_{r(\overline{\boldsymbol{\mu}})+L(i)^{n}}).
    \end{align*}
    here for any integer $i$ such that $0\leq i\leq j$, the function $1_{r(\overline{\boldsymbol{\mu}})+L(i)^{n}}$ is the characteristic function of the set $r(\overline{\boldsymbol{\mu}})+L(i)^{n}\subset V^{n}$.
    \label{reduction}
\end{theorem}
The case $\boldsymbol{\mu}=0$ is easy. Let $M$ be a quadratic lattice whose inner product matrix is given by $T$. The formula (\ref{eq: Whitt=Den}) and \cite[Theorem 7.2.1]{Zhu23} give the following equalities for all positive integers $k\geq0$:
\begin{align*}
    W_{M(T,\boldsymbol{0})}(1,k,1_{(L^{\sharp})^{n+1}})&= \Den(L^{\sharp}\obot H_{2k}^{+},M\obot\Lambda)\\&=\sum\limits_{i=0}^{[l/2]}q^{-2ki}\cdot\textup{Pden}(L^{\sharp}\obot H_{2k}^{+},\Lambda(i))\cdot \Den(L(i)\obot H_{2k}^{+},M) \\&=\frac{q^{nl/2}}{\gamma(L)}\cdot
    \sum\limits_{i=0}^{[l/2]}q^{-(n+2k)i}\cdot\textup{Pden}(L^{\sharp}\obot H_{2k}^{+},\Lambda(i))\cdot W_{T}(1,k,1_{L(i)^{n}}).
\end{align*}
The remaining section will be devoted to the proof of Theorem \ref{thm: reduction formula} for $\boldsymbol{\mu}\neq0$. The readers can skip this  on first reading.
\par
We start with a lemma that reduces the theorem to the case that $\mu_2=\cdots =\mu_n=0$. We keep the assumption that $j$ is the largest integer such that $r(\overline{\boldsymbol{\mu}})\coloneqq\left(r(\overline{\mu_1}),\cdots,r(\overline{\mu_n})\right)\in (L(j)^{\vee}/L)^{n}$.
\begin{lemma}\label{lem: red to special mu}
    Let $T\in\textup{Sym}_{n}(F)$ be a $n\times n$ symmetric matrix and $\boldsymbol{\mu}=(\mu_{1},\mu_{2},\cdots,\mu_{n})\in(\Lambda^{\vee})^{n}$, such that $\overline{\mu_{1}}\in \Lambda^{\vee}/\Lambda$ is not $0$, and there exists $c_{i}\in\Of$ such that $\mu_{i}=c_{i}\mu_{1}$ for $2\leq i\leq n$. Let $\boldsymbol{\mu}_{1}=(\mu_{1},0,\cdots,0)$, and $\gamma_{\boldsymbol{\mu}}$ be the following $n\times n$ matrix,
    \begin{equation*}
        \gamma_{\boldsymbol{\mu}}=\begin{pmatrix}
1 & \begin{matrix}-c_{2}  &\cdots& -c_{n}\end{matrix}\\
\begin{matrix}
    0\\ \vdots \\ 0
\end{matrix} & \boldsymbol{1}_{n-1}\end{pmatrix}
    \end{equation*}
    Let $T_{\boldsymbol{\mu}}={^{t}\gamma_{\boldsymbol{\mu}}T\gamma_{\boldsymbol{\mu}}}$, we have the following equalities for any positive integer $k$ and $0\leq i\leq [\frac{l}{2}]$,
    \begin{align}
        W_{M(T,\boldsymbol{\mu})}(1,k,1_{(L^{\sharp})^{n+1}})&=W_{M(T_{\boldsymbol{\mu}},\boldsymbol{\mu}_{1})}(1,k,1_{(L^{\sharp})^{n+1}});\label{3}\\
        W_{T}(1,k,1_{r(\overline{\boldsymbol{\mu}})+L(i)^{n}})&=W_{T_{\boldsymbol{\mu}}}(1,k,1_{r(\overline{\boldsymbol{\mu}_{1}})+L(i)^{n}})\label{4}.
    \end{align}
\end{lemma}
\begin{proof}
    Let $\Tilde{\gamma}_{\boldsymbol{\mu}}$ be the following $(n+1)\times (n+1)$ matrix,
    \begin{equation*}
        \Tilde{\gamma}_{\boldsymbol{\mu}}=\begin{pmatrix}
1 & 0\\
0 & \gamma_{\boldsymbol{\mu}}\end{pmatrix}
    \end{equation*}
    Since $\Tilde{\gamma}_{\boldsymbol{\mu}}$ preserves the Schwartz function $1_{(L^{\sharp})^{n+1}}$, then 
    \begin{equation*}
        W_{M(T,\boldsymbol{\mu})}(1,k,1_{(L^{\sharp})^{n+1}})=W_{{^{t}\Tilde{\gamma}_{\boldsymbol{\mu}}}M(T,\boldsymbol{\mu})\Tilde{\gamma}_{\boldsymbol{\mu}}}(1,k,1_{(L^{\sharp})^{n+1}}).
    \end{equation*}
    then (\ref{3}) follows from the fact that ${^{t}\Tilde{\gamma}_{\boldsymbol{\mu}}}M(T,\boldsymbol{\mu})\Tilde{\gamma}_{\boldsymbol{\mu}}=M(T_{\boldsymbol{\mu}},\boldsymbol{\mu}_{1})$.
    \par
    The formula (\ref{4}) follows from the following identity of Schwartz functions on the space $V^{n}$.
    \begin{equation*}
        1_{r(\overline{\boldsymbol{\mu}})+L(i)^{n}}=\left(1_{r(\overline{\boldsymbol{\mu}}_1)+L(i)^{n}}\right)\circ\gamma_{\boldsymbol{\mu}}.
    \end{equation*}
\end{proof}

Starting from now, we will mainly focus on the case $\boldsymbol{\mu}=(\mu,0,\cdots,0)$.  
\begin{lemma}\label{red1}
      Let $T\in\textup{Sym}_{n}(F)$ be a $n\times n$ symmetric matrix and $\boldsymbol{\mu}=(\mu,0,\cdots,0)\in(\Lambda^{\vee})^{n}$ such that $\overline{\mu}\in\Lambda^{\vee}/\Lambda$ is not $0$. Let $c\in F^{\times}$ be an element such that $\mu=c\cdot x_0$. Let $\Tilde{T}=\textup{diag}\{q(x_{0}),T\}$. For any positive integer $k$, let $L_{k}^{\sharp}=L^{\sharp}\obot H_{2k}^{+}$ be a quadratic lattice and $U_{k}^{\sharp}=L_{k}^{\sharp}\otimes \mathbb{Q}$, let $\phi_c$ be the following Schwartz function on $(U_{k}^{\sharp})^{n+1}$,
      \begin{equation*}
          \phi_c\left((t_{0},t_{1},\cdots,t_{n})\right)= 1_{L^{\sharp}_{k}}(t_{0})1_{L^{\sharp}_{k}}(c\cdot t_0+t_{1})1_{L^{\sharp}_{k}}(t_{2})\cdots1_{L^{\sharp}_{k}}(t_{n}).
      \end{equation*}
      Then we have
      \begin{align*}
          W_{M(T,\boldsymbol{\mu})}(1,k,1_{(L^{\sharp})^{n+1}})
          =W_{\tilde{T}}(1,k,\phi_c),
      \end{align*}
      where $\boldsymbol{t}=(t_{0},t_{1},\cdots,t_{n})\in (U_{k}^{\sharp})^{n+1}$.
\end{lemma}
\begin{proof}
    Since $\overline{\mu}_{1}=\Lambda^{\vee}/\Lambda$ is not $0$,. Let $\gamma$ be the following matrix of size $(n+1)\times(n+1)$,
    \begin{equation*}
        \gamma=\begin{pmatrix}
1 & \begin{matrix}-c & 0 &\cdots& 0\end{matrix}\\
\begin{matrix}
    0\\ \vdots \\ 0
\end{matrix} & \boldsymbol{1}_{n}\end{pmatrix},
    \end{equation*}
    then we have ${^{t}\gamma}M(T,\boldsymbol{\mu})\gamma=\Tilde{T}$. By definition,
    \begin{align*}
          W_{M(T,\boldsymbol{\mu})}(1,k,1_{(L^{\sharp})^{n+1}})=\gamma(L^{\sharp})\cdot
          \int\limits_{\textup{Sym}_{n+1}(F)}\int\limits_{(U_{k}^{\sharp})^{n+1}}\psi\left(\textup{tr}\left((\frac{1}{2}(\Tilde{\boldsymbol{t}},\Tilde{\boldsymbol{t}})-M(T,\boldsymbol{\mu}))b\right)\right)\cdot
          1_{(L^{\sharp}_{k})^{n+1}}(\Tilde{\boldsymbol{t}})\textup{d}\Tilde{\boldsymbol{t}}\textup{d}b,
      \end{align*}
      while by making the change of variables $\boldsymbol{t}=\Tilde{\boldsymbol{t}}\cdot\gamma$,
    \begin{align*}
          W_{M(T,\boldsymbol{\mu})}&(1,k,1_{(L^{\sharp})^{n+1}})=\int\limits_{\textup{Sym}_{n+1}(F)}\int\limits_{(U_{k}^{\sharp})^{n+1}}\psi\left(\textup{tr}\left({^{t}\gamma^{-1}}(\frac{1}{2}(\boldsymbol{t},\boldsymbol{t})-\Tilde{T})\gamma^{-1}b\right)\right)\cdot
          \phi_c(\boldsymbol{t})\textup{d}\boldsymbol{t}\textup{d}b\\
          &\stackrel{b\mapsto\gamma^{-1}\cdot b\cdot{^{t}\gamma^{-1}}}=\int\limits_{\textup{Sym}_{n+1}(F)}\int\limits_{(U_{k}^{\sharp})^{n+1}}\psi\left(\textup{tr}\left((\frac{1}{2}(\boldsymbol{t},\boldsymbol{t})-\Tilde{T})b\right)\right)\cdot
          \phi_c(\boldsymbol{t})\textup{d}\boldsymbol{t}\textup{d}b=W_{\Tilde{T}}(1,k,\phi_c).\\
      \end{align*}
\end{proof}
To proceed, we need to relate the Whittaker function with the classical representation densities. Recall that the lattice $L$ is a cyclic quadratic lattice of order $l\geq1$, and we have the following filtration as (\ref{intermediate}),
\begin{equation*}
    L=L(0)\subset L(1)\subset L(2)\subset\cdots\subset L\left([\frac{l}{2}]\right)\subset L\left([\frac{l}{2}]\right)^{\vee}\subset L(1)^{\vee}\subset L(0)^{\vee}=L^{\vee}.
\end{equation*}
\begin{definition}\label{rep-sets}
    Let $T\in\textup{Sym}_{n}(F)$ be a $n\times n$ symmetric matrix and $\boldsymbol{\mu}=(\mu,0,\cdots,0)\in(\Lambda^{\vee})^{n}$ such that $\overline{\mu}\in\Lambda^{\vee}/\Lambda$ is not $0$. Let $c\in F^{\times}$ be an element such that $\mu=c\cdot x_0$. Let $\Tilde{T}=\textup{diag}\{q(x_{0}),T\}$. For any positive integer $k$, let $L_{k}^{\sharp}=L^{\sharp}\obot H_{2k}^{+}$ be a quadratic lattice and $U_{k}^{\sharp}=L_{k}^{\sharp}\otimes F$. For any integer $0\leq i\leq [\frac{l}{2}]$, let $L(i)_{k}=L(i)\obot H_{2k}^{+}$. For any positive integer $t$, let $\RepL$ and $\Repl$ be the following sets,
    \begin{align*}
        \textup{Rep}_{\Tilde{T}}^{\boldsymbol{\mu}}(\Of/\varpi^{t})=\{(\overline{t}_0,\overline{t}_{1},(\overline{t}_2,\cdots,\overline{t}_n)):\,\, &\overline{t}_0\in L_{k}^{\sharp}/\varpi^{t+2l}L_{k}^{\sharp}, \overline{t}_{1}\in \varpi^{-l}L_{k}^{\sharp}/\varpi^{t+l}L_{k}^{\sharp},\\
        &(\overline{t}_2,\cdots,\overline{t}_n)\in (L_{k}^{\sharp}/\varpi^{t+l}L_{k}^{\sharp})^{n-1}, \,\,\textup{such that}\\
        &q_{L_{k}^{\sharp}}(\overline{t}_{0})\equiv q_{\Lambda}(x_0)\,\,\textup{mod}\,\varpi^{t}; \frac{1}{2}(\overline{t}_i,\overline{t}_0)\equiv0\,\,\textup{mod}\,\varpi^{t};\\ &\frac{1}{2}(\overline{t}_i,\overline{t}_j)_{i,j\geq1}\equiv T\,\,\textup{mod}\,\varpi^{t}; c\cdot\overline{t}_0+\overline{t}_1\in L^{\sharp}_{k}/\varpi^{t+l}L^{\sharp}_{k}.\}.
    \end{align*}
    \begin{align*}
        \Repl=\{(\overline{t}_{1},\overline{t}_2,\cdots,\overline{t}_n):\,\, &\overline{t}_1\in \varpi^{-l}L(i)_{k}/\varpi^{t+l}L(i)_{k}, (\overline{t}_2,\cdots,\overline{t}_n)\in (L(i)_{k}/\varpi^{t+l}L(i)_{k})^{n-1},\\
        &\textup{such that}\,\,\frac{1}{2}(\overline{t}_i,\overline{t}_j)_{i,j\geq1}\equiv T\,\,\textup{mod}\,\varpi^{t}; t_1\in r(\overline{\mu})+L(i)_{k}\\&\,\,\textup{for an arbitrary lift $t_{1}\in \varpi^{-l}L(i)_{k}$ of $\overline{t}_1$}.\}.
    \end{align*}
\end{definition}
\begin{proposition}
    Let $T\in\textup{Sym}_{n}(F)$ be a $n\times n$ symmetric matrix and $\boldsymbol{\mu}=(\mu,0,\cdots,0)\in(\Lambda^{\vee})^{n}$ such that $\overline{\mu}\in\Lambda^{\vee}/\Lambda$ is not $0$. Let $\Tilde{T}=\textup{diag}\{q(x_{0}),T\}$, we have
    \begin{equation*}
        W_{M(T,\boldsymbol{\mu})}(1,k,1_{(L^{\sharp})^{n+1}})= \lim\limits_{t\rightarrow+\infty}\frac{\sharp \RepL}{q^{(n+2+2k)\left((n+2)l+(n+1)t\right)-t\cdot (n+1)(n+2)/2}}.
    \end{equation*}
    \begin{equation*}
        W_{T}(1,k,1_{(r(\overline{\mu})+L(i))\times L(i)\times\cdots\times L(i)})=\gamma(L)\cdot \lim\limits_{t\rightarrow+\infty}\frac{\sharp \Repl}{q^{\left((n+1+2k)(l+t)+l/2-i\right)n-t\cdot n(n+1)/2}}.
    \end{equation*}
    \label{rep}
\end{proposition}
\begin{proof}
    Proposition \ref{red1} implies that
    \begin{align*}
          &W_{M(T,\boldsymbol{\mu})}(1,k,1_{(L^{\sharp})^{n+1}})=\\
          &\int\limits_{\textup{Sym}_{n+1}(F)}\int\limits_{(U_{k}^{\sharp})^{n+1}}\psi\left(\textup{tr}\left((\frac{1}{2}(\boldsymbol{t},\boldsymbol{t})-\Tilde{T})b\right)\right)\cdot
          1_{L^{\sharp}_{k}}(t_{0})1_{L^{\sharp}_{k}}(c\cdot t_0+t_{1})1_{L^{\sharp}_{k}}(t_{2})\cdots1_{L^{\sharp}_{k}}(t_{n})\textup{d}t_{0}\cdots\textup{d}t_{n}\textup{d}b.
          \end{align*}
    Note that for any positive integer $t$, if $\boldsymbol{t}=(t_0,\cdots,t_n)$ satisfies that
    \begin{align*}
        t_{0}, t_{i}\in L_k^{\sharp}\,\,\textup{for}\,\,i\geq 2,\,\,c\cdot t_0+t_1\in L_k^{\sharp}\,\,\,\textup{and}\,\,\frac{1}{2}(\boldsymbol{t},\boldsymbol{t})\equiv \Tilde{T}\,\,\textup{mod}\,\,\varpi^{t},
    \end{align*}
    it's easy to check that any element in the coset $\boldsymbol{t}+\varpi^{l+t}\cdot(\varpi^{l}L_{k}^{\sharp}\times(L_k^{\sharp})^{n})$ also satisfies the above conditions. Therefore,
     \begin{align*}
          W_{M(T,\boldsymbol{\mu})}(1,k,1_{(L^{\sharp})^{n+1}})&=\lim\limits_{t\rightarrow+\infty}q^{t\cdot(n+1)(n+2)/2}\cdot\sharp \RepL\cdot\textup{vol}(\varpi^{l+t}\cdot(\varpi^{l}L_{k}^{\sharp}\times(L_k^{\sharp})^{n}))\\
          &=\lim\limits_{t\rightarrow+\infty}\frac{\sharp \RepL}{q^{(n+2+2k)\left((n+2)l+(n+1)t\right)-t\cdot (n+1)(n+2)/2}}.
          \end{align*}
    \par
    The integral representation of $ W_{T}(1,k,1_{(r(\overline{\mu})+L(i))\times L(i)\times\cdots\times L(i)})$ takes the following form,
    \begin{align}
        &W_{T}(1,k,1_{(r(\overline{\mu})+L(i))\times L(i)\times\cdots\times L(i)})\notag\\
        &=\gamma(L)\cdot\int\limits_{\textup{Sym}_{n}(F)}\int\limits_{(U_{k})^{n}}\psi\left(\textup{tr}\left((\frac{1}{2}(\boldsymbol{t},\boldsymbol{t})-T)b\right)\right)\cdot
          1_{r(\overline{\mu})+L(i)_{k}}(t_{1})1_{L(i)_{k}}(t_{2})\cdots1_{L(i)_{k}}(t_{n})\textup{d}t_{1}\cdots\textup{d}t_{n}\textup{d}b.
    \end{align}
    Note that for any positive integer $t$, if $\boldsymbol{t}=(t_1,\cdots,t_n)$ satisfies that
    \begin{align*}
        t_{1}\in r(\overline{\mu})+L(i)_{k}, t_{i}\in L(i)_{k}\,\,\textup{for}\,\,i\geq 2\,\,\,\textup{and}\,\,\frac{1}{2}(\boldsymbol{t},\boldsymbol{t})\equiv T\,\,\textup{mod}\,\,\varpi^{t},
    \end{align*}
    it's easy to check that any element in the coset $\boldsymbol{t}+\varpi^{l+t}\cdot(L(i)_{k})^{n}$ also satisfies the above conditions. Therefore,
    \begin{align*}
         W_{T}(1,k,1_{(r(\overline{\mu})+L(i))\times L(i)\times\cdots\times L(i)})&=\gamma(L)\lim\limits_{t\rightarrow+\infty}q^{t\cdot n(n+1)/2}\cdot\Repl\cdot\textup{vol}(\varpi^{t+l}\cdot(L(i)_{k})^{n}).\notag\\         
         &=\gamma(L)\cdot \lim\limits_{t\rightarrow+\infty}\frac{\sharp \Repl}{q^{\left((n+1+2k)(l+t)+l/2-i\right)n-t\cdot n(n+1)/2}}.
    \end{align*}
\end{proof}
Now we are ready to prove Theorem \ref{reduction} assuming several technical lemmas about representation sets that will be proved in \S \ref{sec: rep sets}. 
 \begin{proof}[Proof of Theorem \ref{reduction}] By Lemma \ref{lem: red to special mu},  we only need to consider the case that $\mu_{2}=\cdots=\mu_{n}=0$, and $\mu_1\in\Lambda^{\vee}/\Lambda$ is not 0. In $\S\ref{sec: rep sets}$, we will define a restriction map (cf. (\ref{resmap}))
\begin{equation*}
    \textup{res}:\RepL\longrightarrow\textup{Rep}_{L_{k}^{\sharp},\Lambda}^{t}(\Of/\varpi^{2l+t})\stackrel{\simeq}\rightarrow\bigsqcup\limits_{i=0}^{[l/2]}\textup{PRep}_{L_{k}^{\sharp},\Lambda}^{t-2i}(\Of/\varpi^{2l+t-i}).
\end{equation*}
The sets $\textup{Rep}_{L_{k}^{\sharp},\Lambda}^{t}(\Of/\varpi^{2l+t})$, $\textup{PRep}_{L_{k}^{\sharp},\Lambda}^{t-2i}(\Of/\varpi^{2l+t-i})$ will be defined in the beginning of $\S\ref{sec: rep sets}$. Therefore we have
\begin{align*}
    \sharp \RepL&=\sum\limits_{i=0}^{[l/2]}\sum\limits_{x\in\textup{PRep}_{L_{k}^{\sharp},\Lambda}^{t-2i}(\Of/\varpi^{2l+t-i})}\sharp \textup{res}^{-1}(x)\\
    &=\sum\limits_{i=0}^{j}\sharp \textup{PRep}_{L_{k}^{\sharp},\Lambda}^{t-2i}(\Of/\varpi^{2l+t-i})\cdot q^{n(l+i)}\cdot\sharp \Repl.
\end{align*}
The last equality follows from Corollary \ref{fiber-size}. By Proposition \ref{rep}, we have
\begin{align*}
    &W_{M(T,\boldsymbol{\mu})}(1,k,1_{(L^{\sharp})^{n+1}})=\lim\limits_{t\rightarrow+\infty}\frac{\sharp \RepL}{q^{(n+2+2k)\left((n+2)l+(n+1)t\right)-t\cdot (n+1)(n+2)/2}}.\\
    &=\lim\limits_{t\rightarrow+\infty}\sum\limits_{i=0}^{j}q^{nl/2-(n+2k)i}\cdot\frac{\sharp \textup{PRep}_{L_{k}^{\sharp},\Lambda}^{t-2i}(\Of/\varpi^{2l+t-i})}{q^{(2k+n+1)(2l+t-i)}}\cdot\frac{\sharp \Repl}{q^{\left((n+1+2k)(l+t)+l/2-i)n-t\cdot n(n+1)/2\right)}}\\
    &=\frac{q^{nl/2}}{\gamma(L)}\cdot
    \sum\limits_{i=0}^{j}q^{-(n+2k)i}\cdot\textup{Pden}(L^{\sharp}\obot H_{2k}^{+},\Lambda(i))\cdot W_{T}(1,k,1_{(r(\overline{\mu})+L(i))\times L(i)\times\cdots\times L(i))}).
\end{align*}
\end{proof}

\subsection{Relation between the representation sets}\label{sec: rep sets}

Our next job is to investigate the relationship between the sets $\Repl$ and $\RepL$. For any positive integers $k, t\geq 0$, we consider the following sets,
\begin{equation*}
    \textup{Rep}_{L_{k}^{\sharp},\Lambda}^{t}(\Of/\varpi^{2l+t})=\{x\in L_k^{\sharp}/\varpi^{2l+t}L_k^{\sharp}:q_{L_k^{\sharp}}(x)\equiv q_{\Lambda}(x_0)\,\,\textup{mod}\,\,\varpi^{t}\}.
\end{equation*}
\begin{equation*}
    \textup{PRep}_{L_{k}^{\sharp},\Lambda}^{t-2i}(\Of/\varpi^{2l+t-i})=\{x\in L_k^{\sharp}/\varpi^{2l+t-i}L_k^{\sharp}:q_{L_k^{\sharp}}(x)\equiv \varpi^{-2i}q_{\Lambda}(x_0)\,\,\textup{mod}\,\,\varpi^{t-2i}\},\,\,0\leq i\leq [\frac{l}{2}].
\end{equation*}
\par
For sufficiently large integer $t$, we have a natural decomposition map 
\begin{equation*}
    D:\textup{Rep}_{L_{k}^{\sharp},\Lambda}^{t}(\Of/\varpi^{2l+t})\stackrel{\simeq}\rightarrow\bigsqcup\limits_{i=0}^{[l/2]}\textup{PRep}_{L_{k}^{\sharp},\Lambda}^{t-2i}(\Of/\varpi^{2l+t-i})
\end{equation*}
given as follows: for an element $x\in\textup{Rep}_{L_{k}^{\sharp},\Lambda}^{t}(\Of/\varpi^{2l+t})$, there exists a unique integer $0\leq i\leq [\frac{l}{2}]$ such that $x\in\varpi^{i}L_{k}^{\sharp}/\varpi^{2l+t}L_k^{\sharp}$ but $x\notin\varpi^{i+1}L_{k}^{\sharp}/\varpi^{2l+t}L_k^{\sharp}$, then it's easy to check that $\varpi^{-i}x\in\textup{PRep}_{L_{k}^{\sharp},\Lambda}^{t-2i}(\Of/\varpi^{2l+t-i})$.
\par
By Definition \ref{rep-sets} of the set $\RepL$, there is a restriction map
\begin{align}
    \textup{res}:\RepL&\longrightarrow\textup{Rep}_{L_{k}^{\sharp},\Lambda}^{t}(\Of/\varpi^{2l+t});\label{resmap}\\
    (\overline{t}_0,\overline{t}_{1},(\overline{t}_2,\cdots,\overline{t}_n))&\longmapsto \overline{t}_0\in L_k^{\sharp}/\varpi^{2l+t}L_k^{\sharp}.\notag
\end{align}
Let $x\in\textup{PRep}_{L_{k}^{\sharp},\Lambda}^{t-2i}(\Of/\varpi^{2l+t-i})$, we are interested in the set $(D\circ \textup{res})^{-1}(x)\subset\RepL$, it can be described in the following way,
\begin{align}
    (D\circ \textup{res})^{-1}(x)=\{(\overline{t}_1,\cdots,\overline{t}_n):\,\,&\overline{t}_1\in\varpi^{-l}L_k^{\sharp}/\varpi^{l+t}L_k^{\sharp},\,\,\overline{t}_m\in L_k^{\sharp}/\varpi^{l+t}L_k^{\sharp}\,\,\textup{for}\,\,1\leq m\leq n;\label{description}\\
    &\frac{1}{2}(\varpi^{i}x,\overline{t}_m)\equiv 0\,\,\textup{mod}\,\,\varpi^{t};\,\,\frac{1}{2}(\overline{t}_m,\overline{t}_{m^{\pr}})\equiv T\,\,\textup{mod}\,\,\varpi^{t};\notag\\
    & c\cdot\varpi^{i}x+\overline{t}_1\in L_k^{\sharp}/\varpi^{l+t}L_k^{\sharp}\}.\notag
\end{align}
\par
The element $x\in\textup{PRep}_{L_{k}^{\sharp},\Lambda}^{t-2i}(\Of/\varpi^{2l+t-i})$ can be lifted to $\Of$ and then generates a quadratic lattice isometric to $\Lambda(i)$, then we get the following two exact sequences by the diagram (\ref{compatibility}), 
\begin{equation}
    0\longrightarrow\Lambda(i)\obot\varpi^{-l}L(i)_{k}\stackrel{i_x}\longrightarrow \varpi^{-l}L_k^{\sharp}\longrightarrow Q_0\longrightarrow0,
\end{equation}
\begin{equation}
    0\longrightarrow\Lambda(i)\obot L(i)_{k}\stackrel{j_x}\longrightarrow  L_k^{\sharp}\longrightarrow Q_1\longrightarrow0.
\end{equation}
Note that $Q_0$ and $Q_1$ are finite. For positive integer $t$ large enough, the above exact sequences induce
\begin{equation}
    0\longrightarrow T_0\longrightarrow\Lambda(i)/\varpi^{2l+t}\Lambda(i)\obot\varpi^{-l}L(i)_{k}/\varpi^{l+t}L(i)_{k}\stackrel{\overline{i}_x}\longrightarrow \varpi^{-l}L_k^{\sharp}/\varpi^{l+t}L_k^{\sharp}\longrightarrow Q_0\longrightarrow 0,
    \label{key1}
\end{equation}
\begin{equation}
    0\longrightarrow T_1\longrightarrow\Lambda(i)/\varpi^{l+t}\Lambda(i)\obot L(i)_{k}/\varpi^{l+t}L(i)_{k}\stackrel{\overline{j}_x}\longrightarrow L_k^{\sharp}/\varpi^{l+t}L_k^{\sharp}\longrightarrow Q_1\longrightarrow 0.
    \label{key2}
\end{equation}
\begin{lemma}\label{bigone}
    For any $x\in\textup{PRep}_{L_{k}^{\sharp},\Lambda}^{t-2i}(\Of/\varpi^{2l+t-i})$ and large enough integer $t$, we have 
    \begin{equation*}
        (D\circ \textup{res})^{-1}(x)\subset\textup{Im}\left(\overline{i}_x\times\overline{j}_x^{n-1}\right).
    \end{equation*}
\end{lemma}
\begin{proof}
    For an element $(\overline{t}_1,\cdots,\overline{t}_n)\in(D\circ \textup{res})^{-1}(x)$, let $t_1\in\varpi^{-l}L_k^{\sharp}$, $t_m\in L_k^{\sharp}$ ($m\geq2$) and $\Tilde{x}$ be any lift of $\overline{t}_1$, $\overline{t}_m$ ($m\geq2$) and $x$. By \cite[Theorem 5.2]{Mor79}, there is an isometric map $ L_k^{\sharp}\cap\{\Tilde{x}\}^{\perp}\rightarrow L(i)_k$.
    \par
    We have the following
    \begin{equation*}
        t_m=\frac{(t_m,\Tilde{x})}{2q(\Tilde{x})}\Tilde{x}+t_m^{\pr}\,\,\textup{for $m\geq1$}.
    \end{equation*}
    For $m\geq2$, every element $t_m^{\pr}\in L_k^{\sharp}\cap\{\Tilde{x}\}^{\perp}\simeq L(i)_k$. For $m=1$, $t_1^{\pr}\in \varpi^{-l}L_k^{\sharp}\cap\{\Tilde{x}\}^{\perp}\simeq \varpi^{-l}L(i)_k$, therefore 
    \begin{equation*}
        \overline{t}_1=\overline{i}_x\left(\overline{\frac{(t_1,\Tilde{x})}{2q(\Tilde{x})}\Tilde{x}}+\overline{t_1^{\pr}}\right),\,\,\overline{t}_m=\overline{j}_x\left(\overline{\frac{(t_m,\Tilde{x})}{2q(\Tilde{x})}\Tilde{x}}+\overline{t_m^{\pr}}\right)\,\,\textup{for $m\geq2$}.
    \end{equation*}
    Hence $(\overline{t}_1,\cdots,\overline{t}_n)\in\textup{Im}\left(\overline{i}_x\times\overline{j}_x^{n-1}\right)$, therefore we can conclude that $(D\circ \textup{res})^{-1}(x)\subset\textup{Im}\left(\overline{i}_x\times\overline{j}_x^{n-1}\right)$.
\end{proof}
\begin{lemma}\label{kernel}
    For any $x\in\textup{PRep}_{L_{k}^{\sharp},\Lambda}^{t-2i}(\Of/\varpi^{2l+t-i})$, we have
    \begin{equation*}
        \sharp  T_0=\sharp Q_0= q^{2l-2i},\,\,\sharp  T_1=\sharp Q_1= q^{l-2i},
    \end{equation*}
\end{lemma}
\begin{proof}
    Since the middle two terms in the exact sequences (\ref{key1}) and (\ref{key2}) have the same cardinality, we have $\sharp  T_0=\sharp Q_0$ and $\sharp  T_1=\sharp Q_1$. Note that
    \begin{equation*}
        Q_0\simeq\varpi^{-l}L_{k}^{\sharp}/\left(\Lambda(i)\obot\varpi^{-l}L(i)_k\right)\simeq L_{k}^{\sharp}/\left(\varpi^{l}\Lambda(i)\obot L(i)_k\right)\,\,\Longrightarrow\,\,\sharp Q_0=\vert\textup{det}(\varpi^{l}\Lambda(i)\obot L(i)_k)\vert_{F}^{-\frac{1}{2}}=q^{2l-2i}.
    \end{equation*}
    \begin{equation*}
        Q_1\simeq L_{k}^{\sharp}/\left(\Lambda(i)\obot L(i)_k\right)\,\,\Longrightarrow\,\,\sharp Q_1=\vert\textup{det}(\Lambda(i)\obot L(i)_k)\vert_{F}^{-\frac{1}{2}}=q^{l-2i}.
    \end{equation*}
\end{proof}
\begin{corollary}\label{fiber-size}
    Let $0\leq j\leq [\frac{l}{2}]$ be the largest integer such that $r(\overline{\mu})\in L(j)^{\vee}/L$. For any $x\in\textup{PRep}_{L_{k}^{\sharp},\Lambda}^{t-2i}(\Of/\varpi^{2l+t-i})$ and large enough integer $t$, we have
    \begin{equation*}
        \sharp \textup{res}^{-1}(x)=\begin{cases}
            q^{n(l+i)}\cdot\sharp \Repl, & \textup{if $0\leq i\leq j$};\\
            0, &\textup{if $j< i\leq [\frac{l}{2}]$}.
        \end{cases}.
    \end{equation*}
\end{corollary}
\begin{proof}
    Let $\Tilde{x}\in L_k^{\sharp}$ be a lift of $x$ such that $\mathcal{O}_F\cdot x\simeq\Lambda(i)$. Let $\left((\overline{a_1\Tilde{x}},\overline{t_{1}}),(\overline{a_2\Tilde{x}},\overline{t_{2}}),\cdots,(\overline{a_n\Tilde{x}},\overline{t_{n}})\right)\in \Lambda(i)/\varpi^{2l+t}\Lambda(i)\obot\varpi^{-l}L(i)_k/\varpi^{l+t}L(i)_k\times \left(\Lambda(i)/\varpi^{l+t}\Lambda(i)\obot L(i)_k/\varpi^{l+t}L(i)_k\right)^{n-1}$ be an element such that its image under $\overline{i}_x\times\overline{j}_x^{n-1}$ lies in the set $\textup{res}^{-1}(x)$. By the description in (\ref{description}), we have
    \begin{equation*}
        a_1\in (\varpi^{t+i-l})/(\varpi^{2l+t}),\,\,a_m\in(\varpi^{t+i-l})/(\varpi^{l+t})\,\,\textup{for $m\geq 2$};
    \end{equation*}
    \begin{equation*}
        \frac{1}{2}(\overline{t_m},\overline{t_{m^{\pr}}})_{m,m^{\pr\geq1}}\equiv T\,\,\textup{mod}\,\,\varpi^{t}\,\,\textup{and}\,\,c\cdot\varpi^{i}x+\overline{t_1}\in L_k^{\sharp}/\varpi^{l+t}L_k^{\sharp}.
    \end{equation*}
    \par
    We first consider the case that $j<i\leq[\frac{l}{2}]$. Recall that $x_0$ is a generator of the lattice $\Lambda$, and $\mu=c\cdot x_0\in L^{\vee}$, we have $\nu_F(c)=-l+j$. Then we have $(\Tilde{x},c\cdot\varpi^{i}\Tilde{x})\notin\Of$ because $\nu_F\left((\Tilde{x},c\cdot\varpi^{i}\Tilde{x})\right)=j-i<0$. We can conclude that there doesn't exist $\overline{t_1}\in\varpi^{-l}L(i)_k/\varpi^{l+t}L(i)_k$ such that $c\cdot\varpi^{i}x+\overline{t_1}\in L_k^{\sharp}/\varpi^{l+t}L_k^{\sharp}$ because otherwise we would get $(\Tilde{x},c\cdot\varpi^{i}\Tilde{x})\in\Of$ which is a contradiction, therefore $\textup{res}^{-1}(x)=\varnothing$.
    \par
    Next we consider the case that $0\leq i\leq j$. Recall the following definition of the set $\Repl$ in Definition \ref{rep-sets}:
    \begin{align*}
        \Repl=\{(\overline{t}_{1},\overline{t}_2,\cdots,\overline{t}_n):\,\, &\overline{t}_1\in \varpi^{-l}L(i)_k/\varpi^{t+l}L(i)_k, (\overline{t}_2,\cdots,\overline{t}_n)\in (L(i)_k/\varpi^{t+l}L(i)_k)^{n-1},\\
        &\textup{such that}\,\,\frac{1}{2}(\overline{t}_i,\overline{t}_j)_{i,j\geq1}\equiv T\,\,\textup{mod}\,\varpi^{t}; t_1\in r(\overline{\mu})+L(i)_k\\&\,\,\textup{for an arbitrary lift $t_{1}\in \varpi^{-l}L(i)_k$ of $\overline{t}_1$}.\}.
    \end{align*}
    By Lemma \ref{compatible}, the condition $t_1\in r(\overline{\mu})+L(i)_k$ for an arbitrary lift $t_{1}\in \varpi^{-l}L(i)_k$ of $\overline{t}_1$ is equivalent to saying that $\mu+t_1\in L_k^{\sharp}$, i.e., $\overline{c\cdot\varpi^{i}x}+\overline{t_1}\in L_k^{\sharp}/\varpi^{l+t}L_k^{\sharp}$, hence 
    \begin{equation*}
        \textup{res}^{-1}(x)\simeq \left((\varpi^{t+i-l})/(\varpi^{2l+t})\times \left((\varpi^{t+i-l})/(\varpi^{l+t})\right)^{n-1}\times\Repl\right)/\,\,\textup{ker}(\overline{i}_x\times\overline{j}_x^{n-1}).
    \end{equation*}
    Therefore $\sharp \textup{res}^{-1}(x)=q^{n(l+i)}\cdot\sharp \Repl$ by Lemma \ref{kernel}.
\end{proof}

\part{Geometric side}
\section{Special cycles on GSpin Rapoport--Zink spaces}

In this section, we introduce a family of special Kudla--Rapoport cycles on GSpin Rapoport--Zink spaces for an almost-self dual quadratic lattice.

The work of Hamacher-Kim \cite[Section 4]{HodgeRZparahoric} constructs Rapoport--Zink spaces of Hodge type with parahoric levels, generalizing the work of Howard--Pappas \cite{HP17} and Kim \cite{Kim18} in the hyperspecial case. Since the above works only deal with local Shimura datums of Hodge type over $\Qp$, we take $p$-adic field $F=\Qp$. Let $\breve{F}$ be the completion of a maximal unramified extension of $F$.  Let $\BF$ be an algebraic closure of the residue field $\BF_p$. Let $\sigma$ be the non-trivial Frobenius involution on $\breve{F}$.

\subsection{Rapoport--Zink space associated with a self-dual lattice}
Let $L^{\sharp}$ be a self-dual quadratic lattice of rank $n+2$ over $\mathcal{O}_F$. Let $C^\sharp =C(L^\sharp )$ be the Clifford algebra of $L^\sharp $, and $D^\sharp =\Hom_{O_F}(C^\sharp , O_F)$ be the linear dual of $C^\sharp $. We follow the construction of Rapoport--Zink spaces in \cite{HP17}, see also \cite[Section 4]{LiZhang-orthogonalKR}. Consider a local unramified Shimura-Hodge datum $(G^\sharp ,  b^\sharp , \mu^\sharp , C^\sharp )$ for $L^\sharp $ as in \cite[Proposition 4.2.6]{HP17}, where $G^\sharp =\GSpin(V^\sharp )$, $b^\sharp \in G(\breve{F})$ is a basic element and $\mu: \BG_m \to G^\sharp $ is a certain cocharacter compatible with $b^\sharp$. As in \cite[Leamma 2.2.5]{HP17}, the local unramified Shimura-Hodge datum $(G^\sharp ,  b^\sharp , \mu^\sharp , C^\sharp )$ gives rise to a principally polarized $p$-divisible $\BX^\sharp $ over $\BF$ such that there is an isomorphism
\begin{equation}
    \BD(\BX^\sharp )(O_{\breve{F}})\simeq D^\sharp\otimes_{O_{F}}{O_{\breve{F}}},
    \label{dieu}
\end{equation}
where $\BD(\BX^\sharp )(O_{\breve{F}})$ is the contravariant Dieudonne module of $\BX^{\sharp}$. The Frobenius morphism $\mathbf{F}$ on $ \BD(\BX^\sharp )(O_{\breve{F}})$ is given by $b^\sharp  \circ \sigma$ under the isomorphism (\ref{dieu}). Let $\lambda_0^{\sharp}:\BX^{\sharp}\rightarrow\BX^{\sharp,\vee}$ be the principal polarization.

Let $(s_{\alpha})_{\alpha \in I}$ be a collection of tensors $s_\alpha$ in the total tensor algebra $(C^\sharp )^\otimes$ such that $G^\sharp $ is the stabilizer of $(s_\alpha)$ inside $\GL(C^\sharp )$. Each tensor $s_{\alpha}$ gives rise to a crystalline Tate tensor $t_{\alpha,0}= s_{\alpha} \otimes 1 : \mathbf{1}:= \BD(F/O_F) \to \BD(\BX^\sharp )^\otimes$ on $\BX^\sharp $. Let $\mathbf{x}^{\sharp}\in\mathcal{B}(G^{\sharp},F)$ be the vertex corresponding to the self-dual lattice $L^{\sharp}$. Let $\mathcal{G}_{\mathbf{x}^{\sharp}}$ be the stabilizer of the vertex. The GSpin Rapoport--Zink space $\textup{RZ}_{L^{\sharp
}}$ associated to the self-dual lattice $L^\sharp $ is the Rapoport--Zink space associated to the datum $(G^\sharp ,  b^\sharp , \mu^\sharp , C^\sharp )$ and the hyperspecial subgroup $\mathcal{G}_{\mathbf{x}^{\sharp}}$. 

By \cite[Theorem B]{HP17}, the Rapoport--Zink space $\textup{RZ}_{L^{\sharp}}$ is a formal scheme formally locally of finite type and formally smooth of relative dimension $n$ over $\mathcal{O}_{\breve{F}}$. If $R$ is a formally finitely generated $O_{\breve F}$-algebra (and formally smooth over $O_{\breve F}/\varpi^k$ for some $k \geq 1$), then $\textup{RZ}_{L^{\sharp}}(R)$ is the set of isomorphism classes of tuples $(X^\sharp , (t_{\alpha})_{\alpha \in I}, \rho)$ where $X^\sharp $ is a $p$-divisible group over $R$, $(t_{\alpha})$ is a collection of crystalline Tate tensors on $X$ and $\rho: \BX \otimes R/J \to X^\sharp  \otimes R/J$ is a quasi-isogeny such that $t_\alpha$ pullbacks to $t_{\alpha,0}$ under $\rho$. Here $J$ is some ideal of definition of $R$ such that $p \in J$. It is also equipped with a closed immersion of formal schemes by its construction
\begin{equation*}
    \textup{RZ}_{L^{\sharp}}\rightarrow\textup{RZ}_{\textup{GSp}(C^{\sharp})}.
\end{equation*}
Denote by $(X^{\sharp,\textup{univ}}, \rho^{\sharp,\textup{univ}},\lambda^{\sharp,\textup{univ}})$ the universal $p$-divisible group and the associated quasi-isogeny over $\textup{RZ}_{L^\sharp }$. Let $\CN_{L^{\sharp}}$ be the connected component of $\textup{RZ}_{L^\sharp }$ such that the universal polarization $\lambda^{\textup{univ}}$ and $\lambda_0^{\sharp}$ differ by an element in $\Of^{\times}$ under the quasi-isogeny $\rho^{\sharp,\textup{univ}}$.

\subsection{Special cycles: the self-dual case}
\label{spcyc-self-dual}
The inclusion $L^\sharp  \subseteq (C^\sharp )^{\text{op}}$ (where $L^{\sharp}$ acts
on $C^{\sharp}$ via right multiplication) gives a natural inclusion $(L^\sharp )_{\breve{F}} \subseteq \End_{O_F}(D^\sharp )$ as the space of special quasi-endomorphisms of $D^\sharp $. The lattice $L^\sharp $ gives a natural isocrystal $((L^\sharp )_{\breve{F}}, \Phi=\bar{b^{\sharp}} \circ \sigma)$ where $\bar{b^\sharp } \in \SO(V^\sharp )(\breve{F})$ is the image of $b$ under the quotient map $\GSpin(V^\sharp ) \to \SO(V^\sharp )$.  Taking the $\Phi$-fixed point, we obtain the space of special quasi-endomorphisms for $\BX^\sharp $
\[
\BV^\sharp = ((L^\sharp )_{\breve{F}})^{\Phi=\text{id}} \subseteq \End^\circ(\BX^{\sharp}):=\End(\BX^{\sharp})[1/p].
\]

Consider the inner form $J^\sharp =\GSpin(\BV^\sharp )$ of $G^\sharp $. Via the embedding $J^\sharp  \in \End^\circ(\BX^{\sharp})^\times$, $J^\sharp (F)$ acts on the Rapoport-Zink spaces $\CN_{L^\sharp} $ by changing the framing quasi-isogeny.
\begin{definition}\label{spe-unram}
    Let $\Lambda\subset\BV^{\sharp}$ be a subset. Define the special cycle $\mathcal{Z}^{\sharp}(\Lambda)$ to be closed formal subscheme of $\CN_{L^{\sharp}}$ cut out by the condition
    \begin{equation*}
        \rho^{\sharp,\textup{univ}}\circ x \circ(\rho^{\sharp,\textup{univ}})^{-1}\subset\textup{End}(X^{\sharp,\textup{univ}})
    \end{equation*}
    for all $x\in\Lambda$. When $\Lambda=\{x\}$ consists of one single element, denote by $\CZ^{\sharp}(x)$ the closed formal subscheme $\CZ^{\sharp}(\{x\})$.
\end{definition}
\begin{lemma}
    Let $x\in\BV^{\sharp}$ be a nonzero and integral element. The special cycle $\CZ^{\sharp}(x)$ is an effective Cartier divisor on the formal scheme $\CN_{L^{\sharp}}$.
\end{lemma}
\begin{proof}
    This is proved in \cite[Proposition 4.10.1]{LiZhang-orthogonalKR}.
\end{proof}

\subsection{The crystal $\mathbf{V}_{\textup{crys}}^{\sharp}$}
For the self-dual lattice $L^{\sharp}$, Madapusi constructed a tensor $\boldsymbol{\pi}\in (C^{\sharp})^{\otimes(2,2)}$ (cf. \cite[Lemma 1.4]{Mad16}) which is stabilized by the group $G^{\sharp}$. By \cite[Theorem 4.9.1]{Kim18} we obtain from $\boldsymbol{\pi}$ a universal crystalline Tate tensor $\boldsymbol{\pi}_{\crys}$ on the universal $p$-divisible group $X^{\sharp,\univ}$ over $\textup{RZ}_{L^{\sharp}}$, which induces a projector of crystals $\boldsymbol{\pi}_{\crys} : \textup{End}(\mathbb{D}(X^{\sharp,\univ})) \rightarrow \textup{End}(\mathbb{D}(X^{\sharp,\univ}))$ whose image $\mathbf{V}_{\crys}^{\sharp}\coloneqq \textup{im}(\boldsymbol{\pi}_{\crys})$  is a crystal of $\mathcal{O}_{\textup{RZ}_{L^{\sharp}}/\mathcal{O}_{\breve{F}}}$-module of rank $n+2$.
\par
For any $S\in\textup{Alg}_{\mathcal{O}_{\breve{F}}}$ and any $z\in\textup{RZ}_{L^{\sharp}}(S)$, we similarly have a projector of crystals
\begin{equation*}
    \boldsymbol{\pi}_{\textup{crys},z}:\textup{End}(\mathbb{D}(X_z))\rightarrow\textup{End}(\mathbb{D}(X_z)).
\end{equation*}
whose image $\mathbf{V}_{\crys,z}^{\sharp} := \textup{im}(\boldsymbol{\pi}_{\textup{crys},z})$ is a crystal of $\mathcal{O}_{S}^{\textup{crys}}$-module of rank $n+2$. Here $X_z$ denotes the $p$-divisible group over $S$ obtained by the base change of $X^{\textup{univ}}$ to $z$.
\par
Let $R\rightarrow S$ be a surjection in $\textup{Alg}_{\mathcal{O}_{\breve{F}}}$ whose kernel admits nilpotent divided powers, we have $\vcrys^{\sharp}(R)$ a projective $R$-module of rank $n+2$. It is equipped with a non-degenerate symmetric $R$-bilinear form $(\cdot,\cdot)$. The projective $S$-module $\vcrys^{\sharp}(S)$ is equipped with a Hodge filtration $\textup{Fil}^{1}\vcrys^{\sharp}(S)$, which is an isotropic $S$-line (\cite[$\S$4.3]{LiZhang-orthogonalKR}) such that
\begin{equation*}
    \textup{Fil}^{1}\mathbb{D}(X_z)=\textup{im}(\textup{Fil}^{1}\vcrys^{\sharp}(S)).
\end{equation*}
\par
Moreover, let $M\subset\mathbb{V}^{\sharp}$ be a subset, for $z\in\mathcal{Z}^{\sharp}(M)(S)$ and an arbitrary element $x\in M$, the crystalline realization $x_{\textup{crys},z}(R)\in\textup{End}(\mathbb{D}(X_z))(R)$ lies in the image of $\boldsymbol{\pi}_{\textup{crys},z}(R)$, hence it is an element in $\vcrys^{\sharp}(R)$ (cf. \cite[$\S$4.6]{LiZhang-orthogonalKR}).
\begin{lemma}
    Let $R\rightarrow S$ be a surjection in $\textup{Alg}_{\mathcal{O}_{\breve{F}}}$ whose kernel admits nilpotent divided powers.\\
    (i)\,Let $z_0\in\mathcal{N}_{L^{\sharp}}(\mathbb{F})$ and $\widehat{\mathcal{N}}_{L^{\sharp},z_0}$ be the completion of $\mathcal{N}_{L^{\sharp}}$ at $z_0$. Let $z\in\widehat{\mathcal{N}}_{L^{\sharp},z_0}(S)$. Then there is a natural bijection
    \begin{equation*}
        \left\{\textup{Lifts}\,\, \tilde{z}\in\widehat{\mathcal{N}}_{L^{\sharp},z_0}(R)\,\,\textup{of}\,\,z. \right\}\stackrel{\sim}\longrightarrow\left\{\textup{Isotropic $R$-lines in}\,\,\mathbf{V}^{\sharp}_{\textup{crys},z}(R)\,\,\textup{lifting}\,\,\textup{Fil}^{1}\mathbf{V}_{\textup{crys},z}^{\sharp}(S).\right\},
    \end{equation*}
    which maps a lift $\Tilde{z}\in\widehat{\mathcal{N}}_{L^{\sharp},z_0}(R)$ to the Hodge filtration $\textup{Fil}^{1}\mathbf{V}^{\sharp}_{\textup{crys},\Tilde{z}}(R)\subset\mathbf{V}^{\sharp}_{\textup{crys},\Tilde{z}}(R)\simeq\mathbf{V}^{\sharp}_{\textup{crys},z}(R)$.
    \\
    (ii)\,Let $M\subset\mathbb{V}^{\sharp}$ be a $\mathcal{O}_{F}$-lattice of rank $r\geq1$. Let $z_0\in\mathcal{Z}^{\sharp}(M)(\mathbb{F})$ and $\widehat{\mathcal{Z}^{\sharp}(M)}_{z_0}$ be the completion of $\mathcal{Z}^{\sharp}(M)$ at $z_0$, then there is a natural bijection compatible with the bijection in (i)
\begin{equation*}
      \left\{\textup{Lifts}\,\, \tilde{z}\in\widehat{\mathcal{Z}^{\sharp}(M)}_{z_0}(R)\,\,\textup{of}\,\,z. \right\}\stackrel{\sim}\longrightarrow 
    \left\{\begin{array}{ll}
     \textup{Isotropic $R$-lines in $\mathbf{V}^{\sharp}_{\textup{crys},z}(R)$ lifting $\textup{Fil}^{1}\vcrys^{\sharp}(S)$}\\
    \textup{and orthogonal to $x_{\textup{crys},z}(R)$ for any $x\in M$.}
    \end{array}\right\}.
\end{equation*}
\label{deformo}
\end{lemma}
\begin{proof}
    This follows from the proof of \cite[Lemma 4.6.2]{LiZhang-orthogonalKR}.
\end{proof}

\subsection{Rapoport--Zink space associated to an almost self-dual lattice}\label{rz-con}

Let $L$ be an almost self-dual lattice of rank $n+1$, i.e., $L^{\vee}/L\simeq\Of/(\varpi)$. Let $L^{\sharp}$ be a self-dual lattice of rank $n+2$ equipped with a primitive isometric embedding $\iota: L\rightarrow L^{\sharp}$ as in Lemma \ref{embedding L to self dual}. Let $\Lambda=\{x\in L^{\sharp}:(x,y)=0\,\,\textup{for any}\,\,y\in L\}$ is the sublattice of $L^{\sharp}$ orthogonal to $L$ with rank $1$. Let $V=L_F$ and $V^\sharp =L^\sharp _F$. We have a decomposition $V^\sharp =V \oplus \Lambda_F$. As $\Lambda^\vee/ \Lambda \cong \Of/(\varpi)$, we choose a generator $x_0 \in \Lambda$ (unique up to a unit in $O_F$) such that $(x_0,x_0)=\varpi$.

Let $G=\textup{GSpin}(V)$. Denote by $\mathbf{x}\in\mathcal{B}(G,F)$ the vertex corresponding to $L$. Let $\mathcal{G}_{\mathbf{x}}\subset G(F)$ be the stabilizer of the vertex $\mathbf{x}$. Then there is a natural emebdding $C=C(L) \to C^\sharp=C(L^\sharp)$. Oki \cite[\S 2]{Oki20} constructed an embedding of integral local Shimura--Hodge datum $(G, b, \mu, C) \to (G^\sharp, b^\sharp, \mu^\sharp, C^\sharp)$ such that we have the following relation between $\bar{b^\sharp } \in \SO(V^\sharp )(\breve{F})$ and $\bar{b} \in \SO(V)(\breve{F})$:
\begin{equation}
    \bar{b^\sharp }|_{V}=\bar{b}, \bar{b^\sharp }(x_0)=x_0. 
    \label{relation-basic}
\end{equation}
Therefore, we may regard $x_0 \in \BV^\sharp $.

The Rapoport--Zink space $\CN_L$ associated to $L$ is defined to be the Rapoport--Zink space associated to the local Shimura--Hodge datum $(G, b, \mu, C)$ and the parahoric subgroup $\mathcal{G}_{\mathbf{x}}$. By the construction, we have a closed immersion of formal schemes
\begin{equation*}
    \textup{RZ}_L\rightarrow\textup{RZ}_{\textup{GSp}(C)} \to \textup{RZ}_{\textup{GSp}(C^\sharp)}.
\end{equation*}
Denote by $(X^{\textup{univ}}, \rho^{\textup{univ}},\lambda^{\textup{univ}})$ the universal $p$-divisible group and the associated quasi-isogeny over $\textup{RZ}_{L}$. Let $\CN_{L}$ be the connected component of $\textup{RZ}_{L^\sharp }$ such that the universal polarization $\lambda^{\textup{univ}}$ and $\lambda_0^{\sharp}$ differ by an element in $\Of^{\times}$ under the quasi-isogeny $\rho^{\textup{univ}}$.

By functorality of local Shimura data \cite{HodgeRZparahoric}, the closed immersion $\textup{RZ}_L\rightarrow\textup{RZ}_{\textup{GSp}(C^{\sharp})}$ factors through the closed immersion $\textup{RZ}_{L^{\sharp}}\rightarrow\textup{RZ}_{\textup{GSp}(C^{\sharp})}$.  So we have a closed immersion of formal schemes
\begin{equation*}
    i_{L}:\textup{RZ}_L\rightarrow\textup{RZ}_{L^{\sharp}}\,\,\,\, (\textup{resp.}\,i_L:\CN_L\rightarrow\CN_{L^{\sharp}}).
\end{equation*}

\begin{theorem}\label{almost-divisor}
The closed immersion $i_{L}$ induces an isomorphism of regular formal schemes 
$\CN_{L}\simeq\CZ^{\sharp}(x_0) \subseteq \CN_{L^\sharp }$ of dimension $n$.
\end{theorem}
\begin{proof}
This is a local analog of \cite[Lemma 7.1]{Mad16} (when $t=1$) over generic fiber. 
 The construction in \cite[Section 4]{HodgeRZparahoric} uses global uniformization from Hodge type Shimura varieties and embeddings to Siegel modular varieties. We have a similar uniformization result on global integral models of orthogonal Shimura varieties \cite[Theorem 7.4]{Mad16} with almost self-dual and self-dual level at $p$ over $\Zp$. The theorem could be proved via global uniformizations. We now give a local proof.

Firstly, we claim that $\CN_{L} \subseteq \CZ^{\sharp}(x_0)$. We need to show over $\CN_L$, $x_0 \in \BV^\sharp$ lifts to a special homomorphism, which follows from \cite[Proposition 6.2.2]{HM20}.

Secondly, we claim that $\CZ^{\sharp}(x_0) = \CN_{L}$. Note both are regular of dimension $n$. The regularity for $\CN_L$ follows from local models as \cite{Mad16}\cite[Proposition 6.3.3.]{HM20} .   Since $\nu_p(q(x_0))=1$, we have $\mathcal{Z}^{\sharp}(x_0)$ equals to the difference divisor on $\CN_{L^{\sharp}}$ associated to $x_0$, we know $\CZ^{\sharp}(x_0)$ is regular by the regularity of difference divisors on $\CN_{L^\sharp }$ prove \cite{Zhu23diff}. We just need to check the closed subscheme $\CN_L$ is the whole $\CZ(x_0)$ as Cartier divisors over $\CN_{L^\sharp}$. This follows from the identification of mod $p$ geometric points via Dieudonne theory (or Bruhat-Tits stratification in \S 4.6).
\end{proof}

\subsection{Special cycles: the almost self-dual case}
By similar arguments in $\S$\ref{spcyc-self-dual}. The lattice $L$ gives a natural isocrystal $(L_{\breve{F}}, \Phi=\bar{b} \circ \sigma)$ where $\bar{b} \in \SO(V)(\breve{F})$ is the image of $b$ under the quotient map $\GSpin(V) \to \SO(V)$.  Taking the $\Phi$-fixed point, we obtain the space of special quasi-endomorphisms for $\BX^\sharp $
\[
\BV = (L_{\breve{F}})^{\Phi=\text{id}} \subseteq \End^\circ(\BX^{\sharp}):=\End(\BX^{\sharp})[1/p].
\]
We have a natural decomposition $\BV^\sharp =\BV \obot F x_0$ by (\ref{relation-basic}).
\begin{definition}\label{spe-almo}
    Let $S\subset\BV$ be a subset. Define the special cycle $\mathcal{Z}(\Lambda)\subset\CN_L$ to be closed formal subscheme of $\CN_{L}$ cut out by the condition
    \begin{equation*}
        \rho^{\textup{univ}}\circ x \circ(\rho^{\textup{univ}})^{-1}\subset\textup{End}(X^{\textup{univ}})
    \end{equation*}
    for all $x\in S$. When $S=\{x\}$ consists of one single element, denote by $\CZ(x)$ the closed formal subscheme $\CZ(\{x\})$.
\end{definition}
The following lemma follows from comparing Definition \ref{spe-unram} and Definition \ref{spe-almo}.
\begin{lemma}
    Let $x\in\BV$ be a nonzero and integral element. We have the following equality of closed formal subschemes of $\CN_L$:
    \begin{equation*}
        \CZ(x)=\iota_L^{-1}\left(\CZ^{\sharp}(x)\cap\CZ^{\sharp}(x_0)\right).
    \end{equation*}
    Hence the special cycle $\CZ^{\sharp}(x)$ is an effective Cartier divisor on the formal scheme $\CN_{L}$.\label{cancellation-law}
\end{lemma}

\subsection{Bruhat--Tits stratification}\label{BT-strat}
Let $\mathcal{L}\subset\mathbb{V}$ be a vertex lattice. Then $W_{\mathcal{L}}\coloneqq \mathcal{L}^{\vee}/\mathcal{L}$ is an $\mathbb{F}_p$-vector space of dimension $t(\mathcal{L})$, equipped with a non-degenerate quadratic form induced from $\mathbb{V}$. By \cite[Corollary 7.3]{Oki20}, the type $t(\mathcal{L})$ of a vertex lattice $\mathcal{L}\subset\mathbb{V}$ is always an odd integer such that $1 \leq t(\mathcal{L}) \leq t_{\max}$, where
\begin{equation*}
    t_{\max}=\begin{cases}
        n, &\textup{if $n$ is odd;}\\
        n-1, &\textup{if $n$ is even and $\epsilon(\mathbb{V})=(p,-1)_p^{n/2}$;}\\
        n+1, &\textup{if $n$ is even and $\epsilon(\mathbb{V})\neq(p,-1)_p^{n/2}$.}
    \end{cases}
\end{equation*}
Denote by $\textup{Vert}^{t}(\mathbb{V})$ the set of vertex lattices of type $t$ in the space $\mathbb{V}$. For the vertex lattice $\mathcal{L}$, we have the associated generalized Deligne--Lusztig variety $Y_{W_{\mathcal{L}}}$ of dimension $(t(\mathcal{L})-1)/2$. The reduced subscheme of the minuscule special cycle $\mathcal{V}(\mathcal{L}) \coloneqq \mathcal{Z}(\mathcal{L})^{\textup{red}}$ is isomorphic to $Y_{W_{\mathcal{L}},\mathbb{F}}$. In fact $\mathcal{Z}(\mathcal{L})$ itself is already reduced (\cite[Corollary 7.7]{Oki20}, see also \cite[Theorem B]{Li_Zhu_2018}), so $\mathcal{V}(\mathcal{L}) = \mathcal{Z}(\mathcal{L})$.
\par
The reduced subscheme of $\CN_L$ satisfies $\CN_L^{\textup{red}}=\bigcup\limits_{\mathcal{L}} \mathcal{V}(\mathcal{L})$, where $\mathcal{L}$ runs over all vertex lattices $\mathcal{L}\subset\mathbb{V}$. For two vertex lattices $\mathcal{L}, \mathcal{L}^{\pr}$, we have $\mathcal{V}(\mathcal{L}) \subset \mathcal{V}(\mathcal{L}^{\pr})$ if and only if $\mathcal{L}\supset\mathcal{L}^{\pr}$; and $\mathcal{V}(\mathcal{L}) \cap \mathcal{V}(\mathcal{L}^{\pr})$ is nonempty if and only if $\mathcal{L} +\mathcal{L}^{\pr}$ is also a vertex lattice, in which case it is equal to $\mathcal{V}(\mathcal{L} + \mathcal{L}^{\pr})$. In this way we obtain a Bruhat--Tits stratification of $\CN_L^{\textup{red}}$ by locally closed subvarieties (\cite[Theorem 7.5]{Oki20}, see also \cite[$\S$6.5]{HP17}),
\begin{equation*}
    \CN_L^{\textup{red}}=\bigcup\limits_{\mathcal{L}}\mathcal{V}(\mathcal{L})^{\circ},\,\,\mathcal{V}(\mathcal{L})^{\circ}\coloneqq\mathcal{V}(\mathcal{L})\backslash\left(\bigcup\limits_{\mathcal{L}\subsetneq\mathcal{L}^{\pr}}\mathcal{V}(\mathcal{L}^{\pr})\right).
\end{equation*}

\subsection{The crystal $\mathbf{V}_{\textup{crys}}$}
Denote by $\mathbf{V}^{\sharp}_{\crys}\vert_{\CN_L}$ the pullback of the crystal $\mathbf{V}_{\crys}^{\sharp}$ on $\CN_{L^{\sharp}}$ to the closed formal subscheme $\CN_L$. It is an $\mathcal{O}_{\CN_{L}/\mathcal{O}_{\breve{F}}}$-module crystal of rank $n+2$. The isomorphism $\mathcal{Z}^{\sharp}(x_0)\simeq\CN_L$ implies that there exists an isogeny
\begin{equation*}
    x_0^{\univ}:X^{\univ}\rightarrow X^{\univ}
\end{equation*}
which lifts the quasi-isogeny $x_0\in\mathbb{V}^{\sharp}$. For an object $(S,T,\delta)$ in the site $\textup{NCRIS}_{\ofb}(\CN_L/\textup{Spec}\,\ofb)$ where $S\rightarrow\CN_L$ is a morphism over $\ofb$ and $S\rightarrow T$ is a divided power thickening, denote by $x_{0,\crys}(S,T,\delta)$ the endomorphism of the $\mathcal{O}_T$-module $\mathbb{D}(X^{\univ})(S,T,\delta)$ induced by the isogeny $x_0^{\univ}$, we have
\begin{equation*}
    x_{0,\crys}(S,T,\delta)\in\mathbf{V}^{\sharp}_{\crys}(S,T,\delta)
\end{equation*}
by the definition of the sheaf $\mathbf{V}^{\sharp}_{\crys}$. Denote by $(\cdot,\cdot):\mathbf{V}^{\sharp}_{\crys}\vert_{\CN_L}\times\mathbf{V}^{\sharp}_{\crys}\vert_{\CN_L}\rightarrow\mathcal{O}_{\CN_{L}/\mathcal{O}_{\breve{F}}}$ the bilinear form induced from the bilinear form on $\mathbf{V}^{\sharp}_{\crys}$.
\begin{definition}
    The crystalline sheaf $\mathbf{V}_{\crys}$ in the site $\textup{NCRIS}_{\ofb}(\CN_L/\textup{Spec}\,\ofb)$ is defined in the following way: For the object $(S,T,\delta)$ in this site where $S\rightarrow\CN_L$ is a morphism over $\ofb$ and $S\rightarrow T$ is a divided power thickening, then 
    \begin{equation*}
        \mathbf{V}_{\crys}(S,T,\delta)=\{v\in\mathbf{V}^{\sharp}_{\crys}(S,T,\delta):(v,x_{0,\crys}\left(S,T,\delta\right))=0\}.
    \end{equation*}
\end{definition}

\begin{lemma}
    The crystalline sheaf $\mathbf{V}_{\crys}$ is an $\mathcal{O}_{\CN_{L}/\mathcal{O}_{\breve{F}}}$-module crystal of rank $n+1$.
\end{lemma}
\begin{proof}
    For an object $(S,T,\delta)$ in the site $\textup{NCRIS}_{\ofb}(\CN_L/\textup{Spec}\,\ofb)$. By evaluating at geometric points, we can verify that the submodule generated by $x_{0,\crys}\left(S,T,\delta\right)$ is always a direct summand of the module $\mathbf{V}^{\sharp}_{\crys}(S,T,\delta)$. The non-degeneracy of the pairing $(\cdot,\cdot)$ on $\mathbf{V}^{\sharp}_{\crys}$ implies that $\mathbf{V}_{\crys}$ must be an $\mathcal{O}_{\CN_{L}/\mathcal{O}_{\breve{F}}}$-module crystal of rank $n+1$.
\end{proof}
Let $R\rightarrow S$ be a surjection in $\textup{Alg}_{\mathcal{O}_{\breve{F}}}$ whose kernel admits nilpotent divided power. For a point $z\in\CN_L(S)$, we have the following inclusion by Lemma \ref{deformo}
\begin{equation*}
    \textup{Fil}^{1}\mathbf{V}^{\sharp}_{\crys,z}(S)\subset\mathbf{V}_{\crys,z}(S).
\end{equation*}
Denote by $\textup{Fil}^{1}\mathbf{V}_{\crys,z}(S)$ this filtration. It is an isotropic $S$-line.
\begin{lemma}
    Let $R\rightarrow S$ be a surjection in $\textup{Alg}_{\mathcal{O}_{\breve{F}}}$ whose kernel admits nilpotent divided powers.\\
    (i)\,Let $z_0\in\mathcal{N}_{L}(\mathbb{F})$ and $\widehat{\mathcal{N}}_{L,z_0}$ be the completion of $\mathcal{N}_{L}$ at $z_0$. Let $z\in\widehat{\mathcal{N}}_{L,z_0}(S)$. Then there is a natural bijection
    \begin{equation*}
        \left\{\textup{Lifts}\,\, \tilde{z}\in\widehat{\mathcal{N}}_{L,z_0}(R)\,\,\textup{of}\,\,z. \right\}\stackrel{\sim}\longrightarrow\left\{\textup{Isotropic $R$-lines in}\,\,\mathbf{V}_{\textup{crys},z}(R)\,\,\textup{lifting}\,\,\textup{Fil}^{1}\mathbf{V}_{\textup{crys},z}(S).\right\},
    \end{equation*}
    which maps a lift $\Tilde{z}\in\widehat{\mathcal{N}}_{L,z_0}(R)$ to the Hodge filtration $\textup{Fil}^{1}\mathbf{V}_{\textup{crys},\Tilde{z}}(R)\subset\mathbf{V}_{\textup{crys},\Tilde{z}}(R)\simeq\mathbf{V}_{\textup{crys},z}(R)$.
    \\
    (ii)\,Let $M\subset\mathbb{V}$ be a $\mathcal{O}_{F}$-lattice of rank $r\geq1$. Let $z_0\in\mathcal{Z}(M)(\mathbb{F})$ and $\widehat{\mathcal{Z}(M)}_{z_0}$ be the completion of $\mathcal{Z}(M)$ at $z_0$, then there is a natural bijection compatible with the bijection in (i)
\begin{equation*}
      \left\{\textup{Lifts}\,\, \tilde{z}\in\widehat{\mathcal{Z}(M)}_{z_0}(R)\,\,\textup{of}\,\,z. \right\}\stackrel{\sim}\longrightarrow 
    \left\{\begin{array}{ll}
     \textup{Isotropic $R$-lines in $\mathbf{V}_{\textup{crys},z}(R)$ lifting $\textup{Fil}^{1}\vcrys(S)$}\\
    \textup{and orthogonal to $x_{\textup{crys},z}(R)$ for any $x\in M$.}
    \end{array}\right\}.
\end{equation*}
\label{deformo2}
\end{lemma}
\begin{proof}
    This follows from combining Lemma \ref{deformo}, Theorem \ref{almost-divisor} and Lemma \ref{cancellation-law}.
\end{proof}

\subsection{Local models}
To understand the non-formally smooth locus of our Rapoport-Zink spaces, we need to use the theory of local models. For a quadratic lattice $\Lambda$ over $\mathcal{O}_F$, let $Q(\Lambda)$ be a quadratic over $\textup{Spec}\,\mathcal{O}_F$ such that for an $\mathcal{O}_F$-algebra $R$, the $R$-points of $Q(\Lambda)$ parameterize isotropic lines, i.e., $R$-locally free rank 1 direct summands 
\begin{equation*}
    \mathcal{F}\subset\Lambda\otimes_{\mathcal{O}_F}R,
\end{equation*}
with $(\mathcal{F},\mathcal{F})_R=0$. Here $(\cdot,\cdot)_R$ is the symmetric $R$-bilinear form induced from the quadratic form on $\Lambda$.
\par
Let $\textup{M}^{\textup{loc}}(L)$ (resp. $\textup{M}^{\textup{loc}}(L^{\sharp})$) be the local model for the triple $(G,\mu,\mathbf{x})$ (resp. $(G^{\sharp},\mu^{\sharp},\mathbf{x}^{\sharp})$) in the sense of Pappas and Zhu \cite{PZ13}.
\begin{lemma}\label{lemma: local model}
    There exists a $\textup{GSpin}(L)$ (resp. $\textup{GSpin}(L^{\sharp})$)-equivariant isomorphism
    \begin{equation*}
        \textup{M}^{\textup{loc}}(L)\stackrel{\sim}\rightarrow Q(L),\,\,\,(\textup{resp.}\,\, \textup{M}^{\textup{loc}}(L^{\sharp})\stackrel{\sim}\rightarrow Q(L^{\sharp}))
    \end{equation*}
\end{lemma}
\begin{proof}
    This is proved in \cite[Proposition 12.7]{HPR20}.
\end{proof}

\subsection{non-formally smooth point on the local model $\textup{M}^{\textup{loc}}(L)$}
The lattice $L$ admits a decomposition
\begin{equation*}
    L=H\obot\mathcal{O}_F\cdot l,
\end{equation*}
where $l\in L$ is a vector such that $\nu_\varpi(q_L(l))=1$. Let $\overline{l}\in L/pL$ be the image of the element $l$. Let $s_0\in\textup{M}^{\textup{loc}}(L)(\mathbb{F})$ be the $\mathbb{F}$-point corresponding to the isotropic line $\mathbb{F}\cdot\overline{l}\subset L\otimes_{\Zp}\mathbb{F}$. It is the unique non-smooth point of $\textup{M}^{\textup{loc}}(L)$. Let $\textup{M}^{\textup{loc}}(L)_{\mathcal{O}_{\Breve{F}}}=\textup{M}^{\textup{loc}}(L)\times_{\textup{Spec}\,\mathcal{O}_F}\textup{Spec}\,{\mathcal{O}_{\Breve{F}}}$. Let $\widehat{\mathcal{O}}_{\textup{M}^{\textup{loc}}(L)_{\mathcal{O}_{\Breve{F}}},s_0}$ be the completion of the local ring of the scheme $\textup{M}^{\textup{loc}}(L)_{\mathcal{O}_{\Breve{F}}}$ at the point $s_0$. Let $\{e_i\}_{i=1}^{n}$ be an orthogonal normal basis of $H$. Let $a_i=q_L(e_i)\in\mathcal{O}_F^{\times}$. Then by the local model diagram for Rapoport--Zink spaces (Lemma \ref{lemma: local model}), we have
\begin{equation}
    \widehat{\mathcal{O}}_{\textup{M}^{\textup{loc}}(L)_{\mathcal{O}_{\Breve{F}}},s_0}\simeq\mathcal{O}_{\Breve{F}}[[s_1,\cdots,s_n]]\bigg/\left(q_L(l)+\sum\limits_{i=1}^{n}a_is_i^{2}\right).
    \label{nfs-local-ring}
\end{equation}

\subsection{Geometric points on $\CN_L$}
Let $z_0\in\CN_L(\mathbb{F})$ be a point of $\CN_L$. Denote by $\mathbf{V}_{z_0}^{\sharp}$ (resp. $\mathbf{V}_{z_0}$) the rank $n+2$ (resp. $n+1$) free $\ofb$-module $\mathbf{V}_{\crys}^{\sharp}(\textup{Spec}\,\mathbb{F},\textup{Spf}\,\ofb,\delta)$ (resp. $\mathbf{V}_{\crys}(\textup{Spec}\,\mathbb{F},\textup{Spf}\,\ofb,\delta)$) where the morphism $\textup{Spec}\,\mathbb{F}\rightarrow\CN_L$ corresponds to the point $z_0$. Denote by $\mathbb{D}_{z_0}$ the (covariant) Dieudonne module of the base change of the universal $p$-divisible group $X^{\univ}$ to the point $z_0$. Denote by $x_{0,z_0}$ the crystalline realization of the isogeny $x_0$ at the point $z_0$, it is an element in the $\ofb$-module $\mathbf{V}_{z_0}$, hence also $\mathbf{V}_{z_0}^{\sharp}$. Denote by $\overline{x_{0,z_0}}$ the image of $x_{0,z_0}$ in the $\mathbb{F}$-vector space $\overline{\mathbf{V}}_{z_0}^{\sharp}\coloneqq\mathbf{V}_{z_0}^{\sharp}/\varpi\mathbf{V}_{z_0}^{\sharp}$ (resp. $\overline{\mathbf{V}}_{z_0}\coloneqq\mathbf{V}_{z_0}/\varpi\mathbf{V}_{z_0}$). Denote by $\textup{Fil}^{1}\overline{\mathbf{V}}_{z_0}^{\sharp}$ (resp. $\textup{Fil}^{1}\overline{\mathbf{V}}_{z_0}$) the isotropic line corresponding to the point $z_0$.
\begin{lemma}\label{singular-criterion}
    Let $z_0\in\CN_L(\mathbb{F})$ be a point. The point $z_0$ is a non-formally smooth point of the formal scheme $\CN_L$ if and only if
    \begin{equation*}
        \overline{x_{0,z_0}}\in\textup{Fil}^{1}\overline{\mathbf{V}}_{z_0}\,\,(=\textup{Fil}^{1}\overline{\mathbf{V}}_{z_0}^{\sharp}),
    \end{equation*}
    i.e., $\textup{Fil}^{1}\overline{\mathbf{V}}_{z_0}$ is the radical of the quadratic space $\overline{\mathbf{V}}_{z_0}$.   
\end{lemma}
\begin{proof}
    By Theorem \ref{almost-divisor}, we have an isomorphism $\CN_L\simeq\mathcal{Z}^{\sharp}(x_0)$. By \cite[Theorem 1.2.1]{Zhu23diff}, the divisor $\mathcal{Z}^{\sharp}(x_0)$ is not formally smooth at $z_0\in\mathcal{Z}^{\sharp}(x_0)(\mathbb{F})$ if and only if there doesn't exist a lift of $z_0$ to $z\in\mathcal{Z}^{\sharp}(x_0)(\ofb/\varpi^{2})$. Let $l_0\in\mathbf{V}^{\sharp}_{z_0}$ be an isotropic vector whose image $\overline{l}_0\in\overline{\mathbf{V}}^{\sharp}_{z_0}$ generates the filtration $\textup{Fil}^{1}\overline{\mathbf{V}}_{z_0}^{\sharp}$.
    \par
    If $\overline{x_{0,z_0}}\in\textup{Fil}^{1}\overline{\mathbf{V}}_{z_0}=\textup{Fil}^{1}\overline{\mathbf{V}}_{z_0}^{\sharp}$ (notice that the element $\overline{x_{0,z_0}}$ is nonzero since $\nu_\varpi(q(x_0))=1$), there exists $a\in\ofb^{\times}$ and $v\in\mathbf{V}^{\sharp}_{z_0}$ such that $l_0=a\cdot x_{0,z_0}+\varpi\cdot v$. The isotropicity of $l_0$ implies that $0=a^{2}\cdot q(x_0)+a\varpi\cdot(x_{0,z_0},v)+\varpi^{2}q(v)$. Hence we have
    \begin{align*}
        (l,x_{0,z_0})=2aq(x_0)+\varpi\cdot(v,x_{0,z_0})\equiv aq(x_0)\,\,\textup{mod}\,\,\varpi^{2}.
    \end{align*}
    Therefore $\nu_\varpi\left((l,x_{0,z_0})\right)=1$. Hence $z_0$ is a non-formally smooth point of the formal scheme $\mathcal{Z}^{\sharp}(x_0)\simeq\CN_L$ by \cite[Proposition 5.3.3]{Zhu23diff}.
    \par
    Now we assume $\overline{x_{0,z_0}}\notin\textup{Fil}^{1}\overline{\mathbf{V}}_{z_0}=\textup{Fil}^{1}\overline{\mathbf{V}}_{z_0}^{\sharp}$. If $\nu_{\varpi}((l_0,x_{0,z_0}))\geq2$, the point $z_0$ is a formally smooth point of the formal scheme $\mathcal{Z}^{\sharp}(x_0)\simeq\CN_L$ by \cite{Zhu23diff}. If $\nu_{\varpi}((l_0,x_{0,z_0}))=1$. Then there exists an element $v\in\mathbf{V}^{\sharp}_{z_0}$ such that
    \begin{equation*}
        (v,x_{0,z_0})\equiv-\varpi^{-1}(l_0,x_{0,z_0})\,\,\textup{mod}\,\,\varpi,\,\,\,\,(v,l_0)\equiv0\,\,\textup{mod}\,\,\varpi.
    \end{equation*}
    Then $q(l_0+\varpi v)\equiv 0\,\,\textup{mod}\,\,\varpi^{2}$ and $(l_0+\varpi v,x_{0,z_0})\equiv 0\,\,\textup{mod}\,\,\varpi^{2}$. The isotropic vector $\overline{l_0+\varpi v}$ represents a lift of $z_0$ to $z\in\mathcal{Z}^{\sharp}(x_0)(\ofb/\varpi^{2})$ by Lemma \ref{deformo}. Hence the point $z_0$ is a formally smooth point of the formal scheme $\mathcal{Z}^{\sharp}(x_0)\simeq\CN_L$.
\end{proof}
\begin{remark}
    By the local model diagrams of Rapoport--Zink spaces, we have the following isomorphism if $z_0\in\CN_L(\mathbb{F})$ is a non-formally smooth point of the formal scheme $\CN_L$:
    \begin{equation}
    \widehat{\mathcal{O}}_{\CN_{L},z_0}\simeq\widehat{\mathcal{O}}_{\textup{M}^{\textup{loc}}(L)_{\mathcal{O}_{\Breve{F}}},s_0}\simeq\mathcal{O}_{\Breve{F}}[[s_1,\cdots,s_n]]\bigg/\left(q_L(l)+\sum\limits_{i=1}^{n}a_is_i^{2}\right),
    \label{nfs-local-ring-2}
\end{equation}
i.e., the completed local ring of the formal scheme $\CN_L$ at the non-formally smooth point $z_0$ is isomorphic to the completed local ring of the moduli space of isotropic lines in the lattice $L\otimes_{\mathcal{O}_F}\ofb\simeq\mathbf{V}_{z_0}$ at the radical subspace of $\mathbf{V}_{z_0}/\varpi\mathbf{V}_{z_0}$.
\par
We will use a slightly different coordinate system of the local ring $\widehat{\mathcal{O}}_{\CN_{L},z_0}$. Let $x_1\in\mathbf{V}^{\sharp}_{z_0}$ such that $(x_{0,z_0},x_1)=1$. Let $\mathbf{H}\subset\mathbf{V}^{\sharp}_{z_0}$ be the orthogonal complement of the $\ofb$-submodule spanned by $x_{0,z_0}$ and $x_1$. Then $\mathbf{H}$ is a self-dual quadratic lattice of rank $n$. Let $\{l_i\}_{i=1}^{n}$ be an orthogonal normal basis of $\mathbf{H}$. Let $\boldsymbol{l}$ be the universal isotropic vector over the local ring $\widehat{\mathcal{O}}_{\CN_{L},z_0}$, we have the following expression by Lemma \ref{singular-criterion}, 
    \begin{equation*}
        \boldsymbol{l}=a\cdot x_{0,z_0} +t_0\cdot x_1 +\sum\limits_{i=1}^{n}t_i\cdot l_i,
    \end{equation*}
    for some element $a\in\widehat{\mathcal{O}}_{\CN_{L},z_0}^{\times}$. Then as $\boldsymbol{l}$ is isotropic, we get
    \begin{equation*}
        t_0\left(a+t_0q(x_0)\right)=-a^{2}-\sum\limits_{i=1}^{n}q(l_i)t_i^{2}.
    \end{equation*}
    We can solve $t_0$ in terms of $t_1,\cdots,t_n$. By isomorphism (\ref{nfs-local-ring-2}), we get
     \begin{equation*}
         \widehat{\mathcal{O}}_{\CN_{L^{\sharp}},z_0}\simeq\mathcal{O}_{\Breve{F}}[[t_1,\cdots,t_n]].
     \end{equation*}
     If $\boldsymbol{l}\in\textup{M}^{\textup{loc}}(L)$, we have $(\boldsymbol{l},x_{0,z_0})=0$, then $t_0=-2aq(x_0)$. By isomorphism (\ref{nfs-local-ring-2}), we get
    \begin{equation*}
        \widehat{\mathcal{O}}_{\CN_{L},z_0}\simeq\mathcal{O}_{\Breve{F}}[[t_1,\cdots,t_n]]\bigg/\left(a^{\prime}q(x_0)+\sum\limits_{i=1}^{n}q(l_i)t_i^{2}\right).
    \end{equation*}
    where $a^{\pr}\in\widehat{\mathcal{O}}_{\CN_{L},z_0}^{\times}$. 
    \label{coordinate}
\end{remark}
\begin{corollary}
    Let $z_0\in\CN_L(\mathbb{F})$ be a point. Let $l\in\mathbf{V}_{z_0}$ be an element such that its image in $\overline{\mathbf{V}}_{z_0}$ generates the isotropic line $\textup{Fil}^{1}\overline{\mathbf{V}}_{z_0}$. The formal scheme $\CN_L$ is formally smooth over $\ofb$ at the point $z_0$ if and only if
    \begin{equation*}
        l\circ x_{0,z_0}\notin\varpi\cdot\textup{End}_{\ofb}(\mathbb{D}_{z_0}).
    \end{equation*}
    \label{smooth-criterion2}
\end{corollary}
\begin{proof}
    The element $l\circ x_{0,z_0}\notin\varpi\cdot\textup{End}_{\ofb}(\mathbb{D}_{z_0})$ if and only if the element $\overline{l}\circ\overline{x_{0,z_0}}\in\textup{End}_{\mathbb{F}}(\mathbb{D}_{z_0}/\varpi\mathbb{D}_{z_0})$ is nonzero. Since $\mathbb{D}_{z_0}/\varpi\mathbb{D}_{z_0}$ is isomorphic to the Clifford algebra $C(\overline{\mathbf{V}}^{\sharp}_{z_0})$, the element $\overline{l}\circ\overline{x_{0,z_0}}\in\textup{End}_{\mathbb{F}}(\mathbb{D}_{z_0}/\varpi\mathbb{D}_{z_0})$ is nonzero if and only if the two vectors $\overline{l}$ and $\overline{x_{0,z_0}}$ are linearly independent. The two vectors $\overline{l}$ and $\overline{x_{0,z_0}}$ are linearly independent if and only if the formal scheme $\CN_L$ is formally smooth over $\ofb$ at $z_0$ by Lemma \ref{singular-criterion}.
\end{proof}

\subsection{Vertex lattices of type 1 and non-formally smooth points}
Let $\mathcal{L}\subset\mathbb{V}$ be a vertex lattice of type 1. Recall that in $\S$\ref{BT-strat}, we explained that the special cycle $\mathcal{Z}(\mathcal{L})=\mathcal{V}(\mathcal{L})$ is a reduced projective variety of dimension $0$, hence it is just an $\mathbb{F}$-point of $\CN_L$.
\begin{lemma}
    Let $z_0\in\CN_L(\mathbb{F})$ be a non-formally smooth point of $\CN_L$. There exists a unique vertex lattice $\mathcal{L}\subset\mathbb{V}$ of type $1$ such that $\mathcal{Z}(\mathcal{L})=\{z_0\}$.\label{cor-ver-1-singular}
\end{lemma}
\begin{proof}
    This is proved in \cite[Theorem 7.16]{Oki20}.
\end{proof}

\subsection{The blow up}
Let $\CN^{\textup{nfs}}_L$ be the non-formally smooth locus of the formal scheme $\CN_L$. It is a closed subscheme of both of the formal schemes $\CN_{L}$ and $\CN_{L^{\sharp}}$. Let $\M$ and $\Ms$ be the blow ups of the formal schemes $\CN_L$ and $\CN_{L^{\sharp}}$ along the closed subscheme $\CN^{\textup{nfs}}$ respectively. Denote by $\pi_L$ and $\pis$ be the blow up morphisms $\M\rightarrow\CN_L$ and $\Ms\rightarrow\CN_{L}$ respectively. Denote by $\exc_{L}$ and $\exc_{L^{\sharp}}$ the exceptional divisors on $\M$ and $\Ms$ respectively.
\par
Let a subset $M\subset\mathbb{V}^{\sharp}$ be a nonzero element. Denote by $\mathcal{Z}^{\sharp,\textup{bl}}(M)$ the pullback of the special cycle $\mathcal{Z}^{\sharp}(M)\subset\CN_{L^{\sharp}}$ to $\Ms$, i.e., $\mathcal{Z}^{\sharp,\textup{bl}}(M)=\mathcal{Z}^{\sharp}(M)\times_{\CN_{L^{\sharp}},\pis}\Ms$. For an element $x\in\mathbb{V}^{\sharp}$, denote by $\widetilde{\mathcal{Z}}^{\sharp,\textup{bl}}(x)$ the strict transform of the special divisor $\mathcal{Z}^{\sharp}(x)$. By Theorem \ref{almost-divisor}, the closed immersion $i_L:\CN_L\rightarrow\CN_{L^{\sharp}}$ induces the following isomorphism:
\begin{equation}
    \iota:\M\xrightarrow{\sim}\widetilde{\mathcal{Z}}^{\sharp,\textup{bl}}(x_0).\label{almost-spe-blow-up}
\end{equation}
We still use $i_L$ to denote the closed immersion $\M\rightarrow\Ms$.
\par
Let $z_0\in\CN_L(\mathbb{F})$ be a non-formally smooth point of the formal scheme $\CN_L$. Let $\widehat{\CN}_{L,z_0}$ (resp. $\widehat{\CN}_{L^{\sharp},z_0}$) be the completion of the formal scheme $\CN_L$ at the point $z_0$. Let $\widehat{\mathcal{M}}_{L,z_0}=\M\times_{\CN_L}\widehat{\CN}_{L,z_0}$ (resp. $\widehat{\mathcal{M}}_{L^{\sharp},z_0}=\Ms\times_{\CN_{L^{\sharp}}}\widehat{\CN}_{L,z_0}$). Using the coordinate system in Remark \ref{coordinate}, the formal scheme $\widehat{\mathcal{M}}_{L^{\sharp},z_0}$ is covered by $n+1$ open affine formal subschemes $\{\widehat{\mathcal{M}}_{L^{\sharp},z_0,i}\}_{i=0}^{n}$:
\begin{equation}
    \widehat{\mathcal{M}}_{L^{\sharp},z_0,0}\simeq\textup{Spf}\,\ofb[\{v_{i}\}_{i=1}^{n}],\,\,\,\,\widehat{\mathcal{M}}_{L^{\sharp},z_0,i}\simeq\textup{Spf}\,\ofb[\{u_{ji}\}_{j\neq i},u_i][[t_i]].
    \label{local-piece-smooth}
\end{equation}
Over $\widehat{\mathcal{M}}_{L,z_0,0}$, we have $t_i=v_{i}\varpi$ for $1\leq i\leq n$. While over $\widehat{\mathcal{M}}_{L^{\sharp},z_0,i}$ where $i\geq1$, we have $t_j=u_{ji}t_i$ for $j\neq i$ and $\varpi=u_it_i$. The formal scheme $\widehat{\mathcal{M}}_{L,z_0}$ is covered by $n$ open affine formal subschemes $\{\widehat{\mathcal{M}}_{L,z_0,i}\}_{i=1}^{n}$:
\begin{equation}
    \widehat{\mathcal{M}}_{L,z_0,i}\simeq\textup{Spf}\,\ofb[\{u_{ji}\}_{j\neq i}][[t_i]]\bigg/\left(a^{\pr}q(x_0)+t_i^{2}\cdot\left(q(l_i)+\sum\limits_{j\neq i}q(l_j)u_{ji}^{2}\right)\right).
    \label{local-piece}
\end{equation}
Let $z\in\widehat{\mathcal{M}}_{L,z_0}$ be a point. Denote by $\widehat{\mathcal{M}}_{L,z}$ the completion of the formal scheme $\M$ at the point $z$. Let $\mathfrak{m}_z$ be the maximal ideal of the completed local ring $\widehat{\mathcal{O}}_{\M,z}$. It's easy to see from (\ref{local-piece}) that $\varpi\notin\mathfrak{m}_z^{5}$.
\begin{lemma}
    The following facts are true:
    \begin{itemize}
        \item [(a)]The formal scheme $\Ms$ is formally smooth over $\ofb$ of relative dimension $n$. The formal scheme $\M$ is regular of total dimension $n$.
        \item[(b)] The multiplicity of the exceptional divisor $\exc_L$ on $\M$ in the principal divisor $\textup{div}(\varpi)$ on $\M$ is 2. 
        \item[(c)] Let $z_0\in\CN^{\textup{nfs}}_L$ be a point. Then $i_L:\M\rightarrow\Ms$ induces the following closed immersion:
        \begin{equation*}
            \pi_L^{-1}(z_0)\simeq\mathbb{P}^{n-1}_{\mathbb{F}}\hookrightarrow\pis^{-1}(z_0)\simeq\mathbb{P}^{n}_{\mathbb{F}}.
        \end{equation*}
    \end{itemize}  
    \label{geo-blowup}
\end{lemma}
\begin{proof}
    The morphism $\pi_L:\M\rightarrow\CN_L$ (resp. $\pi_{L^{\sharp}}:\mathcal{M}_{L^{\sharp}}\rightarrow\CN_{L^{\sharp}}$) is an isomorphism over $\CN_L\backslash\CN^{\textup{nfs}}_L$, we only need to study the local ring of the formal scheme $\CM$ (resp. $\CM_{L^{\sharp}}$) at a point $z\in\pi_L^{-1}(\CN_L^{\textup{nfs}})$ (resp. $z\in\pi_{L^{\sharp}}^{-1}(\CN_L^{\textup{nfs}})$). The formulas (\ref{local-piece-smooth}) and (\ref{local-piece}) give an explicit description of them, from which (a) and (b) are obvious. Part (c) can also be proved by combining the two formulas (\ref{local-piece-smooth}) and (\ref{local-piece}).
\end{proof}
The exceptional divisor $\exc_L$ (resp. $\exc_{L^{\sharp}}$) is a disjoint union of projective spaces $\mathbb{P}_{\mathbb{F}}^{n-1}$ (resp. $\mathbb{P}_{\mathbb{F}}^{n}$) indexed by points in $\CN_L^{\textup{nfs}}$. Let $S\subset\CN_L^{\textup{nfs}}$ be a point, denote by $\exc_{L,S}\subset\M$ (resp. $\exc_{L^{\sharp},z}\subset\Ms$) the union of projective spaces $\pi_{L}^{-1}(z)\subset\M$ (resp. $\pis^{-1}(z)\subset\Ms$) indexed by the set $S$. For integer $k$ and non-negative integer $m$, denote by $\mathcal{O}_{\exc_{L,S}}(k)$ (resp. $\mathcal{O}_{\exc_{L^{\sharp},S}}(k)$) the line bundle on $\exc_{L,S}$ (resp. $\exc_{L^{\sharp},S}$) which corresponds to the line bundles $\mathcal{O}_{\mathbb{P}_{\mathbb{F}}^{n-1}}(k)$ (resp. $\mathcal{O}_{\mathbb{P}_{\mathbb{F}}^{n}}(k)$) on each connected component of $\exc_{L,S}$.
\begin{lemma}
    The following facts are true:
    \begin{itemize}
        \item [(a)] We have the following identity in the group $\textup{Gr}^{2}K_0^{\exc_{L^{\sharp}}}(\Ms)\simeq\textup{Pic}(\exc_{L^{\sharp}})$:
    \begin{equation*}
        [\mathcal{O}_{\exc_{L^{\sharp}}}\otimes_{\mathcal{O}_{\Ms}}^{\BL}\mathcal{O}_{\exc_{L^{\sharp}}}]=\mathcal{O}_{\exc_{L^{\sharp}}}(-1).
    \end{equation*}
        \item[(b)] Let $x\in\mathbb{V}^{\sharp}$ be a nonzero element. We have the following identity in the group $\textup{Gr}^{2}K_0^{\exc_{L^{\sharp}}}(\Ms)\simeq\textup{Pic}(\exc_{L^{\sharp}})$:
    \begin{equation*}
        [\mathcal{O}_{\mathcal{Z}^{\sharp,\textup{bl}}(x)}\otimes_{\mathcal{O}_{\Ms}}^{\BL}\mathcal{O}_{\exc_{L^{\sharp}}}]=\mathcal{O}_{\exc_{L^{\sharp}}}(0)
    \end{equation*}
    \end{itemize}\label{int-spe-exc-unramified}
\end{lemma}
\begin{proof}
We first prove (a). We only need to prove this locally at every point $z_0\in\CN^{\textup{nfs}}_L$. The element $[\mathcal{O}_{\exc_{L^{\sharp}}}\otimes^{\mathbb{L}}_{\mathcal{O}_{\Ms}}\mathcal{O}_{\exc_{L^{\sharp}}}]$ is given by restricting the line bundle corresponding to $\exc_{L^{\sharp}}$ on $\mathcal{M}_{L^{\sharp}}$ to $\exc_{L^{\sharp}}$. Using the open cover (\ref{local-piece-smooth}) of $\widehat{\mathcal{M}}_{L^{\sharp},z_0}$, we have
\begin{equation*}
     \begin{tikzcd}
     \widehat{\mathcal{M}}_{L^{\sharp},z_0,0} & \widehat{\mathcal{M}}_{L^{\sharp},z_0,i} & \widehat{\mathcal{M}}_{L^{\sharp},z_0,j} 
    \\
\varpi \arrow[r,shift left,"\times v_i"] \arrow[Subseteq]{u}{}
    &
t_i \arrow[l,shift left,"\times u_i"] \arrow[r,shift left, "\times u_{ji}"] \arrow[Subseteq]{u}{}
&
t_j \arrow[l, shift left, "\times u_{ij}"] \arrow[Subseteq]{u}{}   
\end{tikzcd}
 \end{equation*}
 The same transformation rule also applies to the corresponding open cover of $\exc_{L^{\sharp}}\simeq\mathbb{P}_{\mathbb{F}}^{n}$. Therefore the corresponding line bundle is $\mathcal{O}_{\mathbb{P}_{\mathbb{F}}^{n}}(-1)$.
 \par
    Now we prove (b). Let $z_0\in\CN_L^{\textup{nfs}}$ be a point. Let $f_{x}\in\widehat{\mathcal{O}}_{\CN_{L},z_0}$ the equation of the special divisor $f_x$. Denote by $\mathfrak{m}_{z_0}$ the maximal ideal of the local ring $\widehat{\mathcal{O}}_{\CN_{L},z_0}$. Then there exists an integer $n_{z_0}\geq0$ such that $f_x\in\mathfrak{m}^{n_{z_0}}\backslash\mathfrak{m}^{n_{z_0}+1}$. On each piece $\widehat{\mathcal{M}}_{L^{\sharp},z_0,i}$ of the cover (\ref{local-piece-smooth}), we found that locally at $z_0$, we have
    \begin{equation*}
        f_x=t_i^{n_{z_0}}\cdot (\tilde{f}_{x,i}+t_i\cdot g_i),
    \end{equation*}
    where $t_i$ is the local equation of the exceptional divisor. The element $\tilde{f}_{x,i}\in\ofb[\{v_i\}_{i=1}^{n}]$ if $i=0$ or $\tilde{f}_{x,i}\in\ofb[\{u_{ji}\}_{j\neq i},u_i]$ if $i\geq1$ and is a polynomial of degree $n_{z_0}$. The element $g_i\in\mathcal{O}_{\widehat{\mathcal{M}}_{L^{\sharp},z_0,i}}$. Therefore
    \begin{equation*}
        [\mathcal{O}_{\mathcal{Z}^{\sharp,\textup{bl}}(x)}\otimes_{\mathcal{O}_{\M}}^{\BL}\mathcal{O}_{\exc_{L^{\sharp}}}]=\sum\limits_{z_0\in\CN_L^{\textup{nfs}}}\left(\mathcal{O}_{\mathbb{P}_{\mathbb{F}}^{n}}(-n_{z_0})+\mathcal{O}_{\mathbb{P}_{\mathbb{F}}^{n}}(n_{z_0})\right)=\mathcal{O}_{\exc_{L^{\sharp}}}(0).
    \end{equation*}
\end{proof}

\section{Derived cycles and arithmetic intersection numbers}

\subsection{Special cycles on $\M$}\label{special-cycle-on-blow-up}
\begin{definition}\label{cycle-on-blow-up}
    Let $S\subset\BV$ be a subset. 
    \begin{itemize}
        \item $\mathcal{Z}$-cycles: Define the special cycle $\zm(S)\subset\M$ to be the closed formal subscheme
    \begin{equation*}
        \zm(S)=\CZ(S)\times_{\CN_L}\M.
    \end{equation*}
    When $S=\{x\}$ consists of one single element, denote by $\zm(x)$ the closed formal subscheme $\zm(\{x\})$.\\
        \item $\mathcal{Y}$-cycles: Define the special cycle $\mathcal{Y}(S)\subset\CN_L$ to be closed formal subscheme cut out by the condition
    \begin{equation*}
        \rho^{\textup{univ}}\circ x_0\circ x \circ(\rho^{\textup{univ}})^{-1}\subset\textup{End}(X^{\textup{univ}})
    \end{equation*}
    for all $x\in S$. Moreover, define the special cycle $\ym(S)\subset\M$ to be the closed formal subscheme
    \begin{equation*}
        \ym(S)=\mathcal{Y}(S)\times_{\CN_L}\M.
    \end{equation*}
    \par
    When $S=\{x\}$ consists of one single element, denote by $\CY(x)$ (resp. $\ym(x)$) the closed formal subscheme $\CY(\{x\})$ (resp. $\ym(\{x\})$).\\
    \item Exceptional divisors: Define $\exc_L(S)$ to be
    \begin{equation*}
        \exc_L(S)=\pi^{-1}\left(\mathcal{N}^{\textup{nfs}}\cap\mathcal{Y}(S)\right)=\exc_{L,\mathcal{N}^{\textup{nfs}}\cap\mathcal{Y}(S)}.
    \end{equation*}
    When $S=\{x\}$ consists of one single element, denote by $\exc_L(x)$ the closed formal subscheme $\exc_L(\{x\})$. Notice that $\exc_L(S)$ is a disjoint union of projective space $\mathbb{P}_{\mathbb{F}}^{n-1}$ by Lemma \ref{geo-blowup}, we use the simplified symbols $\mathcal{O}_{\exc_{L}(S)}(k)$ (resp. ${^{\mathbb{L}}\mathcal{O}}^{\otimes m}_{\exc_{L}(S)}(k)$) to denote the corresponding line bundle $\mathcal{O}_{\exc_{L,\mathcal{Y}(S)\cap\CN_L^{\textup{nfs}}}}(k)$ (resp. derived tensor of line bundles ${^{\mathbb{L}}\mathcal{O}}^{\otimes m}_{\exc_{L,\mathcal{Y}(S)\cap\CN_L^{\textup{nfs}}}}(k)$) for all integers $k\in\mathbb{Z}$.
    \end{itemize}
\end{definition}
For a point $z_0\in\CN_L(\mathbb{F})$, recall that we use $\mathbf{V}_{z_0}$ (resp. $\mathbf{V}_{z_0}^{\sharp}$) to denote the $\ofb$-module $\mathbf{V}_{\crys}(\ofb,\mathbb{F},\delta)$ (resp. $\mathbf{V}_{\crys}^{\sharp}(\ofb,\mathbb{F},\delta)$).
\begin{lemma}
    Let $x\in\BV$ be a nonzero element. Let $z_0\in\mathcal{Y}(x)(\mathbb{F})$ be a point such that $z_0\notin\CZ(x)$.
    \begin{itemize}
        \item [(a)]If $z_0\in\CN_L^{\textup{nfs}}$, we have $\varpi x_{\textup{crys},z_0}\in\mathbf{V}_{z_0}$ and
    \begin{equation*}
        \mathbb{F}\cdot\overline{\varpi x_{\textup{crys},z_0}}=\textup{Fil}^{1}\overline{\mathbf{V}}_{z_0}(=\textup{Fil}^{1}\overline{\mathbf{V}}_{z_0}^{\sharp}).
    \end{equation*}
    Moreover, let $f_{\varpi x}$ be the local equation of the special divisor $\mathcal{Z}(\varpi x)$ in the completed local ring $\widehat{\mathcal{O}}_{\CN_{L},z_0}$ of the formal scheme $\CN_{L}$ at $z_0$. Let $\mathfrak{m}_{z_0}\subset\widehat{\mathcal{O}}_{\CN_{L},z_0}$ be the maximal ideal. Then we have
    \begin{equation*}
        f_{\varpi x}\in\mathfrak{m}_{z_0}^{2}\backslash\mathfrak{m}_{z_0}^{3}.
    \end{equation*}
    \item[(b)] If $z_0\notin\CN_L^{\textup{nfs}}$, there exists an element $a\in\Of^{\times}$ such that
    \begin{equation*}
        z_0\in\CZ^{\sharp}\left(x+\frac{a}{\varpi}\cdot x_0\right)\cap\CN_L(\mathbb{F}).
        \end{equation*}
    \end{itemize}\label{from-y-to-weighted}
\end{lemma}
\begin{proof}
     For simplicity, we use $x$ to denote $x_{\textup{crys},z_0}$. Denote by $\mathbb{D}_{z_0}$ the Dieudonne module of the $p$-divisible group $X_{z_0}$. Denote by $\Phi_{z_0}$ the Frobenius-linear morphism on $\mathbf{V}_{z_0}^{\sharp}\otimes_{\ofb}\Breve{F}$ (cf. \cite[$\S$3.2.2]{Zhu23diff}). Let $l\in\mathbf{V}_{z_0}^{\sharp}$ be an element such that its image $\overline{l}$ generates the Hodge filtration $\textup{Fil}^{1}\overline{\mathbf{V}}_{z_0}^{\sharp}$. Notice that the point $z_0\in\mathcal{Y}(x)$ implies that
        \begin{equation*}
            l\circ\mathbb{D}(x_0\circ x)\circ l\in\varpi\cdot\textup{End}_{\ofb}(\mathbb{D}_{z_0}),
        \end{equation*}
    as well as $\mathbb{D}(x_0\circ x)\in\textup{End}_{\ofb}(\mathbb{D}_{z_0})$ and $z_0\in\CZ(\varpi x)$.
    \par
    Using the quadratic form $(\cdot,\cdot)$ on the lattice $\mathbf{V}_{z_0}^{\sharp}$, we have
    \begin{align*}
        l\circ\mathbb{D}(x_0\circ x)\circ l&=l\circ\mathbb{D}(x_0)\circ\mathbb{D}(x)\circ l=l\circ\mathbb{D}(x_0)\circ\left((l,x)-l\circ\mathbb{D}(x)\right)\\
        &=(l,x)\cdot l\circ\mathbb{D}(x_0)-l\circ\left((l,x_0)-l\circ\mathbb{D}(x_0)\right)\circ\mathbb{D}(x)\\
        &=(l,x)\cdot l\circ\left(\mathbb{D}(x_0)-\frac{(l,x_0)}{(l,x)}\mathbb{D}(x)\right)+q(l)\cdot\mathbb{D}(x_0\circ x).
    \end{align*}
    Since $\overline{l}$ is isotropic, we have $\nu_{\varpi}(q(l))\geq1$. The point $z_0\notin\mathcal{Z}(x)$ implies that $\nu_{\varpi}((l,x))=0$. Therefore we have
    \begin{equation}
        l\circ\left(\mathbb{D}(x_0)-\frac{(l,x_0)}{(l,x)}\mathbb{D}(x)\right)\in\varpi\cdot\textup{End}_{\ofb}(\mathbb{D}_{z_0}).
        \label{pr1}
    \end{equation}
    \par
    Let $v=x_0-c\cdot \varpi x$ where $c=\frac{(l,x_0)}{\varpi\cdot(l,x)}\in\ofb$. Since we know that $\nu_{\varpi}((l,x_0))\geq1$ and $z_0\in\CZ(\varpi x)$, we have $v\in\mathbf{V}_{z_0}^{\sharp}$. Then (\ref{pr1}) is equivalent to 
    \begin{equation*}
        \overline{l}\circ \overline{v}=0\in\textup{End}_{\mathbb{F}}(\mathbb{D}_{z_0}/\varpi\mathbb{D}_{z_0}).
    \end{equation*}
    Since $\mathbb{D}_{z_0}/\varpi\mathbb{D}_{z_0}$ is isomorphic to the Clifford algebra associated to the quadratic space $\mathbf{V}_{z_0}^{\sharp}/\varpi\mathbf{V}_{z_0}^{\sharp}$ over $\mathbb{F}$, we conclude that $\overline{v}\in\mathbb{F}\cdot\overline{l}$. Therefore there exists $b\in\ofb$ and $h\in\mathbf{V}_{z_0}^{\sharp}$ such that
    \begin{equation}
        x_0-c\cdot \varpi x=b\cdot l+\varpi\cdot h.
        \label{equation-v}
    \end{equation}
    \par
    Now we prove (a). Since $z_0\in\CN_L^{\textup{nfs}}$, we have $\mathbb{F}\cdot\overline{x_0}=\textup{Fil}^{1}\overline{\mathbf{V}}_{z_0}=\textup{Fil}^{1}\overline{\mathbf{V}}_{z_0}^{\sharp}$ by Lemma \ref{singular-criterion}, and $\nu_{\varpi}((x_0,l))=1$ by \cite[Lemma 5.3.5]{Zhu23diff}, i.e., $c\in\ofb^{\times}$. The equation (\ref{equation-v}) implies that
    \begin{equation*}
        \mathbb{F}\cdot\overline{\varpi x_{\textup{crys},z_0}}=\textup{Fil}^{1}\overline{\mathbf{V}}_{z_0}(=\textup{Fil}^{1}\overline{\mathbf{V}}_{z_0}^{\sharp}).
    \end{equation*}
    Then there exists $a\in\ofb$ and an element $x^{\prime}\in\mathbf{V}_{z_0}^{\sharp}$ such that $\varpi x=ax_0+\varpi x^{\pr}$. Since $(x_0,x)=0$, we get $(x_0,x^{\prime})\in\ofb^{\times}$.
    \par
    Let $W=\ofb[\pi]/(\pi^{2}-\varpi)$, the maximal ideal $(\pi)\subset W$ is equipped with a divided power structure $\delta_W$. For a point $z\in\widehat{\CN}_{L,z_0}(W)$, it corresponds to an $\ofb$-algebra homomorphism $z^{\sharp}:\widehat{\mathcal{O}}_{\CN_{L},z_0}\rightarrow W$. It also corresponds to an isotropic vector $l\in\mathbf{V}_{z_0}(W,\mathbb{F},\delta_W)\simeq\mathbf{V}_{z_0}\otimes_{\ofb}W$ such that $(l,x_0)=0$. Notice that $\mathbb{F}\cdot\overline{l}=\mathbb{F}\cdot\overline{x_0}$ since $z_0\in\CN^{\textup{nfs}}_L$, we have $(l,x^{\prime})\in\ofb^{\times}$. Hence $(l,x)=\varpi(l,x^{\prime})$. Therefore by Lemma \ref{deformo2} (b), we have
    \begin{equation*}
        \left(z^{\sharp}(f_{\varpi x})\right)=\left((l,x)\right)=(\varpi)\subset W.
    \end{equation*}
    Let's now assume $f_{\varpi x}\in\mathfrak{m}_{z_0}^{3}$, we should get $\left(z^{\sharp}(f_{\varpi x})\right)\subset \left(z^{\sharp}(\mathfrak{m}_{z_0}^{3})\right)\subset (\pi^{3})\subset W$, which is a contradiction. Hence we must have $f_{\varpi x}\in\mathfrak{m}_{z_0}^{2}\backslash\mathfrak{m}_{z_0}^{3}$.
    \par
    Next, we prove (b). Since $\CN_L$ is formally smooth over $\ofb$ at $z_0$, we can choose the element $l$ such that it is isotropic and $(l,x_0)=0$. Since $(l,x_0-c\cdot \varpi x)=0$ and $l$ is isotropic, we have $(l,h)=0$. Applying $\Phi_{z_0}$ to (\ref{equation-v}), we get
    \begin{equation}
        x_0-c^{\sigma}\cdot\varpi x=b^{\sigma}\cdot\Phi_{z_0}(l)+\varpi\cdot\Phi_{z_0}(h).
        \label{equation-v-2}
    \end{equation}
    By \cite[Lemma 3.2.4]{Zhu23diff}, we have $\Phi_{z_0}(l)\in\varpi\cdot\mathbf{V}_{z_0}^{\sharp}$. We also have $\Phi_{z_0}(h)\in\mathbf{V}_{z_0}^{\sharp}$ by \cite[Lemma 3.2.5]{Zhu23diff} since $(l,h)=0$. Therefore $\overline{x_0}=\overline{c^{\sigma}}\cdot\overline{\varpi x}$ in $\mathbf{V}_{z_0}^{\sharp}/\varpi\mathbf{V}_{z_0}^{\sharp}$. If $c^{\sigma}\neq c$, we have $\overline{x_0}\in\mathbb{F}\cdot\overline{l}$, which is a contradiction by Lemma \ref{singular-criterion}. Therefore $c^{\sigma}=c$, i.e., $c\in\mathcal{O}_F$. Then (\ref{equation-v-2}) implies that 
    \begin{equation*}
        x_0-c\cdot\varpi x\in\varpi\mathbf{V}_{z_0}^{\sharp},\,\,\textup{i.e.,}\,\,x-\frac{c^{-1}}{\varpi}\cdot x_0\in \mathbf{V}_{z_0}^{\sharp}.
    \end{equation*}
    Then we have $z_0\in\CZ^{\sharp}\left(x-\frac{c^{-1}}{\varpi}\cdot x_0\right)$ by \cite[Proposition 5.3.3]{Zhu23diff}.
\end{proof}

\subsection{non-formally smooth points on $\mathcal{Y}$-cycles}
Let $z_0\in\CN_L^{\textup{nfs}}(\mathbb{F})$ be a point. By Lemma \ref{cor-ver-1-singular}, there exists a unique vertex lattice $\mathcal{L}\subset\mathbb{V}$ of type $1$ such that $\{z_0\}=\mathcal{Z}(\mathcal{L})$.
\begin{lemma}
    Let $\mathcal{L}\subset\mathbb{V}$ be a vertex lattice of type $1$. Let $z_\mathcal{L}$ be the unique $\mathbb{F}$-point of $\mathcal{Z}(\mathcal{L})$. Let $x\in\mathbb{V}$ be an element. The point $z_\mathcal{L}\in\mathcal{Y}(x)$ if and only if $x\in\mathcal{L}^{\vee}$.\label{singular-on-Y} 
\end{lemma}
\begin{proof}
For simplicity, denote by $\mathbb{D}_{\mathcal{L}}$ the Dieudonne module of the $p$-divisible group $X_{z_\mathcal{L}}$. Since $z_\mathcal{L}$ is a non-formally smooth point, the filtration $\textup{Fil}^{1}\overline{\mathbf{V}}_{z_{\mathcal{L}}}\subset\overline{\mathbf{V}}_{z_\mathcal{L}}$ is generated by the element $\overline{x_{0,z_{\mathcal{L}}}}$ by Lemma \ref{singular-criterion}.
\par
    Let's first assume that $z_\mathcal{L}\in\mathcal{Y}(x)$. We want to show that $x\in\mathcal{L}^{\vee}$. If $x\notin\mathcal{L}^{\vee}$, we get $z_\mathcal{L}\notin\mathcal{Z}(x)$ since $\{z_\mathcal{L}\}=\mathcal{Z}(\mathcal{L})$. By Lemma \ref{from-y-to-weighted} (a), we get 
    \begin{equation*}
        \textup{Fil}^{1}\overline{\mathbf{V}}_{z_{\mathcal{L}}}=\mathbb{F}\cdot\overline{x_{0,z_{\mathcal{L}}}}=\mathbb{F}\cdot\overline{\varpi x}.
    \end{equation*}
    Therefore there exists an element $a\in\ofb^{\times}$ such that $\varpi x=ax_{0,z_{\mathcal{L}}}+\varpi v$ for some $v\in\mathbf{V}_{z_\mathcal{L}}$. The fact that $z_{\mathcal{L}}\in\mathcal{Z}(\mathcal{L})$ implies that $\mathcal{L}\subset\mathbf{V}_{z_{\mathcal{L}}}$. Therefore for an element $\mathcal{L}\in\mathcal{L}$, we have
    \begin{equation*}
        (\varpi x, \mathcal{L})=(ax_{0,z_{\mathcal{L}}}+\varpi v,\mathcal{L})=\varpi(v,\mathcal{L})\subset(\varpi).
    \end{equation*}
    Therefore we have $(x,\mathcal{L})\in\mathcal{O}_F$, hence $x\in\mathcal{L}^{\vee}$.
    \par
    Now we assume that $x\in\mathcal{L}^{\vee}$, we want to show that $z_\mathcal{L}\in\mathcal{Y}(x)$. If $x\in\mathcal{L}$, we have $z_\mathcal{L}\in\mathcal{Z}(\mathcal{L})\subset\mathcal{Z}(x)\subset\mathcal{Y}(x)$, so we assume that $x\in\mathcal{L}^{\vee}\backslash\mathcal{L}$ from now on. Let $f\in\mathbf{V}_{z_{\mathcal{L}}}^{\sharp}$ be an element such that $(f,x_0)=1$. Then $\mathbf{V}_{z_{\mathcal{L}}}^{\sharp}=\langle f, x_0\rangle\obot\mathbf{V}_0$ where $\mathbf{V}_0$ is a self-dual $\ofb$-lattice of rank $n-1$. Moreover,
    \begin{equation}
        \mathbf{V}_{z_\mathcal{L}}=\mathbf{V}_0\obot\ofb\cdot\left(2q(x_0)f-x_0\right),\,\,\mathbf{V}_{z_\mathcal{L}}^{\vee}=\mathbf{V}_0\obot\ofb\cdot\left(f-\frac{x_0}{2q(x_0)}\right).\label{exp-vz}
    \end{equation}
    Notice that $\mathbf{V}_{z_\mathcal{L}}$ is a vertex lattice of type $1$ in the $\breve{F}$-space $\mathbf{V}_{z_\mathcal{L}}\otimes_{\ofb}\Breve{F}\simeq\mathbb{V}\otimes_{F}\Breve{F}$ and $\mathcal{L}\subset\mathbf{V}_{z_\mathcal{L}}$, we must have $\mathbf{V}_{z_\mathcal{L}}=\mathcal{L}\otimes_{\mathcal{O}_F}\ofb$ and hence $\mathbf{V}_{z_{\mathcal{L}}}^{\vee}/\mathbf{V}_{z_\mathcal{L}}\simeq\mathcal{L}^{\vee}/\mathcal{L}\otimes_{\mathbb{F}_p}\mathbb{F}$. Therefore
    \begin{equation*}
        x=a q(x_0)^{-1}\cdot x_0+v\,\,\,\,\textup{for an element $a\in\ofb^{\times}$ and $v\in\mathbf{V}_{z_\mathcal{L}}$}.
    \end{equation*}
    Notice that
    \begin{align*}
        \mathbb{D}(x_0x)\cdot \mathbb{D}(x_0)\mathbb{D}_{\mathcal{L}}=-\mathbb{D}(x)\cdot\mathbb{D}(x_0^{2})\mathbb{D}_\mathcal{L}=-(a\mathbb{D}(x_0)+q(x_0)v)\mathbb{D}_\mathcal{L}\subset\mathbb{D}(x_0)\mathbb{D}_\mathcal{L}+\varpi\mathbb{D}_\mathcal{L},
    \end{align*}
    \begin{align*}
        \mathbb{D}(x_0x)\cdot \varpi\mathbb{D}_{\mathcal{L}}=(a+x_0v)\cdot\varpi\mathbb{D}_{\mathcal{L}}\subset\mathbb{D}(x_0)\mathbb{D}_\mathcal{L}+\varpi\mathbb{D}_\mathcal{L}.
    \end{align*}
    Therefore the quasi-endomorphism $\mathbb{D}(x_0x)$ preserves the filtration $\textup{Fil}^1\mathbb{D}_\mathcal{L}=\mathbb{D}(x_0)\mathbb{D}_\mathcal{L}+\varpi\mathbb{D}_\mathcal{L}\subset\mathbb{D}_\mathcal{L}$. Therefore $x_0x$ is an endomorphism of $X_{z_\mathcal{L}}$, i.e., $z_\mathcal{L}\in\mathcal{Y}(x)$.
\end{proof}

\subsection{Weighted cycles}
Recall that we have defined an isomorphism between $\Of$-modules $r:\Lambda^{\vee}/\Lambda\xrightarrow{\sim}L^{\vee}/L$, where $\Lambda$ is a rank $1$ $\Of$-quadratic lattice generated by $x_0$. Let $\eta:\Lambda^{\vee}/\Lambda:L^{\vee}/L\rightarrow\Lambda^{\vee}$ be an arbitrary lift of the map $r^{-1}:L^\vee/L \cong \Lambda^\vee/\Lambda$.
\begin{definition}\label{def KR}
For a non-zero $x \in \BV$ and $\mu \in L^\vee/L$, define the weighted cycle $\CZ(x,\mu)$ on $\CN_L$ associated to the pair $(x,\mu)$ to be the closed formal subscheme cut out by the condition
    \begin{equation}
        \rho^{\textup{univ}}\circ (x+\eta(\mu)) \circ(\rho^{\textup{univ}})^{-1}\subset\textup{End}(X^{\textup{univ}})\label{weighted-moduli}
    \end{equation}
Define the weighted cycle $\zm(x,\mu)$ on $\M$ associated to the pair $(x,\mu)$ to be
\begin{equation*}
    \zm(x,\mu)=\CZ(x,\mu)\times_{\CN_L}\M.
\end{equation*}
\end{definition}

\begin{remark}
    This definition is independent of the choice of the map $\eta$ because different choices differ by a multiple of $x_0$, then the condition (\ref{weighted-moduli}) is unchanged because $x_0$ is always an isogeny on $X^{\textup{univ}}$. It's easy to see that under the closed immersion $i_L:\CN_L\rightarrow\CN_{L^{\sharp}}$,
\begin{equation*}
    \CZ(x,\mu)=\CZ^{\sharp}(x+\eta(\mu))\cap\CN_L.
\end{equation*}
Therefore the weighted cycle $\CZ(x, \mu) \to \CN_L$ (resp. $\zm(x, \mu) \to \M$) is a Cartier divisor on $\CN_L$ (resp. $\M$). See also \cite[Proposition 6.5.2.]{HM20} for the global analog. We may regard the above construction as a local analog of the subset $V_\mu(\mathcal{A}) \subseteq V(\mathcal{A})$ in \cite[(6.4.9)]{HM20}.
\end{remark}
\begin{remark}
    Lemma \ref{from-y-to-weighted} can be re-interpreted in the following way: Let $x\in\BV$ be a nonzero element. Let $z_0\in\mathcal{Y}(x)(\mathbb{F})$ be a point such that $z_0\notin\CN_L^{\textup{nfs}}$. Then there exists an element $\mu\in L^{\vee}/L$ such that
    \begin{equation*}
        z_0\in\CZ(x,\mu)(\mathbb{F}).
    \end{equation*}
    \label{from-y-to-weighted-2}
\end{remark}

\subsection{Special divisors and the exceptional divisor}
\begin{lemma}
    The following facts are true:
    \begin{itemize}
        \item [(a)] We have the following identity in the group $\textup{Gr}^{2}K_0^{\exc_{L}}(\M)\simeq\textup{Pic}(\exc_{L})$:
    \begin{equation*}
        [\mathcal{O}_{\exc_{L}}\otimes_{\mathcal{O}_{\M}}^{\BL}\mathcal{O}_{\exc_{L}}]=\mathcal{O}_{\exc_{L}}(-1).
    \end{equation*}
        \item[(b)] Let $x\in\mathbb{V}$ be a nonzero element and $\mu\in L^{\vee}/L$. We have the following identity in the group $\textup{Gr}^{2}K_0^{\exc_{L}}(\M)\simeq\textup{Pic}(\exc_{L})$:
    \begin{equation*}
        [\mathcal{O}_{\mathcal{Z}^{\textup{bl}}(x,\mu)}\otimes_{\mathcal{O}_{\M}}^{\BL}\mathcal{O}_{\exc_{L}}]=\mathcal{O}_{\exc_{L}}(0).
    \end{equation*}
    \end{itemize}
    \label{inter-spe-exc-M}
\end{lemma}
\begin{proof}
    The proof of (a) is similar to that of Lemma \ref{int-spe-exc-unramified} (a), so we omit it. Let's now prove (b). Notice that there is an isomorphism $\iota:\M\xrightarrow{\sim}\widetilde{\mathcal{Z}}^{\sharp,\textup{bl}}(x_0)\subset\mathcal{M}_{L^{\sharp}}$. Therefore
    \begin{align*}
        [\mathcal{O}_{\mathcal{Z}^{\textup{bl}}(x,\mu)}\otimes_{\mathcal{O}_{\M}}^{\BL}\mathcal{O}_{\exc_{L}}]&=[\mathcal{O}_{\mathcal{Z}^{\sharp,\textup{bl}}(x+\eta^{-1}(\mu))}\otimes_{\mathcal{O}_{\Ms}}^{\BL}\mathcal{O}_{\exc_{L^{\sharp}}}\otimes_{\mathcal{O}_{\Ms}}^{\BL}\mathcal{O}_{\widetilde{\mathcal{Z}}^{\sharp,\textup{bl}}(x_0)}]\\
        &\overset{\textup{Lemma}\,\ref{int-spe-exc-unramified} \textup{(b)}}{=}[\mathcal{O}_{\exc_{L^{\sharp}}}(0)\otimes_{\mathcal{O}_{\Ms}}^{\BL}\mathcal{O}_{\widetilde{\mathcal{Z}}^{\sharp,\textup{bl}}(x_0)}]=\mathcal{O}_{\exc_{L}}(0).
    \end{align*}
\end{proof}

\subsection{Decomposition of $\mathcal{Y}$-cycles}
Our goal in this subsection is to prove the following proposition:
\begin{proposition}
    Let $x\in\BV$ be a nonzero element. Then $\ym(x)\subset\M$ is an effective Cartier divisor. Moreover, we have the following equality between Cartier divisors on $\M$:
    \begin{equation*}
        \ym(x)=\sum\limits_{\mu\in\Lambda^{\vee}/\Lambda}\zm(x,\mu)+\exc_L(x).
    \end{equation*}\label{decom-Y}
\end{proposition}
\begin{proof}
    Notice that both statements in the proposition are local. Let $z\in\M(\mathbb{F})$ be a point. It's easy to check that the equality holds if $z\notin\mathcal{Y}^{\textup{bl}}(x)$. Therefore we will work locally around a point $z\in\ym(y)(\mathbb{F})\subset\M(\mathbb{F})$. Let $z_0=\pi(z)\in\CN_L(\mathbb{F})$. It's easy to see that $z_0\in\mathcal{Y}(x)$, hence $z_0\in\CZ(\varpi x)$ by the definition of $\mathcal{Y}$-cycles. Let $\widehat{\CN}_{L,z_0}$ and $\widehat{\mathcal{M}}_{L,z}$ be the completion of the formal schemes $\CN_L$ and $\M$ at the points $z_0$ and $z$ respectively. For simplicity, denote by $R_0=\mathcal{O}_{\widehat{\CN}_{L,z_0}}$, denote by $R_1=\mathcal{O}_{\widehat{\mathcal{M}}_{L,z}}$. Denote by $\mathfrak{m}_{0}$ and $\mathfrak{m}_{1}$ the maximal ideal of $R_0$ and $R_1$ respectively. Denote by $\mathcal{I}_{\exc}\subset R_1$ the ideal of the exceptional divisor.
    \par
    Let $\mathbf{V}$ (resp. $\mathbf{V}^{\sharp}$) be the $R_0$-module $\mathbf{V}_{\crys}(R_0,R_0,\delta)$ (resp. $\mathbf{V}_{\crys}^{\sharp}(R_0,R_0,\delta)$) where $\widetilde{z}_0:\textup{Spf}\,R_0\rightarrow\CN_L$ is the canonical morphism. Let $X_{\tilde{z}_0}$ be the base change of the universal $p$-divisible group to $R_0$ by the morphism $\tilde{z}_0$. Let $S\in\textup{Alg}_{\ofb}$. For a point $\tilde{z}\in\widehat{\CN}_{L,z_0}(S)$, let $X_{\tilde{z}}=X_{\tilde{z}_0}\times_{\widehat{\CN}_{L,z_0}}S$. Let $\mathbb{D}(X_{\tilde{z}})(S)$ be the Dieudonne crystal of the $p$-divisible group $X_{\tilde{z}}$ evaluated at $S$. Let $\boldsymbol{l}\in\mathbf{V}$ be a generator of the filtration $\textup{Fil}^{1}\mathbf{V}$. Then the Hodge filtration of $X_{\tilde{z}}$ is given by
    \begin{equation*}
        0\rightarrow\textup{Im}(\boldsymbol{l})\otimes_{R_0}S\rightarrow\mathbb{D}(X_{\tilde{z}})(S).
    \end{equation*}
    \noindent$\bullet$ Case 1: The point $z_0\in\CZ(x)(\mathbb{F})$. Let $\mathcal{I}_x\subset R_0$ (resp. $\mathcal{I}_{x_0x}\subset R_0$) be the ideal sheaf of the special cycle $\mathcal{Z}(x)$ (resp. $\mathcal{Y}(x)$) at $z_0$. Denote by $\mathcal{I}^{\pr}_{x}=\mathcal{I}_{x}\cdot R_1$ and $\mathcal{I}^{\pr}_{x_0x}=\mathcal{I}_{x_0x}\cdot R_1$. Let's consider the thickening $R_0/\mathcal{I}_x^{2}\rightarrow R_0/\mathcal{I}_x$ which admits a nilpotent divided power structure. The Hodge filtration of the base change of the $p$-divisible group $X_{\tilde{z}_0}$ to $R_0/\mathcal{I}_x^{2}$ is given by
    \begin{equation*}
        0\rightarrow\textup{Im}(\boldsymbol{l})\otimes_{R_0}R_0/\mathcal{I}_x^{2}\rightarrow\mathbb{D}(X_{\tilde{z}_0})(R_0)\otimes_{R_0}R_0/\mathcal{I}_x^{2}\simeq\mathbb{D}(X_{\tilde{z}_0}\times_{R_0}R_0/\mathcal{I}_x^{2})(R_0/\mathcal{I}_x^{2}).
    \end{equation*}
    The endomorphism $\mathbb{D}(x_0\circ x)(R_0/\mathcal{I}_x^{2})$ preserves the above Hodge filtration if and only if
    \begin{equation*}
        \boldsymbol{l}\circ\mathbb{D}(x_0\circ x)(R_0/\mathcal{I}_x^{2})\circ \boldsymbol{l}=0\in\textup{End}_{R_0/\mathcal{I}_x^{2}}(\mathbb{D}(X_{\tilde{z}_0})(R_0)\otimes_{R_0}R_0/\mathcal{I}_x^{2}).
    \end{equation*}
    By the definition of the element $\boldsymbol{l}$, we have
    \begin{align}
        \boldsymbol{l}\,\circ&\,\mathbb{D}(x_0\circ x)(R_0/\mathcal{I}_x^{2})\circ \boldsymbol{l}=\boldsymbol{l}\circ\mathbb{D}(x_0)\circ\mathbb{D}(x)\circ \boldsymbol{l}=\boldsymbol{l}\circ\mathbb{D}(x_0)\circ\left((\boldsymbol{l},x)-\boldsymbol{l}\circ\mathbb{D}(x)\right)\label{ldl}\\
        &=(\boldsymbol{l},x)\cdot\boldsymbol{l}\circ\mathbb{D}(x_0)-\boldsymbol{l}\circ\mathbb{D}(x_0)\circ \boldsymbol{l}\circ\mathbb{D}(x)=(\boldsymbol{l},x)\cdot \boldsymbol{l}\circ\mathbb{D}(x_0)+\boldsymbol{l}\circ \boldsymbol{l}\circ\mathbb{D}(x_0)\circ\mathbb{D}(x)\notag\\
        &=(\boldsymbol{l},x)\cdot\boldsymbol{l}\circ\mathbb{D}(x_0).\notag
    \end{align}
    Notice that the above computation also apply to the thickening $R_1/\mathcal{I}_x^{\pr2}\rightarrow R_1/\mathcal{I}_x^{\pr}$. There are two different cases:
    \begin{itemize}
        \item [1a.] The point $z_0\notin\CN_L^{\textup{nfs}}$. Then we have $R_0\simeq R_1$. 
    If we pick a basis for the $R_0/\mathcal{I}_x^{2}$-module $\mathbb{D}(X_{\tilde{z}_0})(R_0)\otimes_{R_0}R_0/\mathcal{I}_x^{2}$, the endomorphism $\boldsymbol{l}\circ\mathbb{D}(x_0)$ is a matrix which contains invertible coefficients by Corollary \ref{smooth-criterion2}. Therefore $\boldsymbol{l}\circ\mathbb{D}(x_0\circ x)(R_0/\mathcal{I}_x^{2})\circ \boldsymbol{l}=0$ is equivalent to $(\boldsymbol{l},x)=0$. Hence
    \begin{equation*}
        \left(\mathcal{I}_{x_0x}+\mathcal{I}_x^{2}\right)\big/\mathcal{I}_x^{2}\simeq\mathcal{I}_x/\mathcal{I}_x^{2}
    \end{equation*}
    by Lemma \ref{deformo2}(ii). Notice that $\mathcal{I}_{x_0x}\subset\mathcal{I}_x$ since $\mathcal{Z}(x)\subset\mathcal{Y}(x)$, hence $\mathcal{I}_{x_0x}=\mathcal{I}_x$. Therefore the two statements of the proposition are true in this case.
    \item [1b.] The point $z_0\in\CN_L^{\textup{nfs}}$. We use the coordinate fixed in Remark \ref{coordinate}. Then
    \begin{equation*}
        \boldsymbol{l}\circ\mathbb{D}(x_0)=aq(x_0)(1-2x_1\circ x_0)+\sum\limits_{i=1}^{n}t_i\cdot l_i\circ x_{0,z_0}
    \end{equation*}
    for some invertible element $a\in\widehat{\mathcal{O}}_{\CN_{L},z_0}^{\times}$.
    The fiber product $\M\times_{\CN_L}\textup{Spf}\,\widehat{\mathcal{O}}_{\CN_L,z_0}$ is covered by $n$ open formal subschemes $\{\widehat{\mathcal{M}}_{L,z_0,i}\}_{i=1}^{n}$ in (\ref{local-piece}). Suppose that $z\in\widehat{\mathcal{M}}_{L,z_0,i}(\mathbb{F})$ for some $i$. Over the formal scheme $\widehat{\mathcal{M}}_{L,z_0,i}$, we have 
    \begin{equation*}
        \boldsymbol{l}\circ\mathbb{D}(x_0)=t_i\cdot\left(l_i\circ x_{0,z_0}+\sum\limits_{j\neq i}u_{ji}\cdot l_j\circ x_{0,z_0}+\tilde{a}\cdot t_i\right)
    \end{equation*}
    where $\tilde{a}\in\mathcal{O}_{\widehat{\mathcal{M}}_{L,z_0,i}}$. If we pick a basis for the $R_1/\mathcal{I}^{\pr2}_x$-module $\mathbb{D}(X_{\tilde{z}_0})(R_0)\otimes_{R_0}R_1/\mathcal{I}^{\pr2}_x$, the endomorphism $\boldsymbol{l}\circ\mathbb{D}(x_0)=t_i\cdot N$ where $N$ is a matrix containing an invertible entry since $\overline{l}_i$ and $\overline{x_{0,z_0}}$ are linearly independent in the $\mathbb{F}$-vector space $\mathbf{V}^{\sharp}_{z_0}/\varpi\mathbf{V}^{\sharp}_{z_0}$. Therefore
    \begin{equation*}
        \left(\mathcal{I}_{x_0x}^{\pr}+\mathcal{I}_x^{\pr2}\right)\big/\mathcal{I}_x^{\pr2}\simeq\left(\mathcal{I}_x^{\pr}\cdot\mathcal{I}_{\exc}\right)/\mathcal{I}_x^{\pr2}.
    \end{equation*}
    Notice that $\mathcal{I}_{x_0x}^{\pr}\subset\mathcal{I}_x^{\pr}$ and $\mathcal{I}_{\exc}\subset\mathcal{I}_x^{\pr}$, hence we have $\mathcal{I}_{x_0x}^{\pr}=\mathcal{I}_x^{\pr}\cdot\mathcal{I}_{\exc}$. Therefore the two statements of the proposition are true in this case.
    \end{itemize}
    \par
    \noindent$\bullet$ Case 2: The point $z_0\notin\CZ(x)(\mathbb{F})$. Recall that we use $\mathbf{V}_{z_0}$ (resp. $\mathbf{V}_{z_0}^{\sharp}$) to denote the $\ofb$-module $\mathbf{V}_{\crys}(\ofb,\mathbb{F},\delta)$ (resp. $\mathbf{V}_{\crys}^{\sharp}(\ofb,\mathbb{F},\delta)$). We use $\mathbb{D}_{z_0}$ to denote the Dieudonne module of the $p$-divisible group $X_{z_0}$. Let $l\in\mathbf{V}_{z_0}^{\sharp}$ be an element such that its image $\overline{l}$ generates the Hodge filtration $\textup{Fil}^{1}\overline{\mathbf{V}}_{z_0}^{\sharp}$. Notice that the point $z_0\in\mathcal{Y}_{\mathcal{M}}(x)$ if and only if
        \begin{equation*}
            l\circ\mathbb{D}(x_0\circ x)\circ l\in\varpi\cdot\textup{End}_{\ofb}(\mathbb{D}_{z_0}).
        \end{equation*}
        By similar computations as in (\ref{ldl}), we get
        \begin{equation*}
            l\circ\mathbb{D}(x_0\circ x)\circ l=(l,x)\cdot l\circ\mathbb{D}(x_0)+q(l)\cdot\mathbb{D}(x_0\circ x).
        \end{equation*}
        By Lemma \ref{deformo2}(ii), we have $(l,x)\in\ofb^{\times}$ since $z_0\notin\CZ(x)(\mathbb{F})$. By Corollary \ref{smooth-criterion2}, we know that the endomorphism $l\circ\mathbb{D}(x_0)\in\varpi\cdot\textup{End}_{\ofb}(\mathbb{D}_{z_0})$ if and only if $z_0\in\CN_L^{\textup{nfs}}$. Therefore we consider the following two cases:
    \begin{itemize}
        \item [2a.] The point $z_0\notin\CN_L^{\textup{nfs}}$. Then there exists $\mu\in L^{\vee}/L$ such that $z_0\in\CZ(x,\mu)$. We can replace $x$ by $x+r^{-1}(\mu)$ in (\ref{ldl}) and use the exact the same arguments in case (1a) to deduce that locally around the point $z_0\in\CN_L(\mathbb{F})$, we have
        \begin{equation*}
            \mathcal{Y}(x)=\mathcal{Z}(x,\mu).
        \end{equation*}
        This implies that locally around the point $z\in\M(\mathbb{F})$, we have
        \begin{equation*}
             \mathcal{Y}^{\textup{bl}}(x)=\mathcal{Z}^{\textup{bl}}(x,\mu).
        \end{equation*}
        \item [2b.] The point $z_0\in\CN_L^{\textup{nfs}}$. We know that $z_0\in\mathcal{Y}(x)$ by the above arguments. Hence $\exc_L\subset\mathcal{Y}_{\mathcal{M}}(x)$ on the fiber product $\M\times_{\CN_L}\textup{Spf}\,\widehat{\mathcal{O}}_{\CN_L,z_0}=\bigcup\limits_{i=1}^{n}\widehat{\mathcal{M}}_{L,z_0,i}$. We will now show we actually have an equality $\mathcal{Y}_{\mathcal{M}}(x)=\exc_L$ on each formal scheme $\widehat{\mathcal{M}}_{L,z_0,i}$. Let $\mathcal{O}_i$ be the structure ring of the affine formal scheme $\widehat{\mathcal{M}}_{L,z_0,i}$. Let $\overline{\mathcal{O}_i}=\mathcal{O}_i/(t_i)$. By (\ref{local-piece}), we have
        \begin{equation}
            \overline{\mathcal{O}_i}\simeq\mathbb{F}[\{u_{ji}\}_{j\neq i}]. \label{i-th coordinate}
        \end{equation}
        \par
        We consider the divided power thickening $R\coloneqq R_0/\mathcal{I}_{x_0x}^{ p}\rightarrow \overline{R}\coloneqq R_0/\mathcal{I}_{x_0x}$. Let $\mathcal{O}_i^{\pr}=\mathcal{O}_i/\mathcal{I}_{x_0x}^{p}\cdot\mathcal{O}_i\simeq R\otimes_{R_0}\mathcal{O}_i$. Fix a $R$-basis $\boldsymbol{e}=\{e_i\}_{i=1}^{2^{n+2}}$ of the Dieudonne crystal $\mathbb{D}(X_{\tilde{z}_0})\otimes_{R_0}R$. Let $M\in\textup{M}_{2^{n+2}}(R)$ be the matrix of the endomorphism $\boldsymbol{l}\circ\mathbb{D}(x_0x)(R)\circ\boldsymbol{l}$ under the basis $\boldsymbol{e}$. Let $f_{\varpi x}\in R_0$ be an element defining the divisor $\mathcal{Z}(\varpi x)$. Then $t_i\cdot f_{\varpi x}\in \mathcal{O}_i$ is the equation of the divisor $\mathcal{Y}_{\mathcal{M}}(\varpi x)$ (see the case (1b)). Therefore there exists a matrix $N\in\textup{M}_{2^{n+2}}(\mathcal{O}_i^{\pr})$ which contains an invertible entry such that, under the basis $\boldsymbol{e}$,
        \begin{equation*}
            \boldsymbol{l}\circ\mathbb{D}(\varpi x_0x)(\mathcal{O}_i^{\pr})\circ\boldsymbol{l}=\overline{t_i\cdot f_{\varpi x}}\cdot N =\varpi\cdot M\in\textup{M}_{2^{n+2}}(\mathcal{O}_i^{\pr}),
        \end{equation*}
        Then there exists an entry $\overline{m}\in R$ of the matrix $M$ such that $\overline{t_i\cdot f_{\varpi x}}=\overline{u\cdot \varpi\cdot m}\in\mathcal{O}_i^{\pr}$ where $u\in \mathcal{O}_i^{\times}$ is invertible. Therefore there exists an element $s\in\mathcal{I}_{x_0x}^{p}\cdot\mathcal{O}_i$ such that
        \begin{equation}
            t_i\cdot f_{\varpi x}=u\cdot \varpi\cdot m+s\in \mathcal{O}_i,
            \label{equ-val}
        \end{equation}
        where $m\in \mathfrak{m}_0\subset R_0$ have image $\overline{m}\in R$. By Lemma \ref{geo-blowup} (b), we have $\varpi=t_i^{2}\cdot g_i$ where $g_i\in\mathcal{O}_i$ and $\overline{g_i}\in\overline{\mathcal{O}_i}$ is a polynomial of degree $2$ under the isomorphism (\ref{i-th coordinate}). Let $m_i=t_i^{-1}m$ and $\tilde{s}_1=t_i^{-p}s_1\in\mathcal{O}_i$. Then (\ref{equ-val}) becomes
        \begin{equation}
            f_{\varpi x}=t_i^{2}\cdot\left(u\cdot g_i\cdot m_i+t_i^{p-3}\tilde{s}_1\right).\label{equ-val-2}
        \end{equation}
        Since $p>3$, the image of $\left(u\cdot g_i\cdot m_i+t_i^{p-3}\tilde{s}_1\right)\in\mathcal{O}_i$ in the ring $\overline{\mathcal{O}_i}$ is $\overline{u}\cdot\overline{g_i}\cdot\overline{m_i}$. It should be a degree $2$ polynomial under the isomorphism (\ref{i-th coordinate}) by Lemma \ref{from-y-to-weighted} (a). The polynomial $\overline{g_i}$ already has degree $2$, we conclude that $\overline{m_i}$ is a nonzero constant. Hence $m=u^{\prime}\cdot t_i$ for some invertible element $u^{\prime}\in \mathcal{O}_i$. Therefore we have the following identity of matrices
       \begin{equation*}
           M=\overline{t_i}\cdot M^{\prime}\in\textup{M}_{2^{n+2}}(\mathcal{O}_i^{\prime}),
       \end{equation*}
       where the matrix $M^{\pr}\in\textup{M}_{2^{n+2}}(\mathcal{O}_i^{\prime})$ contains an invertible element. Hence $\mathcal{I}_{x_0x}^{\prime}\vert_{\widehat{\mathcal{M}}_{L,z_0,i}}=(t_i)=\mathcal{I}_{\exc}\vert_{\widehat{\mathcal{M}}_{L,z_0,i}}$. Therefore $\mathcal{I}_{x_0x}^{\prime}=\mathcal{I}_{\exc}$.
    \end{itemize}
    
\end{proof}

\subsection{Derived $\mathcal{Z}$-cycles}
\begin{lemma}
    Let $x\in\mathbb{V}^{\sharp}$ be a nonzero element. Let $\widetilde{\mathcal{D}}^{\sharp,\textup{bl}}(x)$ be the strict transform of the difference divisor $\mathcal{D}^{\sharp}(x)\coloneqq\mathcal{Z}^{\sharp}(x)-\CZ^{\sharp}(\varpi x)$. Then $\widetilde{\mathcal{D}}^{\sharp,\textup{bl}}(x)$ is a regular formal scheme.\label{regularity-diff-divisor-blow-up}
\end{lemma}
\begin{proof}
    Let $z\in\widetilde{\mathcal{D}}^{\sharp,\textup{bl}}(x)(\mathbb{F})$ be a point. Let $z_0=\pis(z)\in\mathcal{D}^{\sharp,\textup{bl}}(x)(\mathbb{F})$. Let $\widetilde{d}_{x,z}$ (resp. $d_{x,z_0}$) be the local equation of the difference divisor $\widetilde{\mathcal{D}}^{\sharp,\textup{bl}}(x)$ (resp. $\mathcal{D}^{\sharp,\textup{bl}}(x)$) in the completed local ring $\widehat{\mathcal{O}}_{\mathcal{M}_{L^{\sharp}},z}$ (resp. $\widehat{\mathcal{O}}_{\mathcal{N}_{L^{\sharp}},z_0}$). Let $\mathfrak{m}_z\subset\widehat{\mathcal{O}}_{\mathcal{M}_{L^{\sharp}},z}$ (resp. $\mathfrak{m}_{z_0}\subset\widehat{\mathcal{O}}_{\mathcal{N}_{L^{\sharp}},z_0}$) be the maximal ideal. If $z_0\notin\mathcal{N}^{\textup{nfs}}_L(\mathbb{F})$, then we have $\widetilde{d}_{x,z}\in\mathfrak{m}_z\backslash\mathfrak{m}_z^{2}$ by \cite[Theorem 1.2.1]{Zhu23diff}, hence the formal scheme $\widetilde{\mathcal{D}}^{\sharp,\textup{bl}}(x)$ is regular at $z$. Now we consider the case $z_0\in\CN_L^{\textup{nfs}}$. We first consider the case that $\mathcal{D}^{\sharp,\textup{bl}}(x)$ is not formally smooth over $\ofb$ at $z_0$. Then $d_{x,z_0}\equiv\varpi\,\,\textup{mod}\,\,\mathfrak{m}_{z_0}^{2}$ by \cite{Zhu23diff}. Using the open cover $\{\widehat{\mathcal{M}}_{L^{\sharp},z_0,i}\}_{i=0}^{n}$ in (\ref{local-piece-smooth}). Let $d_{x,z_0,i}\in\mathcal{O}_{\widehat{\mathcal{M}}_{L^{\sharp},z_0,i}}$ be the image of $d_{x,z_0}$ in the $i$-th cover. Then we have
        \begin{equation*}
            d_{x,z_0,0}=(\textup{unit})\cdot\varpi,\,\,d_{x,z_0,i}=t_i\cdot\left(u_i+t_i\cdot g_i\right)\,\,\textup{where $g_i\in\mathcal{O}_{\widehat{\mathcal{M}}_{L^{\sharp},z_0,i}}$}.
        \end{equation*}
    Then we see that $z\in\widehat{\mathcal{M}}_{L^{\sharp},z_0,i}$ for some $i\geq1$ and we have $\widetilde{d}_{x,z}\equiv\varpi\,\,\textup{mod}\,\,\mathfrak{m}_{z}^{2}$. Therefore the formal scheme $\widetilde{\mathcal{D}}^{\sharp,\textup{bl}}(x)$ is regular at $z$. The proof for the case that $\mathcal{D}^{\sharp,\textup{bl}}(x)$ is formally smooth over $\ofb$ at $z_0$ is similar and easier, so we omit it.
\end{proof}
\begin{lemma}
    Let $x,y\in\mathbb{V}^{\sharp}$ be two linearly independent elements. Then the two divisors $\widetilde{\mathcal{Z}}^{\sharp,\textup{bl}}(x)$ and $\mathcal{Z}^{\sharp,\textup{bl}}(y)$ intersect properly, i.e.,
    \begin{equation*}
        \mathcal{O}_{\widetilde{\mathcal{Z}}^{\sharp,\textup{bl}}(x)}\otimes^{\mathbb{L}}_{\mathcal{O}_{\mathcal{M}_{L^{\sharp}}}}\mathcal{O}_{\mathcal{Z}^{\sharp,\textup{bl}}(y)}=\mathcal{O}_{\widetilde{\mathcal{Z}}^{\sharp,\textup{bl}}(x)}\otimes_{\mathcal{O}_{\mathcal{M}_{L^{\sharp}}}}\mathcal{O}_{\mathcal{Z}^{\sharp,\textup{bl}}(y)}
    \end{equation*}\label{proper-with-strict}
\end{lemma}
\begin{proof}
    The special divisor $\CZ^{\sharp}(x)$ is the summation of difference divisors $\mathcal{D}^{\sharp}(\varpi^{-i}x)\coloneqq\mathcal{Z}^{\sharp}(\varpi^{-i}x)-\CZ^{\sharp}(\varpi^{-i-1}x)$ where $0\leq i\leq [\frac{\nu_{\varpi}(q(x))}{2}]$. Denote by $\widetilde{\mathcal{D}}^{\sharp,\textup{bl}}(\varpi^{-i}x)$ the strict transform of the difference divisor $\mathcal{D}^{\sharp}(\varpi^{-i}x)$. It's easy to see that the divisor $\widetilde{\mathcal{Z}}^{\sharp,\textup{bl}}(x)$ is the summation of divisors $\widetilde{\mathcal{D}}^{\sharp,\textup{bl}}(\varpi^{-i}x)$. By Lemma \ref{regularity-diff-divisor-blow-up}, all the divisors $\widetilde{\mathcal{D}}^{\sharp,\textup{bl}}(\varpi^{-i}x)$ are regular formal schemes. We will prove that the two divisors $\widetilde{\mathcal{D}}^{\sharp,\textup{bl}}(\varpi^{-i}x)$ and $\mathcal{Z}^{\sharp,\textup{bl}}(y)$ intersect properly. Without loss of generality, we only need to consider the case $i=0$.
    \par
    Let $\mathfrak{X}$ be the formal completion of the formal scheme $\mathcal{M}_{L^{\sharp}}$ along the $\widetilde{\mathcal{D}}^{\sharp,\textup{bl}}(x)^{\textup{red}}$. Let $\mathfrak{B}$ be the set of points $z\in\mathfrak{X}(\mathbb{F})=\widetilde{\mathcal{D}}^{\sharp,\textup{bl}}(x)^{\textup{red}}(\mathbb{F})$ such that the local equations of $\widetilde{\mathcal{D}}^{\sharp,\textup{bl}}(x)$ and $\mathcal{Z}^{\sharp,\textup{bl}}(y)$ share a common divisor. Since the formal scheme $\widetilde{\mathcal{D}}^{\sharp,\textup{bl}}(x)$ is regular, the set $\mathfrak{B}$ is both open and closed in $\widetilde{\mathcal{D}}^{\sharp,\textup{bl}}(x)^{\textup{red}}$ by \cite[Lemma 3.6]{KR1}.
    \par
    Moreover, \cite[Lemma 4.11.1]{LiZhang-orthogonalKR} implies that $\mathfrak{B}\subset\widetilde{\mathcal{D}}^{\sharp,\textup{bl}}(x)^{\textup{red}}\cap\exc_{\Ms}$. If $\mathfrak{B}\neq\emptyset$, the set $\mathfrak{B}$ is a nonempty union of some of the connected components of the exceptional divisor $\exc_{L^{\sharp}}$. However, each connected component of $\exc_{L^{\sharp}}$ is not a connected component of $\widetilde{\mathcal{D}}^{\sharp,\textup{bl}}(x)^{\textup{red}}$. This is a contradiction. Therefore the set $\mathfrak{B}$ must be empty.
    \end{proof}

We have the following linear invariance result on these special cycles.
    
\begin{lemma}
    Let $M\subset\mathbb{V}^{\sharp}$ be an $\Of$-lattice of rank $2$. Let $x,y\in\mathbb{V}^{\sharp}$ a basis of $M$. Then $[\mathcal{O}_{\mathcal{Z}^{\sharp,\textup{bl}}(x)}\otimes^{\mathbb{L}}_{\mathcal{O}_{\mathcal{M}_{L^{\sharp}}}}\mathcal{O}_{\mathcal{Z}^{\sharp,\textup{bl}}(y)}]\in\textup{Gr}^{2}K_0^{\mathcal{Z}^{\sharp,\textup{bl}}(M)}(\mathcal{M}_{L^{\sharp}})$ is independent of the choice of the basis $x,y$ of $M$.\label{linear-invariance-smooth}
\end{lemma}
\begin{proof}
By Proposition \ref{int-spe-exc-unramified}, the derived tensor product of $\CO_{\mathcal{Z}^{\sharp,\textup{bl}}(x)}$ with $\exc_{L^{\sharp}}$ is the trivial line bundle on $\exc_{L^{\sharp}}$. Therefore we have:
\begin{equation}
    [\mathcal{O}_{\mathcal{Z}^{\sharp,\textup{bl}}(x)}\otimes^{\mathbb{L}}_{\mathcal{O}_{\mathcal{M}_{L^{\sharp}}}}\mathcal{O}_{\mathcal{Z}^{\sharp,\textup{bl}}(y)}]=[\mathcal{O}_{\widetilde{\mathcal{Z}^{\sharp,\textup{bl}}}(x)}\otimes^{\mathbb{L}}_{\mathcal{O}_{\mathcal{M}_{L^{\sharp}}}}\mathcal{O}_{\mathcal{Z}^{\sharp,\textup{bl}}(y)}].\label{tilde-inv}
\end{equation}
    Let $x^{\prime},y^{\prime}\in M$ be another $\Of$-basis of $M$. Then there exist $a,b,c,d\in\Of$ such that
    \begin{equation*}
        x^{\prime}=ax+by,\,\,\,\,y^{\prime}=cx+dy,\,\,\,\,ad-bc\in\Of^{\times}.
    \end{equation*}
    We first consider some special cases. 
    \begin{itemize}
        \item [Case 1]: $x^{\prime}=x$, $y^{\prime}=cx+y$. Then by Lemma \ref{proper-with-strict}, we have
        \begin{equation*}
            \CO_{\widetilde{\mathcal{Z}}^{\sharp,\textup{bl}}(x^{\prime})}\otimes^{\mathbb{L}}_{\CO_{\mathcal{M}_{L^{\sharp}}}}\CO_{\mathcal{Z}^{\sharp,\textup{bl}}(y^{\prime})}=\CO_{\widetilde{\mathcal{Z}}^{\sharp,\textup{bl}}(x)}\otimes_{\CO_{\mathcal{M}_{L^{\sharp}}}}\CO_{\mathcal{Z}^{\sharp,\textup{bl}}(y^{\prime})}=\CO_{\widetilde{\mathcal{Z}}^{\sharp,\textup{bl}}(x)\cap\mathcal{Z}^{\sharp,\textup{bl}}(y^{\prime})}.
        \end{equation*}
        By the moduli interpretations of the special cycles, we have $\widetilde{\mathcal{Z}}^{\sharp,\textup{bl}}(x)\cap\mathcal{Z}^{\sharp,\textup{bl}}(y^{\prime})=\widetilde{\mathcal{Z}}^{\sharp,\textup{bl}}(x)\cap\mathcal{Z}^{\sharp,\textup{bl}}(y)$. Therefore
        \begin{equation*}
            \CO_{\widetilde{\mathcal{Z}}^{\sharp,\textup{bl}}(x^{\prime})}\otimes^{\mathbb{L}}_{\CO_{\mathcal{M}_{L^{\sharp}}}}\CO_{\mathcal{Z}^{\sharp,\textup{bl}}(y^{\prime})}=\CO_{\widetilde{\mathcal{Z}}^{\sharp,\textup{bl}}(x)}\otimes^{\mathbb{L}}_{\CO_{\mathcal{M}_{L^{\sharp}}}}\CO_{\mathcal{Z}^{\sharp,\textup{bl}}(y)}.
        \end{equation*}
        \item [Case 2]: $x^{\prime}=x+ay$, $y^{\prime}=y$. We have:
        \begin{align*}
        [\CO_{\widetilde{\mathcal{Z}}^{\sharp,\textup{bl}}(x^{\prime})}\otimes^{\mathbb{L}}_{\CO_{\Ms}}\CO_{\mathcal{Z}^{\sharp,\textup{bl}}(y^{\prime})}&]=[\CO_{\mathcal{Z}^{\sharp,\textup{bl}}(x^{\prime})}\otimes^{\mathbb{L}}_{\CO_{\Ms}}\CO_{\mathcal{Z}^{\sharp,\textup{bl}}(y)}]=[\CO_{\mathcal{Z}^{\sharp,\textup{bl}}(x^{\prime})}\otimes^{\mathbb{L}}_{\CO_{\Ms}}\CO_{\widetilde{\mathcal{Z}}^{\sharp,\textup{bl}}(y)}]\\
            &\overset{\textup{Case}\,1}{=}[\CO_{\mathcal{Z}^{\sharp,\textup{bl}}(x)}\otimes^{\mathbb{L}}_{\CO_{\Ms}}\CO_{\widetilde{\mathcal{Z}}^{\sharp,\textup{bl}}(y)}]=[\CO_{\widetilde{\mathcal{Z}}^{\sharp,\textup{bl}}(x)}\otimes^{\mathbb{L}}_{\CO_{\Ms}}\CO_{\mathcal{Z}^{\sharp,\textup{bl}}(y)}].
        \end{align*}
        \par
        Now we come back to prove the proposition. There are two situations.
    \begin{itemize}
        \item If $a\in\Of^{\times}$. Scaling $x^{\prime}$ by $a^{-1}\in\Of^{\times}$, we can assume $a=1$. Then $y^{\prime}=cx^{\prime}+(d-bc)y$. Scaling $y^{\prime}$ by $(d-bc)^{-1}\in\Of^{\times}$, we can assume $d-bc=1$. Then $x^{\prime}=x+by, y^{\prime}=cx^{\prime}+y$. Therefore
        \begin{align*}
            &[\CO_{\mathcal{Z}^{\sharp,\textup{bl}}(x^{\prime})}\otimes^{\mathbb{L}}_{\CO_{\Ms}}\CO_{\mathcal{Z}^{\sharp,\textup{bl}}(y^{\prime})}]=[\CO_{\widetilde{\mathcal{Z}}^{\sharp,\textup{bl}}(x^{\prime})}\otimes^{\mathbb{L}}_{\CO_{\Ms}}\CO_{\mathcal{Z}^{\sharp,\textup{bl}}(y^{\prime})}]\overset{\textup{Case}\,1}{=}[\CO_{\widetilde{\mathcal{Z}}^{\sharp,\textup{bl}}(x^{\prime})}\otimes^{\mathbb{L}}_{\CO_{\Ms}}\CO_{\mathcal{Z}^{\sharp,\textup{bl}}(y)}]\\&\stackrel{\textup{Case 2}}=[\CO_{\widetilde{\mathcal{Z}}^{\sharp,\textup{bl}}(x)}\otimes^{\mathbb{L}}_{\CO_{\Ms}}\CO_{\mathcal{Z}^{\sharp,\textup{bl}}(y)}]=[\CO_{\mathcal{Z}^{\sharp,\textup{bl}}(x)}\otimes^{\mathbb{L}}_{\CO_{\Ms}}\CO_{\mathcal{Z}^{\sharp,\textup{bl}}(y)}].
        \end{align*}
        \item If $\nu_\varpi(a)\geq1$. Then $b,c\in\Of^{\times}$. Scaling $x^{\prime}$ by $b^{-1}$, we can assume $b=1$. Then $y^{\prime}=dx^{\prime}+(c-ad)x$. Scaling $y^{\prime}$ by $(c-ad)^{-1}$, we can assume $c-ad=1$. Then $x^{\prime}=ax+y, y^{\prime}=x+dx^{\prime}$. Therefore
        \begin{align*}
            &[\CO_{\mathcal{Z}^{\sharp,\textup{bl}}(x^{\prime})}\otimes^{\mathbb{L}}_{\CO_{\Ms}}\CO_{\mathcal{Z}^{\sharp,\textup{bl}}(y^{\prime})}]=[\CO_{\widetilde{\mathcal{Z}}^{\sharp,\textup{bl}}(x^{\prime})}\otimes^{\mathbb{L}}_{\CO_{\Ms}}\CO_{\mathcal{Z}^{\sharp,\textup{bl}}(y^{\prime})}]\stackrel{\textup{Case 1}}=[\CO_{\widetilde{\mathcal{Z}}^{\sharp,\textup{bl}}(x^{\prime})}\otimes^{\mathbb{L}}_{\CO_{\Ms}}\CO_{\mathcal{Z}^{\sharp,\textup{bl}}(x)}]\\
            &=[\CO_{\widetilde{\mathcal{Z}}^{\sharp,\textup{bl}}(x)}\otimes^{\mathbb{L}}_{\CO_{\Ms}}\CO_{\mathcal{Z}^{\sharp,\textup{bl}}(x^{\prime})}]\stackrel{\textup{Case 1}}=[\CO_{\widetilde{\mathcal{Z}}^{\sharp,\textup{bl}}(x)}\otimes^{\mathbb{L}}_{\CO_{\Ms}}\CO_{\mathcal{Z}^{\sharp,\textup{bl}}(y)}]=[\CO_{\mathcal{Z}^{\sharp,\textup{bl}}(x)}\otimes^{\mathbb{L}}_{\CO_{\Ms}}\CO_{\mathcal{Z}^{\sharp,\textup{bl}}(y)}].
        \end{align*}
    \end{itemize}
    \end{itemize}
\end{proof}
Let $M\subset\mathbb{V}^{\sharp}$ be a $\Of$-lattice of rank $r$ where $1\leq r\leq n+1$. Let $\boldsymbol{x}=\{x_1,\cdots,x_r\}$ be a basis of $M$. Define
\begin{equation*}
        {^{\mathbb{L}}\mathcal{Z}^{\sharp,\textup{bl}}}(\boldsymbol{x})\coloneqq[\CO_{\mathcal{Z}^{\sharp,\textup{bl}}(x_1)}\otimes^{\mathbb{L}}_{\CO_{\Ms}}\cdots\otimes^{\mathbb{L}}_{\CO_{\Ms}}\CO_{\mathcal{Z}^{\sharp,\textup{bl}}(x_r)}]\in\textup{Gr}^{r}K_0^{\mathcal{Z}^{\sharp,\textup{bl}}(M)}(\Ms).
    \end{equation*}
Notice that the element ${^{\mathbb{L}}\mathcal{Z}^{\sharp,\textup{bl}}}(\boldsymbol{x})$ is invariant under permutations and operations of the form $x_i\rightarrow a_ix_i+a_jx_j$ for some $a_i\in\Of^{\times}$ and $a_j\in\Of$ by Lemma \ref{linear-invariance-smooth}. We can also transform the basis $\boldsymbol{x}$ by permutations and operations of the form $x_i\rightarrow a_ix_i+a_jx_j$ for some $a_i\in\Of^{\times}$ and $a_j\in\Of$ to get any another basis $\boldsymbol{x}^{\prime}=\{x_1^{\prime},\cdots,x_r^{\prime}\}$ of $M$. Therefore the element ${^{\mathbb{L}}\mathcal{Z}^{\sharp,\textup{bl}}}(\boldsymbol{x})$ only depend on $M$. We denote this element by ${^{\mathbb{L}}\mathcal{Z}^{\sharp,\textup{bl}}}(M)$.
\begin{corollary}
    Let $M\subset\mathbb{V}$ be an $\Of$-lattice of rank $2$. Let $x,y\in\mathbb{V}^{\sharp}$ a basis of $M$. Then $[\mathcal{O}_{\mathcal{Z}^{\textup{bl}}(x)}\otimes^{\mathbb{L}}_{\mathcal{O}_{\M}}\mathcal{O}_{\mathcal{Z}^{\textup{bl}}(y)}]\in\textup{Gr}^{2}K_0^{\mathcal{Z}^{\textup{bl}}(M)}(\M)$ is independent of the choice of the basis $x,y$ of $M$.\label{linear-invariance-al}
\end{corollary}
\begin{proof}
    Let $M^{\prime}=M\obot\Of\cdot x_0$ be an $\Of$-lattice of rank $3$. We have
    \begin{align*}
        {^{\mathbb{L}}\mathcal{Z}^{\sharp,\textup{bl}}}(M^{\prime})&=[\CO_{\mathcal{Z}^{\sharp,\textup{bl}}(x_0)}\otimes^{\mathbb{L}}_{\CO_{\Ms}}\CO_{\mathcal{Z}^{\sharp,\textup{bl}}(x)}\otimes^{\mathbb{L}}_{\CO_{\Ms}}\CO_{\mathcal{Z}^{\sharp,\textup{bl}}(y)}]\\
        &=[\CO_{\widetilde{\mathcal{Z}}^{\sharp,\textup{bl}}(x_0)}\otimes^{\mathbb{L}}_{\CO_{\Ms}}\CO_{\mathcal{Z}^{\sharp,\textup{bl}}(x)}\otimes^{\mathbb{L}}_{\CO_{\Ms}}\CO_{\mathcal{Z}^{\sharp,\textup{bl}}(y)}]\\
        &=[\left(\CO_{\widetilde{\mathcal{Z}}^{\sharp,\textup{bl}}(x_0)}\otimes^{\mathbb{L}}_{\CO_{\Ms}}\CO_{\mathcal{Z}^{\sharp,\textup{bl}}(x)}\right)\otimes^{\mathbb{L}}_{\CO_{\widetilde{\mathcal{Z}}^{\sharp,\textup{bl}}(x_0)}}\left(\CO_{\widetilde{\mathcal{Z}}^{\sharp,\textup{bl}}(x_0)}\otimes^{\mathbb{L}}_{\CO_{\Ms}}\CO_{\mathcal{Z}^{\sharp,\textup{bl}}(y)}\right)]\\
        &=[\CO_{\widetilde{\mathcal{Z}}^{\sharp,\textup{bl}}(x_0)\cap\mathcal{Z}^{\sharp,\textup{bl}}(x)}\otimes^{\mathbb{L}}_{\CO_{\widetilde{\mathcal{Z}}^{\sharp,\textup{bl}}(x_0)}}\CO_{\widetilde{\mathcal{Z}}^{\sharp,\textup{bl}}(x_0)\cap\mathcal{Z}^{\sharp,\textup{bl}}(y)}].
    \end{align*}
    Under the isomorphism $\iota:\M\xrightarrow{\sim}\widetilde{\mathcal{Z}}^{\sharp,\textup{bl}}(x_0)$ in (\ref{almost-spe-blow-up}), we have
    \begin{equation*}
        [\CO_{\widetilde{\mathcal{Z}}^{\sharp,\textup{bl}}(x_0)\cap\mathcal{Z}^{\sharp,\textup{bl}}(x)}\otimes^{\mathbb{L}}_{\CO_{\widetilde{\mathcal{Z}}^{\sharp,\textup{bl}}(x_0)}}\CO_{\widetilde{\mathcal{Z}}^{\sharp,\textup{bl}}(x_0)\cap\mathcal{Z}^{\sharp,\textup{bl}}(y)}]\simeq[\mathcal{O}_{\mathcal{Z}^{\textup{bl}}(x)}\otimes^{\mathbb{L}}_{\mathcal{O}_{\M}}\mathcal{O}_{\mathcal{Z}^{\textup{bl}}(y)}].
    \end{equation*}
    The element ${^{\mathbb{L}}\mathcal{Z}^{\sharp,\textup{bl}}}(M^{\prime})$ is independent of the choice of the basis of $M^{\prime}$, hence $M$. Therefore the element $[\mathcal{O}_{\mathcal{Z}^{\textup{bl}}(x)}\otimes^{\mathbb{L}}_{\mathcal{O}_{\M}}\mathcal{O}_{\mathcal{Z}^{\textup{bl}}(y)}]\in\textup{Gr}^{2}K_0^{\mathcal{Z}^{\textup{bl}}(M)}(\M)$ is independent of the choice of the basis $x,y$ of $M$.
\end{proof}
\begin{definition}
    Let $M\subset\mathbb{V}$ be a $\Of$-lattice of rank $r$ where $1\leq r\leq n$. Let $\boldsymbol{x}=\{x_1,\cdots,x_r\}$ be a basis of $M$. Define 
    \begin{equation}
        {^{\mathbb{L}}\mathcal{Z}}(M)\coloneqq[\CO_{\mathcal{Z}(x_1)}\otimes^{\mathbb{L}}_{\CO_{\CN_L}}\cdots\otimes^{\mathbb{L}}_{\CO_{\CN_L}}\CO_{\mathcal{Z}}(x_r)]\in\textup{Gr}^{r}K_0^{\mathcal{Z}(M)}(\CN_L),\label{der-ori}
    \end{equation}
    and
    \begin{equation}
        {^{\mathbb{L}}\mathcal{Z}^{\textup{bl}}}(M)\coloneqq[\CO_{\mathcal{Z}^{\textup{bl}}(x_1)}\otimes^{\mathbb{L}}_{\CO_{\M}}\cdots\otimes^{\mathbb{L}}_{\CO_{\M}}\CO_{\mathcal{Z}^{\textup{bl}}(x_r)}]\in\textup{Gr}^{r}K_0^{\mathcal{Z}^{\textup{bl}}(M)}(\M).\label{der-bl}
    \end{equation}\label{derived-spe-cycle}
\end{definition}
\begin{remark}
    Lemma \ref{cancellation-law} and \cite[Corollary 4.11.2]{LiZhang-orthogonalKR} imply the right hand side of the formula (\ref{der-ori}) only depends on the lattice $M$. Corollary \ref{linear-invariance-al} implies that right hand side of the formula (\ref{der-bl}) only depends on the lattice $M$. See also \cite{Ho2} for the unitary case.
\end{remark}

\subsection{Derived $\mathcal{Y}$-cycles}
\begin{lemma}
    Let $M\subset\mathbb{V}$ be an $\Of$-lattice of rank $2$. Let $x,y\in\mathbb{V}$ a basis of $M$. Then $[\mathcal{O}_{\mathcal{Y}^{\textup{bl}}(x)}\otimes^{\mathbb{L}}_{\mathcal{O}_{\M}}\mathcal{O}_{\mathcal{Y}^{\textup{bl}}(y)}]\in\textup{Gr}^{2}K_0^{\mathcal{Y}^{\textup{bl}}(M)}(\M)$ is independent of the choice of the basis $x,y$ of $M$.\label{linear-invariance-Y}
\end{lemma}
\begin{proof}
    Let $\mathcal{Z}\subset\widetilde{\CZ}^{\sharp,\textup{bl}}(x_0)\simeq\M$ be a closed formal subscheme. We have an isomorphism
    \begin{equation*}
        \textup{Gr}^{r}K_0^{\CZ}(\Ms)\simeq\textup{Gr}^{r-1}K_0^{\CZ}(\M)
    \end{equation*}
    By Proposition \ref{decom-Y}, we have the following equality for a nonzero element $x\in\mathbb{V}$:
    \begin{align*}
        [\mathcal{O}_{\mathcal{Y}^{\textup{bl}}(x)}]&=\sum\limits_{\mu\in L^{\vee}/L}[\CO_{\mathcal{Z}^{\textup{bl}}(x,\mu)}]+[\CO_{\exc_{L}(x)}]=\sum\limits_{\mu\in L^{\vee}/L}[\CO_{\mathcal{Z}^{\sharp,\textup{bl}}(x+\eta(\mu))}\otimes^{\mathbb{L}}_{\CO_{\Ms}}\CO_{\widetilde{\CZ}^{\sharp,\textup{bl}}(x_0)}]+[\CO_{\exc_{L}(x)}].
    \end{align*}
    Therefore
    \begin{align}
        [\mathcal{O}_{\mathcal{Y}^{\textup{bl}}(x)}\otimes^{\mathbb{L}}_{\mathcal{O}_{\M}}\mathcal{O}_{\mathcal{Y}^{\textup{bl}}(y)}]=\sum\limits_{\boldsymbol{\mu}\in(L^{\vee}/L)^{2}}&[\CO_{\mathcal{Z}^{\sharp,\textup{bl}}(x+\eta(\mu_1))}\otimes^{\mathbb{L}}_{\CO_{\Ms}}\CO_{\mathcal{Z}^{\sharp,\textup{bl}}(y+\eta(\mu_2))}\otimes^{\mathbb{L}}_{\CO_{\Ms}}\CO_{\widetilde{\CZ}^{\sharp,\textup{bl}}(x_0)}]\label{step1}\\&+[\CO_{\exc_{L}(M)}(-1)].\notag
    \end{align}
    Notice that
    \begin{align}
        [\CO_{\mathcal{Z}^{\sharp,\textup{bl}}(x+\eta(\mu_1))}\otimes^{\mathbb{L}}_{\CO_{\Ms}}&\CO_{\mathcal{Z}^{\sharp,\textup{bl}}(y+\eta(\mu_2))}\otimes^{\mathbb{L}}_{\CO_{\Ms}}\CO_{\widetilde{\CZ}^{\sharp,\textup{bl}}(x_0)}]\label{step2}\\
        &=[\CO_{\mathcal{Z}^{\sharp,\textup{bl}}(x+\eta(\mu_1))}\otimes^{\mathbb{L}}_{\CO_{\Ms}}\CO_{\mathcal{Z}^{\sharp,\textup{bl}}(y+\eta(\mu_2))}\otimes^{\mathbb{L}}_{\CO_{\Ms}}\CO_{\CZ^{\sharp,\textup{bl}}(x_0)}].\notag
    \end{align}
    For an element $\boldsymbol{\mu}=(\mu_1,\mu_2)\in (L^{\vee}/L)^{2}$, let $M(x,y,\boldsymbol{\mu})$ be the lattice spanned by $x_0,x+\eta(\mu_1),y+\eta(\mu_2)$. We get the following by (\ref{step1}) and (\ref{step2}):
    \begin{equation*}
        [\mathcal{O}_{\mathcal{Y}^{\textup{bl}}(x)}\otimes^{\mathbb{L}}_{\mathcal{O}_{\M}}\mathcal{O}_{\mathcal{Y}^{\textup{bl}}(y)}]=[\CO_{\exc_{L}(M)}(-1)]+\sum\limits_{\boldsymbol{\mu}\in(L^{\vee}/L)^{2}}{^{\mathbb{L}}\CZ^{\sharp,\textup{bl}}(M(x,y,\boldsymbol{\mu}))}.
    \end{equation*}
    Let $x^{\prime},y^{\prime}$ be another basis of $M$. The set $\{M(x^{\prime},y^{\prime},\boldsymbol{\mu})\}_{\boldsymbol{\mu}\in(L^{\vee}/L)^{2}}$ equals to the set $\{M(x,y,\boldsymbol{\mu})\}_{\boldsymbol{\mu}\in(L^{\vee}/L)^{2}}$. Therefore by Lemma \ref{linear-invariance-smooth}, we have
    \begin{equation*}
        [\mathcal{O}_{\mathcal{Y}^{\textup{bl}}(x)}\otimes^{\mathbb{L}}_{\mathcal{O}_{\M}}\mathcal{O}_{\mathcal{Y}^{\textup{bl}}(y)}]=[\mathcal{O}_{\mathcal{Y}^{\textup{bl}}(x^{\prime})}\otimes^{\mathbb{L}}_{\mathcal{O}_{\M}}\mathcal{O}_{\mathcal{Y}^{\textup{bl}}(y^{\prime})}].
    \end{equation*}
\end{proof}

\begin{definition}
    Let $M\subset\mathbb{V}$ be a $\Of$-lattice of rank $r$ where $1\leq r\leq n$. Let $\boldsymbol{x}=\{x_1,\cdots,x_r\}$ be a basis of $M$. Define 
    \begin{equation}
        {^{\mathbb{L}}\mathcal{Y}}^{\textup{bl}}(M)\coloneqq[\CO_{\mathcal{Y}^{\textup{bl}}(x_1)}\otimes^{\mathbb{L}}_{\CO_{\M}}\cdots\otimes^{\mathbb{L}}_{\CO_{\M}}\CO_{\mathcal{Y}^{\textup{bl}}}(x_r)]\in\textup{Gr}^{r}K_0^{\mathcal{Y}^{\textup{bl}}(M)}(\M),\label{der-ori-y}
    \end{equation}\label{derived-spe-cycle-Y}
\end{definition}

\subsection{Derived weighted cycles}
\begin{definition}
    Let $\boldsymbol{\mu}=(\mu_1, ..., \mu_r) \in (\Lambda^\vee)^r$ be an element where $1\leq r\leq n$. Let $\boldsymbol{x}=(x_1, \hdots x_r) \in \BV^{r}$ be an element such that the vectors $\{x_i\}_{i=1}^{r}$ are linearly independent. Define 
    \begin{equation*}
        {^{\mathbb{L}}\mathcal{Z}_{\boldsymbol{\mu}}}(\boldsymbol{x})\coloneqq[\CO_{\mathcal{Z}(x_1,\mu_1)}\otimes^{\mathbb{L}}_{\CO_{\CN_L}}\cdots\otimes^{\mathbb{L}}_{\CO_{\CN_L}}\CO_{\mathcal{Z}(x_r,\mu_r)}]\in\textup{Gr}^{r}K_0^{\bigcap_{i=1}^{r}\mathcal{Z}(x_i+\mu_i)}(\CN_L),
    \end{equation*}
    and
    \begin{equation*}
        {^{\mathbb{L}}\mathcal{Z}_{\boldsymbol{\mu}}^{\textup{bl}}}(\boldsymbol{x})\coloneqq[\CO_{\mathcal{Z}^{\textup{bl}}(x_1,\mu_1)}\otimes^{\mathbb{L}}_{\CO_{\M}}\cdots\otimes^{\mathbb{L}}_{\CO_{\M}}\CO_{\mathcal{Z}^{\textup{bl}}(x_r,\mu_r)}]\in\textup{Gr}^{r}K_0^{\bigcap_{i=1}^{r}\mathcal{Z}^{\textup{bl}}(x_i+\mu_i)}(\M).
    \end{equation*}
\end{definition}
\begin{remark}
    Notice that when $\boldsymbol{\mu}=\boldsymbol{0}$. Let $M\subset\mathbb{V}$ be the rank $n$ $\Of$-lattice spanned by $x_i$. Then we have
    \begin{equation*}
        {^{\mathbb{L}}\mathcal{Z}_{\boldsymbol{0}}}(\boldsymbol{x})={^{\mathbb{L}}\mathcal{Z}}(M),\,\,{^{\mathbb{L}}\mathcal{Z}_{\boldsymbol{0}}^{\textup{bl}}}(\boldsymbol{x})={^{\mathbb{L}}\mathcal{Z}^{\textup{bl}}}(M).
    \end{equation*}
\end{remark}

\subsection{Local arithmetic intersection numbers}
\begin{lemma}
    Let $M\subset\mathbb{V}$ be a rank $n$ $\Of$-lattice with basis $x_1,\cdots,x_n$. Let $\boldsymbol{\mu}\in(L^{\vee}/L)^{n}$ be an element. The formal schemes $\mathcal{Y}^{\textup{bl}}(M)$, $\mathcal{Z}^{\textup{bl}}(M)$ and $\bigcap_{i=1}^{n} \CZ^{\textup{bl}}(x_i, \mu_i)$ (resp. $\mathcal{Y}(M)$, $\mathcal{Z}(M)$ and $\bigcap_{i=1}^{n} \CZ(x_i, \mu_i)$) are all proper schemes.\label{proper-schemes}
\end{lemma}
\begin{proof}
We only need to prove the statement for $\mathcal{Z}(M)$ since we have $\CZ(x)\subset\mathcal{Y}(x)\subset\CZ(\varpi x)$ and the blow-up version only adds finitely many projective spaces to the special cycles. By Lemma \ref{cancellation-law}, we have $\mathcal{Z}(M)=\iota^{-1}\left(\mathcal{Z}^{\sharp}(M\obot\Of\cdot x_0)\right)$. The formal scheme $\mathcal{Z}^{\sharp}(M\obot\Of\cdot x_0)\subset\CN_{L^{\sharp}}$ is a proper scheme by \cite[Lemma 4.13.1]{LiZhang-orthogonalKR}. Hence $\mathcal{Z}(M)$ is also a proper scheme.
\end{proof}
\begin{definition}
    Let $M\subset\mathbb{V}$ be an $\Of$-lattice of rank $n$. Define \textit{local arithmetic intersection numbers}
    \begin{itemize}
        \item $\mathcal{Z}$-cycles:
        \begin{itemize}
            \item On $\M$: \begin{equation}
        \Int^{\mathcal{Z}}(\M,M)=\chi(\M, {^{\mathbb{L}}\mathcal{Z}^{\textup{bl}}}(M)).\label{int-bl}\end{equation}
        \item On $\CN_L$: \begin{equation}
        \Int^{\mathcal{Z}}(\CN_L,M)=\chi(\CN_L, {^{\mathbb{L}}\mathcal{Z}}(M)).\label{int-ori}
    \end{equation}
        \end{itemize}
    \item $\mathcal{Y}$-cycles:\begin{equation}
        \Int^{\mathcal{Y}}(\M,M)=\chi(\M, {^{\mathbb{L}}\mathcal{Y}^{\textup{bl}}}(M)),\,\,\Int^{\overline{\mathcal{Y}}}(\M,M)=\chi(\M, {^{\mathbb{L}}\overline{\mathcal{Y}}^{\textup{bl}}}(M)).\label{int-Y}\end{equation}
    \item Weighted cycles: Let $\boldsymbol{x}=(x_1,\cdots,x_n)\in\mathbb{V}^{n}$ be an element such that $\{x_{i}\}_{i=1}^{n}$ be a basis of $M$. Let $\boldsymbol{\mu}\in(L^{\vee}/L)^{n}$ be an element.
    \begin{itemize}
        \item On $\M$:\begin{equation}
        \Int_{\boldsymbol{\mu}}(\M,\boldsymbol{x})=\chi(\M,{^{\mathbb{L}}\mathcal{Z}_{\boldsymbol{\mu}}^{\textup{bl}}}(\boldsymbol{x})).
    \end{equation}
    \item On $\CN_L$:\begin{equation}
        \Int_{\boldsymbol{\mu}}(\CN_L,\boldsymbol{x})=\chi(\CN_L,{^{\mathbb{L}}\mathcal{Z}_{\boldsymbol{\mu}}}(\boldsymbol{x})).
    \end{equation}
    \end{itemize}
    \end{itemize}
    All this numbers are finite by Lemma \ref{proper-schemes}.
    \label{int-spe-cycle}
\end{definition}
Now we are able to state the main theorem of this article. Recall that the almost self-dual lattice $L\simeq H_{n}^{\varepsilon}\obot\langle x\rangle$. Let $\mathcal{L}\subset\mathbb{V}$ be the unique (up to isometry) type 1 lattice of rank $n+1$.
\begin{theorem}\label{main}
Let $M\subset\mathbb{V}$ be a $\Of$-lattice of rank $n$. Let $\boldsymbol{x}=(x_1,\cdots,x_n)\in\mathbb{V}^{n}$ be an element such that $\{x_{i}\}_{i=1}^{n}$ be a basis of $M$. Let $T$ be the inner product matrix of $M$ with respect to the basis $\boldsymbol{x}$. Let $\boldsymbol{\mu}\in(L^{\vee}/L)^{n}$ be an element. Then
    \begin{align}
        \Int_{\boldsymbol{\mu}}(\M,\boldsymbol{x})=\Int_{\boldsymbol{\mu}}(\CN_L,\boldsymbol{x})=\frac{2q^{n/2}}{\gamma(L)\cdot\log(q)\cdot\textup{Nor}(1,H_{n}^{\varepsilon})}\cdot W_{T}^{\prime}(1,0,1_{\boldsymbol{\mu}+L^{n}}).\label{weighted-int}
    \end{align}
Especially, when $\boldsymbol{\mu}=\boldsymbol{0}$, we have
\begin{equation}
    \Int^{\mathcal{Z}}(\M,M)=\Int(\CN_L,M)=\partial\Den^{L}(M).\label{Z-int}
\end{equation}
Moreover,
\begin{equation}\label{Y-int}
    \Int^{\mathcal{Y}}(\M,M)=\partial\Den^{L^{\vee}}(M)+(-1)^{n}\cdot\Den^{\mathcal{L}^{\vee}}(M).
\end{equation}
\end{theorem}
The proof of this theorem will be given in the next section.

\section{Local modularity}
\subsection{Blow up invariance}
\begin{proposition}\label{inv}
Let $\boldsymbol{\mu}\in (\Lambda^\vee/\Lambda)^n$ be an element. Let $\boldsymbol{x}=(x_1, \hdots x_n) \in \BV^{n}$ be an element such that the vectors $\{x_i\}_{i=1}^{n}$ are linearly independent. We have
\begin{equation}\label{eq: pushforward along blow up}
    {}^\BL \mathcal{Z}_{\boldsymbol{\mu}}(\boldsymbol{x}) = \textup{\textup{R}} \pi_*({}^{\mathbb{L}}\mathcal{Z}_{\boldsymbol{\mu}}^{\mathrm{bl}}(\boldsymbol{x})).
\end{equation}
\end{proposition}
\begin{proof}
    Let $\pi_L: \mathcal{M}_L \to \CN_L$ denote the blow-up morphism. According to \cite[Lemma 20.54.3]{stack}, we have  
    \begin{align*}
        \bold{\textup{R}} \pi_{L,*} \CO_{\M} \otimes_{\mathcal{O}_{\M}}^{\BL}  {}^{\mathbb{L}}\mathcal{Z}_{\boldsymbol{\mu}}(\boldsymbol{x}) =\bold{\textup{R}} \pi_{L,*}\left( \CO_{\M} \otimes_{\mathcal{O}_{\M}}^{\BL} \boldsymbol{\textup{L}} \pi_L^* ({}^{\mathbb{L}}\mathcal{Z}_{\boldsymbol{\mu}}(\boldsymbol{x}))\right)=\bold{\textup{R}} \pi_{L,*}({}^{\mathbb{L}}\mathcal{Z}_{\boldsymbol{\mu}}^{\mathrm{bl}}(\boldsymbol{x})).
    \end{align*}
The second equality follows from the fact that ${}^{\mathbb{L}}\mathcal{Z}_{\boldsymbol{\mu}}(\boldsymbol{x})$ is the derived tensor product of divisors. Since $\pi$ is $\CO$-connected, we have $\boldsymbol{\textup{R}} \pi_{L,*} \CO_{\M}=\CO_{\CN_L},$ hence
\begin{align*}
   {}^\BL \mathcal{Z}_{\boldsymbol{\mu}}(\boldsymbol{x}) = \bold{\textup{R}} \pi_{L,*}({}^{\mathbb{L}}\mathcal{Z}_{\boldsymbol{\mu}}^{\mathrm{bl}}(\boldsymbol{x})).
    \end{align*}
\end{proof}
\begin{corollary}\label{cor: blowup do not change int}
    Let $\boldsymbol{\mu}\in (\Lambda^\vee/\Lambda)^n$ be an element. Let $\boldsymbol{x}=(x_1, \hdots x_n) \in \BV^{n}$ be an element such that the vectors $\{x_i\}_{i=1}^{n}$ are linearly independent. We have
    \begin{equation*}
        \Int_{\boldsymbol{\mu}}(\M,\boldsymbol{x})=\Int_{\boldsymbol{\mu}}(\CN_L,\boldsymbol{x}).
    \end{equation*}
    Especially if $\boldsymbol{\mu}=\boldsymbol{0}$. Let $M\subset\mathbb{V}$ be the lattice spanned by $x_1,\hdots,x_n$. We have
    \begin{equation*}
        \Int^{\mathcal{Z}}(\M,M)=\Int^{\mathcal{Z}}(\CN_L,M).
    \end{equation*}
\end{corollary}
\begin{proof}
This follows from Proposition \ref{inv} and \cite[Lemma 10.52.3, Lemma 33.33.5]{stack}.
\end{proof}

\subsection{Proof of Theorem \ref{main}}
In the following, we use the notation as the statement of Theorem \ref{main}. We first prove the equality (\ref{weighted-int}). Recall that $\eta:L^{\vee}/L\rightarrow\Lambda^{\vee}$ is a lift of the isomorphism $r^{-1}:\Lambda^{\vee}/\Lambda\rightarrow L^{\vee}/L$. By Theorem \ref{almost-divisor} and Lemma \ref{cancellation-law}, we have
\begin{align*}
    {^{\mathbb{L}}\mathcal{Z}_{\boldsymbol{\mu}}^{\textup{bl}}}(\boldsymbol{x})&=[\CO_{\mathcal{Z}(x_1,\mu_1)}\otimes^{\mathbb{L}}_{\CO_{\CN_L}}\cdots\otimes^{\mathbb{L}}_{\CO_{\CN_L}}\CO_{\mathcal{Z}(x_r,\mu_n)}]\\
    &=[\CO_{\mathcal{Z}^{\sharp}(x_0)}\otimes^{\mathbb{L}}_{\CO_{\CN_{L^{\sharp}}}}\CO_{\mathcal{Z}^{\sharp}(x_1+\eta(\mu_1))}\otimes^{\mathbb{L}}_{\CO_{\CN_{L^{\sharp}}}}\cdots\otimes^{\mathbb{L}}_{\CO_{\CN_{L^{\sharp}}}}\CO_{\mathcal{Z}^{\sharp}(x_r+\eta(\mu_n))}].
\end{align*}
Therefore we have
\begin{align*}
    \Int_{\boldsymbol{\mu}}(\CN_L,\boldsymbol{x})&=\chi\left(\CN_L,{^{\mathbb{L}}\mathcal{Z}_{\boldsymbol{\mu}}^{\textup{bl}}}(\boldsymbol{x})\right)\\&=\chi\left(\CN_{L^{\sharp}},[\CO_{\mathcal{Z}^{\sharp}(x_0)}\otimes^{\mathbb{L}}_{\CO_{\CN_{L^{\sharp}}}}\CO_{\mathcal{Z}^{\sharp}(x_1+\eta(\mu_1))}\otimes^{\mathbb{L}}_{\CO_{\CN_{L^{\sharp}}}}\cdots\otimes^{\mathbb{L}}_{\CO_{\CN_{L^{\sharp}}}}\CO_{\mathcal{Z}^{\sharp}(x_r+\eta(\mu_n))}]\right).
\end{align*}
By \cite[Theorem 1.2.1]{LiZhang-orthogonalKR} and (\ref{eq: Whitt=Den}), we have
\begin{align*}
    \chi\left(\CN_{L^{\sharp}},[\CO_{\mathcal{Z}^{\sharp}(x_0)}\otimes^{\mathbb{L}}_{\CO_{\CN_{L^{\sharp}}}}\CO_{\mathcal{Z}^{\sharp}(x_1+\eta(\mu_1))}\otimes^{\mathbb{L}}_{\CO_{\CN_{L^{\sharp}}}}\cdots\otimes^{\mathbb{L}}_{\CO_{\CN_{L^{\sharp}}}}\CO_{\mathcal{Z}^{\sharp}(x_r+\eta(\mu_n))}]\right)=\frac{W_{M(T,\eta(\boldsymbol{\mu}))}^{\pr}(1,0,1_{(L^{\sharp})^{n+1}})}{\log(q)\cdot\textup{Nor}(1,H_{n+2}^{\varepsilon})}.
\end{align*}
Notice that $\textup{Nor}(1,H_{n+2}^{\varepsilon})=\frac{1}{2}\cdot\textup{Pden}(H_{n+2}^{\varepsilon},\Lambda)\cdot\textup{Nor}(1,H_{n}^{\varepsilon})$. Combining with Theorem \ref{thm: reduction formula}, we get the desired formula (\ref{weighted-int}). When $\boldsymbol{\mu}=\boldsymbol{0}$, the formula (\ref{Z-int}) follows from (\ref{weighted-int}), (\ref{eq: Whitt=Den}) and the definition of the derived local density $\partial\Den^{L}(M)$ in Definition \ref{normalized-def}.
\par
Now we prove the formula (\ref{Y-int}) about the intersection of $\mathcal{Y}$-cycles. By Proposition \ref{decom-Y}, Lemma \ref{inter-spe-exc-M} and Lemma \ref{singular-on-Y}, we have
\begin{align*}
    \Int^{\mathcal{Y}}&(\M,M)=\sum\limits_{\boldsymbol{\mu}\in \left(L^{\vee}/L\right)^{n}}\Int_{\boldsymbol{\mu}}(\M,\boldsymbol{x})+\sum\limits_{z\in\CN_L^{\textup{nfs}}\cap\mathcal{Y}(M)}\chi\left(\mathbb{P}_{\mathbb{F}}^{n-1},{^{\mathbb{L}}\otimes}_{i=1}^{n}\mathcal{O}_{\mathbb{P}_{\mathbb{F}}^{n-1}}(-1)\right)\\
    &=\sum\limits_{\boldsymbol{\mu}\in \left(L^{\vee}/L\right)^{n}}\frac{2q^{n/2}}{\gamma(L)\cdot\log(q)\cdot\textup{Nor}(1,H_{n}^{\varepsilon})}\cdot W_{T}^{\prime}(1,0,1_{\boldsymbol{\mu}+L^{n}})+(-1)^{n}\cdot\sharp\left(\CN_L^{\textup{nfs}}\cap\mathcal{Y}(M)\right)\\
    &=\frac{2q^{n/2}}{\gamma(L)\cdot\log(q)\cdot\textup{Nor}(1,H_{n}^{\varepsilon})}\cdot W_{T}^{\prime}(1,0,1_{(L^{\vee})^{n}})+(-1)^{n}\cdot\sharp\{\mathcal{L}^{\pr}\subset\mathbb{V}:\mathcal{L}^{\pr}\overset{\textup{isometric}}{\sim}\mathcal{L},\,\,M\subset\mathcal{L}^{\pr}\}.
\end{align*}
Then the formula (\ref{Y-int}) follows from Proposition \ref{lattice-counting}, (\ref{eq: Whitt=Den}), the definition of the derived local density $\partial\Den^{L^{\vee}}(M)$ and normalized local density $\Den^{\mathcal{L}^{\vee}}(M)$ in Definition \ref{normalized-def}.

\subsection{Local modularity}
Let $\mathcal{L}\in\textup{Vert}^{3}(\mathbb{V})$. The special cycle $\mathcal{Z}(\mathcal{L})$ is isomorphic to $\mathbb{P}_{\mathbb{F}}^{1}$ (cf. $\S$\ref{BT-strat}). Let $\widetilde{\mathcal{Z}}(\mathcal{L})$ be the strict transform of the special cycle $\mathcal{Z}(\mathcal{L})$ under the blow up morphism $\pi_L:\M\rightarrow\CN_L$. It is isomorphic to $\mathcal{Z}(\mathcal{L})$ under $\pi_L$. Define the following  functions,
\begin{equation}
    \textup{Int}^{\mathcal{Y}}_{\mathcal{L}}(x)=\chi\left(\M,\mathcal{O}_{\widetilde{\mathcal{Z}}(\mathcal{L})}\otimes_{\mathcal{O}_{\M}}^{\BL} \mathcal{O}_{\mathcal{Y}(x)}\right),\,\,\,\,x\in\BV\backslash\{0\}.
\end{equation}
\begin{equation}
    \textup{Int}^{\mathcal{Z}}_{\mathcal{L}}(x)=\chi\left(\M,\mathcal{O}_{\widetilde{\mathcal{Z}}(\mathcal{L})}\otimes_{\mathcal{O}_{\M}}^{\BL}  {}^{\mathbb{L}}\mathcal{O}_{\mathcal{Z}^{\textup{bl}}(x)}\right),\,\,\,\,x\in\BV\backslash\{0\}.
\end{equation}
\begin{equation}
    \textup{Int}^{\mathcal{Z}_{\textup{mix}}}_{\mathcal{L}}(x)=\sum\limits_{\mu\in L^{\vee}/L}\chi\left(\M,\mathcal{O}_{\widetilde{\mathcal{Z}}(\mathcal{L})}\otimes_{\mathcal{O}_{\M}}^{\BL}  {}^{\mathbb{L}}\mathcal{O}_{\mathcal{Z}^{\textup{bl}}(x,\mu)}\right),\,\,\,\,x\in\BV\backslash\{0\}.
\end{equation}
\begin{theorem}\label{explicit value}
    Let $\mathcal{L}\in\textup{Vert}^{3}(\mathbb{V})$. Then
    \begin{equation*}
        \textup{Int}^{\mathcal{Z}}_{\mathcal{L}}(x)=\begin{cases}
            1-q, &\textup{if $x\in\mathcal{L}\backslash\{0\}$;}\\
            1, &\textup{if $x\in\mathcal{L}^{\vee}\backslash\mathcal{L}$ and $\nu_\varpi(q(x))\geq0$;}\\
            0, &\textup{otherwise.}
        \end{cases}
    \end{equation*}
    In particular, the function $\textup{Int}^{\mathcal{Z}}_{\mathcal{L}}$ on $\mathbb{V}\backslash\{0\}$ extends to a (necessarily unique) function in $\mathscr{S}(\mathbb{V})$, which we still denote by the same symbol.
\end{theorem}
\begin{proof}
     Notice that we have an isomorphism $\pi_L:\widetilde{\mathcal{Z}}(\mathcal{L})\rightarrow\mathcal{Z}(\mathcal{L})$. The same proof as Proposition \ref{inv} implies that 
     \begin{align*}
          \CO_{\mathcal{Z}(\mathcal{L})} \otimes_{\mathcal{O}_{\CN_L}}^{\BL}  \mathcal{O}_{\mathcal{Z}(x)}&=\bold{\textup{R}} \pi_{L,*} \CO_{\widetilde{\mathcal{Z}}(\mathcal{L})} \otimes_{\mathcal{O}_{\CN_L}}^{\BL}  \mathcal{O}_{\mathcal{Z}(x)} \\&=\bold{\textup{R}} \pi_{L,*}\left( \CO_{\widetilde{\mathcal{Z}}(\mathcal{L})} \otimes_{\mathcal{O}_{\M}}^{\BL} \boldsymbol{\textup{L}} \pi_L^* (\mathcal{O}_{\mathcal{Z}(x)})\right)=\bold{\textup{R}} \pi_{L,*}\left( \CO_{\widetilde{\mathcal{Z}}(\mathcal{L})} \otimes_{\mathcal{O}_{\M}}^{\BL} \mathcal{O}_{\mathcal{Z}^{\textup{bl}}(x)}\right).
     \end{align*}
     Hence $\textup{Int}^{\mathcal{Z}}_{\mathcal{L}}(x)=\chi\left(\CN_L,\CO_{\mathcal{Z}(\mathcal{L})} \otimes_{\mathcal{O}_{\CN_L}}^{\BL}  \mathcal{O}_{\mathcal{Z}(x)}\right)=\chi\left(\CN_{L^{\sharp}},\mathcal{O}_{\mathcal{Z}(\mathcal{L}\obot\Lambda)}\otimes_{\mathcal{N}_{L^{\sharp}}}^{\mathbb{L}}\mathcal{O}_{\mathcal{Z}^{\sharp}(x)}\right)$. Then the theorem follows from \cite[Theorem 7.4.1]{LiZhang-orthogonalKR}.
\end{proof}
\begin{corollary}\label{explicit-int-z}
    Let $\mathcal{L}\in\textup{Vert}^{3}(\mathbb{V})$. Then
    \begin{equation*}
        \textup{Int}^{\mathcal{Z}}_{\mathcal{L}}(x)=-2q\cdot \boldsymbol{1}_{\mathcal{L}}(x)+\sum\limits_{\mathcal{L}\subset\mathcal{L}^{\pr}\in\textup{Vert}^{1}(\mathbb{V})}\boldsymbol{1}_{\mathcal{L}^{\pr}}(x).
    \end{equation*}
\end{corollary}
\begin{proof}
    We compute the value of the right-hand-side at $x \in \mathbb{V}$ according to three cases.
    \begin{itemize}
        \item [(i).] If $x\in\mathcal{L}$, notice that there exist $q+1$ vertex lattices of type $1$ containing $\mathcal{L}$. Then the right hand side is $-2q+(q+1)=1-q$.
        \item[(ii).] If $x\in\mathcal{L}^{\vee}\backslash\mathcal{L}$ and $\nu_\varpi(q(x))\geq0$, then $x+\mathcal{L}$ is a vertex lattice of type $1$. The right hand side is 1.
        \item[(iii).] If $x\notin\mathcal{L}^{\vee}$ or $\nu_\varpi(q(x))< 0$, then $\mathcal{L}+x$ is not integral, and the value is clearly 0. 
    \end{itemize}
\end{proof}
\begin{lemma}\label{local-modularity-z}
    Let $\mathcal{L}\in\textup{Vert}^{3}(\mathbb{V})$. Then $\textup{Int}^{\mathcal{Z}}_{\mathcal{L}}\in\mathscr{S}(\mathbb{V})$ satisfies
    \begin{equation*}
        \widehat{\textup{Int}^{\mathcal{Z}}_{\mathcal{L}}}=-q^{-1/2}\cdot\textup{Int}^{\mathcal{Z}}_{\mathcal{L}}.
    \end{equation*}
\end{lemma}
\begin{proof}
    By Corollary \ref{explicit-int-z}, we have
    \begin{align*}
        \widehat{\textup{Int}^{\mathcal{Z}}_{\mathcal{L}}}(x)&=-2q\cdot\textup{vol}(\mathcal{L})\cdot\boldsymbol{1}_{\mathcal{L}^{\vee}}(x)+\sum\limits_{\mathcal{L}\subset\mathcal{L}^{\prime}
        \in\textup{Vert}^{1}(\mathbb{V})}\textup{vol}(\mathcal{L}^{\pr})\cdot\boldsymbol{1}_{\mathcal{L}^{\pr\vee}}(x)\\
        &=-2q^{-1/2}\cdot\boldsymbol{1}_{\mathcal{L}^{\vee}}(x)+\sum\limits_{\mathcal{L}\subset\mathcal{L}^{\prime}
        \in\textup{Vert}^{1}(\mathbb{V})}q^{-1/2}\cdot\boldsymbol{1}_{\mathcal{L}^{\pr\vee}}(x).
    \end{align*}
    \begin{itemize}
        \item [(i).] If $x\in\mathcal{L}$, the right hand side is $-2q^{-1/2}+(q+1)q^{-1/2}=-q^{-1/2}(1-q)$.
        \item[(ii).] If $x\in\mathcal{L}^{\vee}\backslash\mathcal{L}$ and $\nu_\varpi(q(x))\geq0$, then $\mathcal{L}^{\pr}=x+\mathcal{L}$ is a vertex lattice of type $1$. Clearly $x\in\mathcal{L}^{\pr\vee}$. The image $\mathcal{L}^{\pr}/\mathcal{L}$ is an isotropic vector in the nondegenerate 3-dimensional quadratic space $\mathbb{F}_p$-vector space $\mathcal{L}^{\vee}/\mathcal{L}$. If there exists another $\mathcal{L}^{\pr\pr}\in\textup{Vert}^{1}(\mathbb{V})$ such that $x\in\left(\mathcal{L}^{\pr\pr}\right)^{\vee}$, then the nondegenerate quadratic space $\mathcal{L}^{\vee}/\mathcal{L}$ contains an isotropic plane which is spanned by $\mathcal{L}^{\pr}/\mathcal{L}$ and $\mathcal{L}^{\pr\pr}/\mathcal{L}$, which is impossible. Hence $\mathcal{L}^{\pr}$ is the unique type $1$ vertex lattice such that $x\in\mathcal{L}^{\pr\vee}$. Hence the right hand side is $-2q^{-1/2}+q^{-1/2}=-q^{-1/2}$.
        \item[(iii).] If $x\in\mathcal{L}^{\vee}\backslash\mathcal{L}$ and $\nu_\varpi(q(x))=-1$, there exists exactly two type 1 vertex lattices $\mathcal{L}^{\pr}$ such that $x\in\mathcal{L}^{\pr}$, hence the right hand side is $-2q^{-1/2}+2q^{-1/2}=0$. If $x\notin\mathcal{L}^{\vee}$, the right hand side is clearly $0$. 
    \end{itemize}
    The result then follows from Theorem \ref{explicit value}.
\end{proof}
\begin{lemma}\label{Y-bar-z}
    Let $\mathcal{L}\in\textup{Vert}^{3}(\mathbb{V})$, then
    \begin{equation*}
        \widehat{\textup{Int}^{\mathcal{Z}_{\textup{mix}}}_{\mathcal{L}}}=-q^{1/2}\cdot\textup{Int}^{\mathcal{Z}}_{\mathcal{L}}.
    \end{equation*}
\end{lemma}
\begin{proof}
    For any $x\in\BV^{\sharp}-\{0\}$, we consider the following function
    \begin{equation*}
        f(x)=\chi\left(\CN_{L^{\sharp}},\mathcal{O}_{\mathcal{Z}^{\sharp}(\mathcal{L}\obot\langle x_0\rangle)}\otimes^{\mathbb{L}}_{\mathcal{O}_{\CN_L^{\sharp}}}\mathcal{O}_{\mathcal{Z}^{\sharp}(x)}\right).
    \end{equation*}
    By \cite[Theorem 7.4.1, Corollary 7.6.10]{LiZhang-orthogonalKR}, the function $f$ can be extended uniquely to a function in $\mathscr{S}(\BV^{\sharp})$, it is invariant under translation by elements of $\mathcal{L}$, i.e., for any $\lambda\in\mathcal{L}$, we have $f(x+\lambda)=f(x)$. We notice that the function $f$ is supported on the set $\mathcal{L}^{\vee}\obot\BV$. By the definitions of $\textup{Int}^{\overline{\mathcal{Y}}}_{\mathcal{L}}$ and $\textup{Int}^{\mathcal{Z}}_{\mathcal{L}}$, we have
    \begin{equation*}
        \textup{Int}^{\mathcal{Z}_{\textup{mix}}}_{\mathcal{L}}(x)=\sum\limits_{\mu\in L^{\vee}/L}f(x+\eta(\mu)),\,\,\,\,\textup{Int}^{\mathcal{Z}}_{\mathcal{L}}(x)=f(x), \,\,x\in\mathbb{V}.
    \end{equation*}
    It also satisfies the following local modularity property,
    \begin{equation}
        \widehat{f}(x)=-f(x),
        \label{smooth-local-mod}
    \end{equation}
    note that here $\widehat{f}$ is the Fourier transform of $f$ in the Schwartz space $\mathscr{S}(\BV^{\sharp})$ with respect to the additive character $\psi$, i.e., 
    \begin{align*}
        \widehat{f}(x)=\int\limits_{\BV^{\sharp}}f(y)\psi\left(\langle x,y\rangle\right)\textup{d}y.
    \end{align*}
    When $x\in\BV$, we have
    \begin{align*}
        &\widehat{f}(x)=\int\limits_{\BV^{\sharp}}f(y)\psi(\langle x,y\rangle)\textup{d}y=
        \int\limits_{\Lambda^{\vee}\obot\BV}f(y)\psi(\langle x,y\rangle)\textup{d}y=\sum\limits_{\mu\in L^{\vee}/L}\int\limits_{\eta(\mu)+\Lambda\obot\BV}f(y)\psi(\langle x,y\rangle)\textup{d}y\\
        &=\sum\limits_{\mu\in L^{\vee}/L}\int\limits_{\Lambda\obot\BV}f\left(\eta(\mu)+y\right)\psi\left(\langle x,\eta(\mu)+y\rangle\right)\textup{d}y=\sum\limits_{\mu\in L^{\vee}/L}\int\limits_{\Lambda}\int\limits_{\BV}f(\eta(\mu)+\lambda+y)\psi(\langle x,\eta(\mu)+\lambda+y\rangle)\textup{d}y\textup{d}\lambda\\
        &=\sum\limits_{\mu\in L^{\vee}/L}\int\limits_{\Lambda}\int\limits_{\BV}f(\eta(\mu)+y)\psi(\langle x,y\rangle)\textup{d}y\textup{d}\lambda=\sum\limits_{\mu\in L^{\vee}/L}\textup{vol}(\Lambda)\int\limits_{\BV}f(y+\eta(\mu))\psi(\langle x,y\rangle)\textup{d}y.
    \end{align*}
   Therefore,
   \begin{align*}
        \widehat{\textup{Int}^{\mathcal{Z}_{\textup{mix}}}_{\mathcal{L}}}(x)&=\int\limits_{\BV}\textup{Int}^{\mathcal{Z}_{\textup{mix}}}_{\mathcal{L}}(y)\psi(\langle x,y\rangle)\textup{d}y=\sum\limits_{\mu\in L^{\vee}/L}\int\limits_{\BV}f(y+\eta(\mu))\psi(\langle x,y\rangle)\textup{d}y\\
        &=\textup{vol}(\Lambda)^{-1}\widehat{f}(x)=-q^{1/2}f(x)=-q^{-1/2}\cdot\textup{Int}^{\mathcal{Z}}_{\mathcal{L}}(x).
   \end{align*}
\end{proof}
Notice that $\widehat{\textup{Int}^{\mathcal{Z}_{\textup{mix}}}_{\mathcal{L}}}=\textup{Int}^{\mathcal{Z}_{\textup{mix}}}_{\mathcal{L}}$, combining Lemma \ref{local-modularity-z} and Lemma \ref{Y-bar-z}, we get
\begin{corollary}\label{y-bar=z}
    Let $\mathcal{L}\in\textup{Vert}^{3}(\mathbb{V})$, then
    \begin{equation*}
        \textup{Int}^{\mathcal{Z}_{\textup{mix}}}_{\mathcal{L}}=\textup{Int}^{\mathcal{Z}}_{\mathcal{L}}.
    \end{equation*}
\end{corollary}
By Proposition \ref{decom-Y}, we have
\begin{equation*}
    \mathcal{Y}(x)=\sum\limits_{\mu\in L^{\vee}/L}\mathcal{Z}^{\textup{bl}}(x,\mu)+\exc_L(x), \,\,x\in\mathbb{V}\backslash\{0\}.
\end{equation*}
Hence we have
\begin{align}
   \begin{split} \textup{Int}^{\mathcal{Y}}_{\mathcal{L}}(x)&=\textup{Int}^{\mathcal{Z}_{\textup{mix}}}_{\mathcal{L}}(x)+\chi\left(\M,\mathcal{O}_{\widetilde{\mathcal{Z}}(\mathcal{L})}\otimes^{\mathbb{L}}_{\mathcal{O}_{\mathcal{M}_L}}\mathcal{O}_{\exc_L(x)}\right)\\
    &=\textup{Int}^{\mathcal{Z}}_{\mathcal{L}}(x)+\sharp\{\mathcal{L}^{\pr}\in\textup{Vert}^{1}(\mathbb{V}):\mathcal{L}\subset\mathcal{L}^{\pr},\,\,x\in\mathcal{L}^{\pr\vee}\}.
    \end{split}
\end{align}
\begin{lemma}\label{exc-local}
    Let $\mathcal{L}\in\textup{Vert}^{3}(\mathbb{V})$, then
    \begin{align*}
        \chi\left(\M,\mathcal{O}_{\widetilde{\mathcal{Z}}(\mathcal{L})}\otimes^{\mathbb{L}}_{\mathcal{O}_{\mathcal{M}_L}}\mathcal{O}_{\exc_L(x)}\right)&=\sharp\{\mathcal{L}^{\pr}\in\textup{Vert}^{1}(\mathbb{V}):\mathcal{L}\subset\mathcal{L}^{\pr},\,\,x\in\mathcal{L}^{\pr\vee}\}\\&=\begin{cases}
            q+1, &\textup{if $x\in\mathcal{L}\backslash\{0\}$;}\\
            1, &\textup{if $x\in\mathcal{L}^{\vee}\backslash\mathcal{L}$ and $\nu_\varpi(q(x))\geq0$;}\\
            2, &\textup{if $x\in\mathcal{L}^{\vee}\backslash\mathcal{L}$ and $\nu_\varpi(q(x))=-1$.}
        \end{cases}
    \end{align*}
\end{lemma}
\begin{proof}
    The argument is exactly the same as Lemma \ref{local-modularity-z}.
\end{proof}
Combining Lemma \ref{exc-local}, Lemma \ref{y-bar=z} and Theorem \ref{explicit value}, we get
\begin{corollary}
    Let $\mathcal{L}\in\textup{Vert}^{3}(\mathbb{V})$, then
    \begin{equation*}
        \textup{Int}^{\mathcal{Y}}_{\mathcal{L}}=2\cdot\boldsymbol{1}_{\mathcal{L}^{\vee}}.
    \end{equation*}
\end{corollary}

\bibliographystyle{alpha}
\bibliography{reference}

\end{document}